\documentclass[12pt]{amsart}

\usepackage{fancyhdr}
\fancyheadoffset[loh,reh]{0mm}
\usepackage[T1]{fontenc}
\usepackage[latin9]{inputenc}
\usepackage{amsmath}
\usepackage{enumitem}
\usepackage{subcaption}
\usepackage{amssymb}
\usepackage{yhmath}
\usepackage{stmaryrd}
\usepackage{esint,color,xcolor}
\usepackage{babel}
\usepackage{graphicx} 
\usepackage{tikz}
\usepackage[margin=30mm]{geometry}
\usepackage{eucal,mathrsfs,dsfont}
\usepackage{url,verbatim,hyperref}

\newtheorem{thm}{Theorem}[section]

\newtheorem{lemma}[thm]{Lemma}

\newtheorem{prop}[thm]{Proposition}

\newtheorem{theorem}[thm]{Theorem}

\theoremstyle{definition}

\newtheorem{remark}[thm]{Remark}

\numberwithin{equation}{section}

\newcommand \Om{\Omega}
\newcommand{\eps}{\varepsilon}

\newcommand{\prt}{\partial}
\newcommand{\normal}{\mathbf{n}}
\newcommand{\wt}{\widetilde}

\newcommand{\E}{{\mathbb{E}}}

\newcommand{\calB}{\mathcal{B}}

\newcommand{\calT}{\mathcal{T}}

\newcommand{\C}{\mathbb{C}}
\newcommand{\calA}{\mathcal{A}}
\newcommand{\calS}{\mathcal{S}}
\renewcommand{\P}{{\mathbb{P}}}
\newcommand{\R}{{\mathbb{R}}}
\newcommand{\N}{{\mathbb{N}}}

\DeclareMathOperator{\mflim}{mf-lim}
\usepackage{tikz, hyperref}
\usetikzlibrary{calc}
\usepackage{pgfplots}

\begin{document}

\newcommand{\phanuel}[1]{{\color{blue} Phanuel says: #1}}
\newcommand{\chris}[1]{{\color{teal} Chris says: #1}}
\newcommand{\ilias}[1]{{\color{red} Ilias says: #1}}
\newcommand{\CYK}[1]{{\color{teal} CYK says: #1}}

\title[Torsion problem]{Geometric properties of optimizers for \\ the maximum gradient of the torsion function}
\author{Krzysztof Burdzy}
\address{Department of Mathematics, Box 354350, University of Washington, Seattle, WA 98195, U.S.A.}
\email{burdzy@uw.edu}

\author{Ilias Ftouhi}
\address{Laboratoire MIPA, N\^imes University, Site des Carmes, Place Gabriel P\'eri, 30000 N\^imes, France}
\email{ilias.ftouhi@unimes.fr}



\author{Phanuel Mariano}
\address{Department of Mathematics, Union College, Schenectady, NY 12308, U.S.A.}
\email{marianop@union.edu}


\begin{abstract}
Consider
$J(\Om):= \|\nabla u_\Om\|_\infty/\sqrt{|\Om|} $ and $J_P(\Om):= \|\nabla u_\Om\|_\infty/P(\Om) $,
where $\Om$ is a planar convex domain, $u_\Om$ is the torsion function, $P(\Omega)$ is the perimeter of $\Omega$ and $|\Om|$ its area. We prove that there exist planar convex domains that maximize the functionals $J$ and $J_P$, and any maximizer has a $C^1$ boundary that contains a line segment on which $|\nabla u_\Om|$ attains its maximum.   
\end{abstract}

\maketitle

\section{Introduction and main Results}
\label{Defs}
Let $\Omega\subset\mathbb{R}^{2}$ be a bounded domain. The torsion function $u_\Om$ is defined by
\begin{align}\label{eq:udef}
\begin{cases}
-\Delta u_\Om=1 & \text
{ in } \Omega,\\
\ \ \ \ \ u_\Om=0 & \text{ on } \partial\Omega.
\end{cases}
\end{align}
The study of this function originated with Saint-Venant in 1856 (\cite{de1856memoire}). 
Shape optimization problems for the torsion function have been studied for various functionals in the literature. Isoperimetric problems for $\left\Vert u_{\Omega}\right\Vert _{\infty}$ have been known as early as the work by Talenti \cite{Talenti-1976} for general domains, while results for convex domains have been studied as well (\cite{Makar-Limanov-1971,Payne-Philippin-1983}). We refer to \cite{keady} for a nice survey of qualitative properties of the solutions of the torsion problem. Spectral functionals involving the torsion function have been studied in works such as \cite{Banuelos-Carrol1994,Vandenberg-Carroll-2009,Payne-1981,Vogt-2019a}. There have also been many results regarding functionals for $\left\Vert u_{\Omega}\right\Vert _{1}$ known as the torsional rigidity, such as in \cite{Banuelos-Mariano-2024,Berg-Buttazzo-Pratelli-2021,zbMATH06464861,zbMATH07716571,ftouhi-2020,Henrot-Lucardesi-Philippin-2018,henrot_tor,zbMATH06708424} to name a few. Recently, there has been interest in studying the maximum of the gradient of the torsion function (see \cite{AIM2019,hoskins2021towards,huang,li-etall-torsion2025,li2024location}).  To save space, we refer the reader to two recent papers \cite{hoskins2021towards,huang} for more information on the long history of the problem and related results.

We will study the maximal rate of change of stress, namely $\|\nabla u_\Omega\|_\infty$.
To make the problem interesting, it is natural to restrict attention to convex domains $\Om$ and to scale them as follows. Let $|\Om|$ be the area of $\Om$ and let $P(\Om)$ be its perimeter. We consider 
\begin{align}\label{defJ}
    J(\Omega)=\frac{\|\nabla u_\Om\|_{\infty}}{|\Omega|^{\frac{1}{2}}},\qquad 
    J_P(\Omega)=\frac{\|\nabla u_\Om\|_{\infty}}{P(\Om)}.
\end{align}
Both functionals 
$J$ and $J_P$ are bounded from above; see, for example, \cite{huang}.

We will study the existence and properties of maximizers of $J$ over the class of convex bodies, i.e., convex domains  $\Omega^*\subset \R^2$ such that for every  convex $\Omega$, one has 
\begin{align}\label{defOptimizer}
    J(\Omega)\leq J(\Omega^*).
\end{align}
Similarly, we will provide results about maximizers of $J_P$ over the class of convex bodies, i.e., convex domains  $\Omega'\subset \R^2$ such that for every  convex $\Omega$, one has 
\begin{align}\label{defPOptimizer}
    J_P(\Omega)\leq J_P(\Omega').
\end{align}

The present paper has been inspired by an article by Hoskins and Steinerberger \cite{hoskins2021towards} and personal communications with the second author of that article. Among other results,  \cite{hoskins2021towards} contains a convincing numerical approximation to the maximizer of $J(\Om)$ over the family of planar convex domains, showing clearly a line segment on the boundary of the maximizer. 
Our own numerical approximations of the optimal shapes are given in Table \ref{tab:max_area_peri}. The colors represent 
$|\nabla u_{\Om^{*}}|/ \sqrt{|\Omega^{*}|}$ and $|\nabla u_{\Om'}|/ P(\Om')$, respectively. The pictures show that both $\Omega^{*}$ and $\Om'$ have $C^1$ boundaries containing line segments, in agreement with our Theorems \ref{main:thm} and \ref{th:perimeter}. The maximum values of $|\nabla u_{\Omega^{*}}|$ and $|\nabla u_{\Omega'}|$ occur in the middle of the line segments on the boundaries. More details on the numerical simulations and other results will appear in the work in progress \cite{progress}.


\begin{table}[h!]
\centering
\begin{tabular}{|c|c|}
\hline
\includegraphics[width=0.46\linewidth]{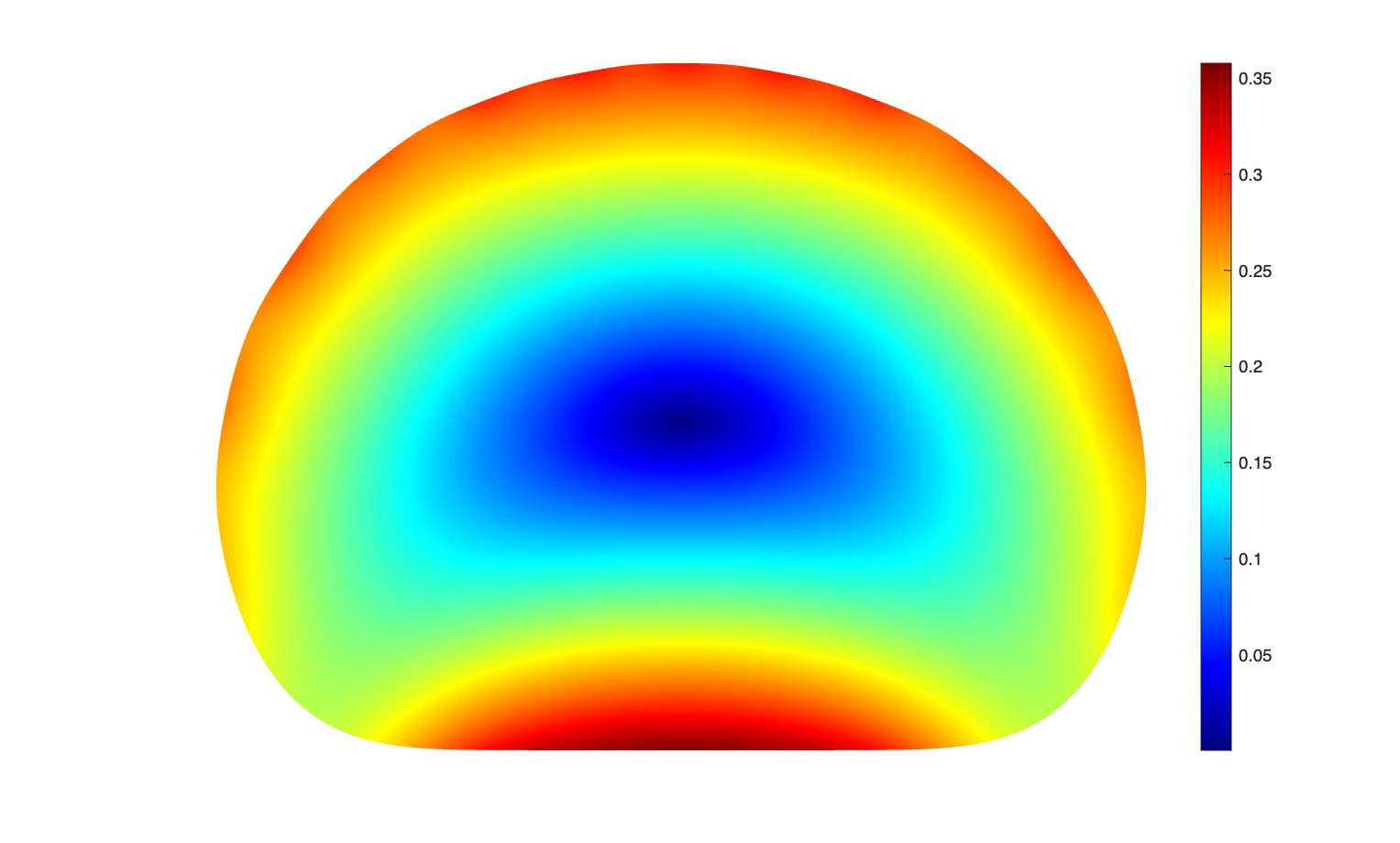} &
\includegraphics[width=0.48\linewidth]{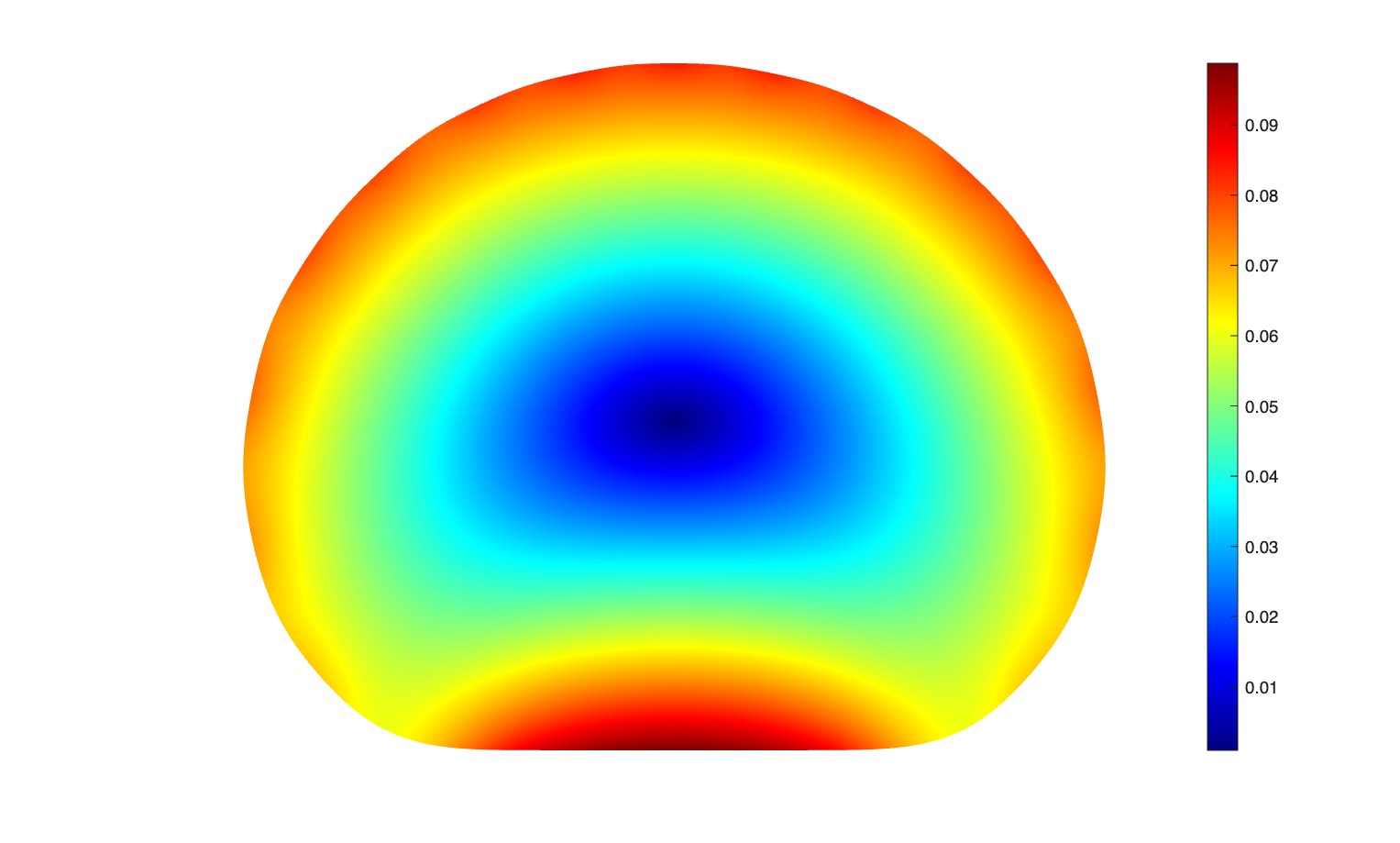} \\ 
\hline
 $J(\Om^*)\approx 0.3577...$ & $J_P(\Om')\approx 0.0988...$  \\
 \hline 
\end{tabular}
\caption{Maximizers of $J$ (on the left) and $J_P$ (on the right).}
\label{tab:max_area_peri}
\end{table}

We summarize our main results in the following theorems. 

\begin{theorem}\label{main:thm}
    \textit{(i)} There exists a convex maximizer satisfying \eqref{defOptimizer}.
    
    \textit{(ii)} If  $\Om$ is a maximizer,  then its boundary $\partial \Om$ is $C^1$, i.e., it has no corners.
    
    \textit{(iii)} If  $\Om$ is a maximizer,  then its boundary $\partial \Om$ contains a line segment.
\end{theorem}

\begin{theorem}\label{th:perimeter}
    \textit{(i)} There exists a convex maximizer satisfying \eqref{defPOptimizer}.
    
    \textit{(ii)} If  $\Om$ is a maximizer,  then its boundary $\partial \Om$ is $C^1$, i.e., it has no corners.
    
    \textit{(iii)} If  $\Om$ is a maximizer,  then its boundary $\partial \Om$ contains a line segment.
\end{theorem}

The above results improve (some of) recent theorems proved in \cite{huang}, where it is shown that
if $\Om$ is a convex maximizer, then either the boundary of $\Om$ is not $C^{2+\eps}$ for some $\eps>0$ or
the points of zero curvature on $\prt\Om$ accumulate around the point at which
the gradient is maximal. The existence of a convex optimizer is proved neither in \cite{huang}
nor in any earlier paper, to our best knowledge.

On the technical side, we will prove a new version of the boundary Harnack principle
in Proposition \ref{Lem:hvsu}. Roughly speaking, the proposition says that, in a Lipschitz domain
with the Lipschitz constant less than 1, the torsion function is comparable to a harmonic function
with zero boundary values in a neighborhood of the boundary.

We organize the paper as follows. In Section \ref{sec:prelim}, we 
present notation, definitions, and review some known results.
Section \ref{sec:prelimres} contains several versions
of the boundary Harnack principle for harmonic functions.
A new boundary Harnack principle for the torsion function
is proved in Section \ref{sec:BHPtor}. Section \ref{sec:boundtor} is devoted
to the boundary behavior of the torsion function.
The three parts of Theorem \ref{main:thm} are proved in Sections
\ref{sec:existopt}, \ref{sec:nocorners} and \ref{sec:line}.
The proof of Theorem \ref{th:perimeter}, presented in Section \ref{sec:perim}, is short because it is
essentially the same as that of Theorem \ref{main:thm}.

\section{Preliminaries}\label{sec:prelim}

We start with some notation, conventions and observations.

We will use both complex and vector notation, identifying $\R^2$ and $\C$.
The standard basis in $\R^2$ will be denoted by $(\mathbf{e}_1,\mathbf{e}_2)$.
The area of a set $A\subset \R^2$ will be denoted by $|A|$,  its perimeter by $P(A)$, its diameter by $d(A)$, and its inradius (i.e., the radius of the largest disc contained in $A$) by $r(A)$.
 
By abuse of notation, we will also write $|x|$ to denote the norm of $x\in \R^2$.
We will denote balls in different norms as
\begin{align*}
    \calB(x,r)&=\{y:|x-y|<r\},\\
\calS(r)&= \{(x_1,x_2): \max(|x_1|,|x_2|)< r\}.
\end{align*}

Let $G(x,y)$ denote the Green's function in $\Om$ with Dirichlet boundary values. Then
the torsion function $u_\Om(x)$ defined in \eqref{eq:udef} can be represented as
\begin{align}\label{uGreen}
    u_\Om(x) = \int_\Om G(x,y) dy.
\end{align}
Actually, the equality in \eqref{uGreen} is ``up to a constant.'' We normalize the Green's function so that \eqref{uGreen} holds as stated.

A function $f:\R\to\R$ is called Lipschitz  with the Lipschitz constant $\lambda<\infty$ if
$|f(x)-f(y)|\leq \lambda |x-y|$ for all $x$ and $y$.
A domain $\Om$ is called Lipschitz with the Lipschitz constant $\lambda$ if for every boundary point, there is a neighborhood
and an orthonormal coordinate system such that $\prt \Om$ is the graph of a Lipschitz function with constant $\lambda$
 in this neighborhood. 

 Let $W_t$ denote a planar Brownian motion and let $\P_x$ and $\E_x$ denote its distribution
and the associated expectation corresponding to the starting point $W_0=x$. For a set $A\subset \R^2$, let $\tau(A) := \inf\{t\geq 0: W_t\in A\}$
and $\sigma(A):= \inf\{t\geq 0: W_t\notin A\}$.
If $A\subset \Om$ and $h$ is the positive harmonic function  
in $\Om\setminus A$ with zero boundary values on $\prt\Om$ and boundary values equal to 1
on $\prt A$ then $h(x) = \P_x(\tau(A) < \sigma(\Om))$ for $x\in \Om\setminus A$.

The torsion function has the following probabilistic representation, $u_\Omega(x)=\E_x(\sigma(\Omega) ) $ for any open set $\Omega$ and $x\in \Omega$. 
We note that Brownian motion's infinitesimal generator is one-half of the usual Laplacian, so the last formula holds
not for the torsion function defined in \eqref{eq:udef} but for its
constant multiple. This is irrelevant to our arguments so we will ignore
the factor $1/2$ in our formulas for typographical reasons.

The following definitions are adapted from \cite[p. 9]{BurdzyBook} and \cite{Doob}.  
Suppose that $\Om$ is a planar domain, $z\in\prt\Om$ and $y\in\Om$. A set $A\subset \Om$ is called minimal thin at $z$ if $\tau(A)>0$, a.s., where $\tau(A)$ is defined relative to Brownian motion starting from $z$ and conditioned to go to $y$ before hitting $\prt \Om$ (except at time 0). Note that the notion of minimal tinniness does not depend on  the choice of $y\in \Omega$. A set $A\subset \Om$ is called a minimal fine neighborhood of $z$ if $\Om\setminus A$ is minimal thin at $z$.

A function $f:\Om\to \R$ has a minimal fine limit $b$ at $z$ if $\lim_{t\downarrow 0} f(W_t) = b$,  a.s., where $W$ is Brownian motion starting from $z$ and conditioned to go to $y\in \Om$ before hitting $\prt \Om$ (except at time 0). Here also, we note that the notion of minimal fine limit is independent of the choice of the interior point $y\in \Omega$.

For future reference, we will now state some known results.

\begin{remark}\label{oldres}
\textit{(i)} A simple calculation (see \cite[p. 85]{Sperb-1981}) leads to 
\begin{align}\label{subhar}
\Delta|\nabla u_\Om|^{2} & =2\left((\prt^2_{xx} u_\Omega )^{2}+(\prt^2_{yy} u_\Omega )^{2} +(\prt^2_{xy} u_\Omega )^{2}+(\prt^2_{yx} u_\Omega )^{2}\right)\\
&\ge 2\left(\frac{1}{2}(\prt^2_{xx} u_\Omega +\prt^2_{yy} u_\Omega)^2 + (\prt^2_{xy} u_\Omega )^{2}+(\prt^2_{yx} u_\Omega )^{2} \right)\notag \\ 
&= 1 + 2\left( (\prt^2_{xy} u_\Omega )^{2}+(\prt^2_{yx} u_\Omega )^{2} \right)> 0.\notag
\end{align}

In other words, the function $|\nabla u_\Om|^{2}$ is subharmonic.
By the maximum principle, $|\nabla u_\Om|$ cannot reach its maximum
in the interior. 

\textit{(ii)} By the result of \cite{Makar} (see translation in \cite{Makar-Limanov-1971}), the superlevel sets of the function $u_\Om$ are convex. 

\textit{(iii)} The Blaschke selection theorem states that, given a sequence 
$K_{n}$ of convex bodies contained in a bounded set, there exists a subsequence  $(K_{n_{m}})$ and a convex body $K$ such that $(K_{n_{m}})$ converges to $K$ in the Hausdorff metric. See \cite[Sec. 6.3]{KellyWeiss} or \cite[Thm. 1.8.7]{schneider}.

\end{remark}

\section{Auxiliary results}\label{sec:prelimres}

\begin{lemma}\label{lem:BHPgeneral}
(i) There exists $c_1>0$ depending only on $\lambda<\infty$ such that if
 $f:\R\to\R$ is a Lipschitz function with the Lipschitz constant $\lambda$,
 $\Om=\{(x_1,x_2)\in \R^2: x_2 > f(x_1)\} $, $v\in \R$,
 and $z=(v,f(v))$,
then for all  functions $h_1$ and $h_2$ that are positive and harmonic in
    $\calB(z,2r)\cap \Omega$  and have zero boundary values on $\calB(z,2r)\cap \prt \Om$, we have for all $x,y\in \calB(z,r)\cap \Om$,   
    \begin{align*}
        \frac{h_1(x)}{h_1(y)}\geq c_1\frac{h_2(x)}{h_2(y)}.
    \end{align*}

(ii) 
There exists $c_2>0$ depending only on $\lambda<\infty$ and $m<\infty$ 
with the following property.
Suppose that 
$\Om$ is a Lipschitz domain with the Lipschitz constant $\lambda$ and
$A\subset\Om$. 
Suppose that there exist  points $x_1, x_2, \dots, x_m$ in the closure of $A$ 
and $r_j >0$, $j=1,2,\dots,m$,
such that $ \bigcup_{1\leq j\leq m} \calB(x_j,r_j) $ is connected,
$A\subset \bigcup_{1\leq j\leq m} \calB(x_j,r_j) $,
and for every $j=1,\dots,m$, either
\begin{enumerate}[label=(\alph*)]
    \item $\calB(x_j,2r_j)\subset \Om$, or
    \item $x_j\in \prt A \cap \prt \Om$ and $\prt \Om\cap \calB(x_j,2r_j)$
is  the graph of
a Lipschitz function with constant $\lambda$ in some orthonormal coordinate system.
\end{enumerate}

Suppose that $h_1$ and $h_2$ are positive and harmonic functions in
    $\Om\cap \bigcup_{1\leq j\leq m} \calB(x_j,2r_j)$ and have zero boundary values on $\bigcup_{1\leq j\leq m} \calB(x_j,2r_j) \cap \prt \Om$. Then, for $x,y\in A$,
    \begin{align*}
        \frac{h_1(x)}{h_1(y)}\geq c_2\frac{h_2(x)}{h_2(y)}.
    \end{align*}
\end{lemma}

\begin{proof}
    The lemma follows directly from \cite[Thm. 1.1]{chenPT}. 
    The setup in that paper is very general but it is straightforward
    to check that it applies in the context of part \textit{(i)}.
    
The usual Harnack inequality for harmonic functions inside balls
and the standard Harnack chain argument (see \cite[Def. 2.10]{chenPT})
 imply \textit{(ii)}.
\end{proof}

\begin{remark}
    Lemma \ref{lem:BHPgeneral} is a version of the well-known boundary Harnack principle (BHP),
    and merely a corollary of \cite[Thm. 1.1]{chenPT}. 
    The new and crucial features of the BHP proved in \cite[Thm. 1.1]{chenPT} is that
    it applies to very general operators and
    the constants $c_1$ and $c_2$ depend only on $\lambda$ and $k$. In other words, it is a ``uniform'' version of BHP.
\end{remark}

The following lemma is a highly specialized corollary of Lemma \ref{lem:BHPgeneral}
needed in the proof of Theorem \ref{main:thm}.

\begin{lemma}\label{lem:BHP}

Suppose that $\Om\subset\{(x_1,x_2): x_2>0\}$ is a convex domain, $0\in\prt \Om$
and $\Om$ does not have a corner at $0$.
For $b>0$, let
    \begin{align*}
\calS(r)&= \{(x_1,x_2): \max(|x_1|,|x_2|)\leq r\},\qquad r>0,\\
    N_1&= \calS(b/1000) \cap \Om,\\
    N_2 &= (\calS(6b)\setminus \calS(4b)) \cap \Om,\\
    R_1&= \prt \calS(5b) \cap \Om,\\
    R_2&= \prt\calS(7b) \cap\Om.
\end{align*}
See Fig. \ref{figBHP}.

\begin{figure} [h!]
\includegraphics[width=0.7\linewidth]{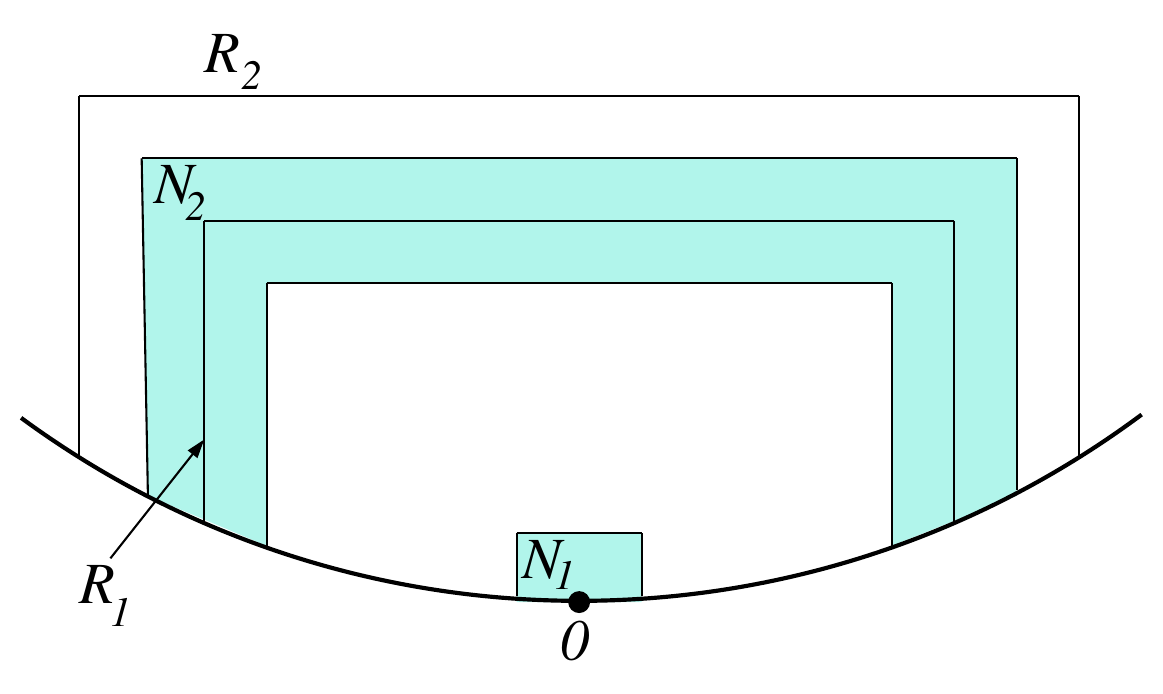}
\caption{
Subdomains of $\Om$. Drawing not to scale.}
\label{figBHP}
\end{figure}

(i)    Suppose that functions $h_1$ and $h_2$ are positive harmonic in
    $\calS(3b)\cap \Om$ and have zero boundary values on $\prt\Om$.
    There exist $\eps_1,c_1>0$ such that if $b\in(0,\eps_1)$ and
    $x,y\in N_1$ then
    \begin{align*}
        \frac{h_1(x)}{h_1(y)}\geq c_1\frac{h_2(x)}{h_2(y)}.
    \end{align*}
    
(ii)    Suppose that functions $h_1$ and $h_2$ are positive harmonic in
    $N_2$ and have zero boundary values on $\prt\Om$.
    There exist $\eps_1,c_1>0$ such that if $b\in(0,\eps_1)$ and
    $x,y\in R_1$ then
    \begin{align*}
        \frac{h_1(x)}{h_1(y)}\geq c_1\frac{h_2(x)}{h_2(y)}.
    \end{align*}
    
(iii)  Suppose that functions $h_1$ and $h_2$ are positive harmonic in
    $(\calS(8b)\setminus \calS(6.5 b))\cap \Om$ and have zero boundary values on $\prt\Om$.
    There exist $\eps_1,c_1>0$ such that if $b\in(0,\eps_1)$ and
    $x,y\in R_1$ then
    \begin{align*}
        \frac{h_1(x)}{h_1(y)}\geq c_1\frac{h_2(x)}{h_2(y)}.
    \end{align*}

\end{lemma}

\begin{proof}
All parts of the corollary are straightforward applications of Lemma \ref{lem:BHPgeneral} \textit{(ii)}.
\end{proof}

\section{Boundary Harnack principle for torsion function}\label{sec:BHPtor}
 
\begin{prop}\label{Lem:hvsu}

Suppose that $\Om\subset\R^2$ is an open set and $A\subset \Om$ is a disc with a positive radius such that the distance from $A$ to $\prt \Om$ is greater than 0. Recall $u_\Om$ defined in \eqref{eq:udef}. 
Let $h:x\longmapsto \P_x(\tau(A)< \sigma(\Om))$. The function $h$ is the unique positive harmonic function in $\Om\setminus A$ with boundary values 1 on $\prt A$ and 0 on $\prt \Om$. We have:
\begin{enumerate}[label=(\roman*)]
    \item For some $c_0>0$ depending on $\Om$ and $A$, and all $x\in \Om\setminus A$,
\begin{align}\label{eq:hlessu}
 u_\Om(x) \geq c_0 h(x).
\end{align}
\item If $\Om$ is a bounded Lipschitz domain with the Lipschitz constant  $\lambda<1$, then, for
some $c_1>0$ depending on $\Om$ and $A$, and all $x\in \Om\setminus A$,
\begin{align}\label{eq:ulessh}
 h(x)\geq c_1 u_\Om(x) .
\end{align}
\end{enumerate}

\end{prop}

\begin{proof}
\textit{(i)} It is easy to see that for some $c_0>0$ and all $y\in A$, $\E_y(\sigma(\Om)) \geq c_0$, so for $x\in \Om\setminus A$,
\begin{align*}
    u_\Om(x) &= \E_x(\sigma(\Om)) \geq \int _{\prt A}  \E_y(\sigma(\Om)) \P_x(W_{\tau(A)}\in dy, \tau(A) < \sigma(\Om))\\
    &\geq \int _{\prt A}  c_0 \P_x(W_{\tau(A)}\in dy, \tau(A) < \sigma(\Om))
    = c_0 \P_x (\tau(A) < \sigma(\Om)) =c_0 h(x).
\end{align*}
This proves \eqref{eq:hlessu}.

\textit{(ii)}
Let $G(x,y)$ be the Green's function in $\Om$ with Dirichlet boundary values.
Fix $z_1\in \Om$ and let $B =\{x\in \Om: G(z_1,x) \geq 1\}$.

Suppose that $\sup_{x\in \Om\setminus A} u_\Om(x) /h(x) =+\infty$.
We will show that this assumption leads to a contradiction.

Let $(y_k)_{k\in \N}$ be a sequence of points in $\Om$ such that $u_\Om(y_k)/h(y_k) > k$ for all $k$. The distance from $y_k$ to $\prt \Om$ must decrease to $0$ because
$u_\Om$ and $h$
are continuous and strictly positive functions in $ \Om\setminus A$
and $h$ has boundary values $1$ on $\prt A$. 

Let $v_k$ be one of the points in $\prt\Om$ with the smallest distance to $y_k$.
The Martin and Euclidean boundaries can be identified in Lipschitz domains
so let $K_k(x)$ be the Martin kernel in $\Om$ associated with $v_k$,
i.e., a positive harmonic function which vanishes continuously at every point of $\prt \Om$
except at $v_k$.
We normalize $K_k$ so that $K_k(z_1)=1$.

If $r_k = |y_k-v_k|$, Lemma \ref{lem:BHPgeneral} \textit{(ii)}
applied 
to the positive harmonic functions $z\longmapsto G(y_k, z)$ and $z\longmapsto K_k(z)$
in $(\calB(v_k, 3r_k)\setminus\calB(v_k, 2r_k))\cap \Om$
implies that there is $c_3>0$ such that for $z,z'\in (\calB(v_k, 3r_k)\setminus\calB(v_k, 2r_k))\cap \Om$,
\begin{align}\label{eq:GKold}
 \frac{1}{c_3} \frac{ G(y_k, z')}{K_k(z')} \leq    \frac{ G(y_k, z)}{K_k(z)} \leq c_3  \frac{ G(y_k, z')}{K_k(z')}.
\end{align}

The constant $c_3$ depends only on $\lambda$ because, by scaling and the Lipschitz property of $\Om$, one can choose $m$ in Lemma \ref{lem:BHPgeneral} \textit{(ii)} independent of $k$ in the present argument.

Fix any  $z'\in (\calB(v_k, 3r_k)\setminus\calB(v_k, 2r_k))\cap \Om$ and let $a_k =  G(y_k, z')/K_k(z')$.
Then \eqref{eq:GKold} yields
\begin{align}\label{oc19.1}
    \frac{1}{c_3} a_k K_k(z) \leq   G(y_k, z) \leq c_3  a_k K_k(z),
\end{align}
for all $z\in (\calB(v_k, 3r_k)\setminus\calB(v_k, 2r_k))\cap \Om$.

The functions $z\longmapsto G(y_k, z)$ and $z\longmapsto K_k(z)$ are positive and harmonic in
$\Om \setminus \calB(v_k, 2r_k)$. Their boundary values satisfy \eqref{oc19.1} on
$\prt \calB(v_k, 2r_k) \cap \Om$ and they are 0 on the remaining part of the boundary of $\Om\setminus\calB(v_k, 2r_k)$. 
Hence, \eqref{oc19.1} can be extended to all $z\in \Om\setminus\calB(v_k, 2r_k)$, in particular, it can be applied to $z_1$. Recall that $K_k(z_1)=1$. We have
\begin{align*}
  \frac{1}{c_3} a_k=  \frac{1}{c_3} a_k K_k(z_1) \leq   G(y_k, z_1) \leq c_3  a_k K_k(z_1) = c_3 a_k,
\end{align*}
so this, combined with \eqref{oc19.1} gives
\begin{align}\label{eq:GK}
    \frac{ G(y_k, z)}{G(y_k,z_1)} \leq \frac{c_3}{a_k}c_3 a_k K_k(z) = c_3^2 K_k(z),
\end{align}
for
$z\in \Om\setminus\calB(v_k, 2r_k)$.

The functions $x\longmapsto h(x)$ and $x\longmapsto G(z_1,x)$ are positive harmonic in $\Om \setminus (A\cup B)$, have zero boundary values
on $\prt \Om$ 
and have comparable boundary values on $\prt (A \cup  B)$  
so  for some $c_4>0$
and all $k$, $h(y_k)> c_4 G(z_1,y_k)$, and, therefore, $u_\Omega(y_k)/G(z_1,y_k) >c_4 k$.

In view of \eqref{uGreen} and \eqref{eq:GK}, 
\begin{align}\label{eq:comp11}
    c_4 k &<\frac{u_\Omega(y_k)}{G(z_1,y_k)} 
    = \frac{\int_\Om G(y_k, z)dz}{G(z_1,y_k)}
    =\int_\Om \frac{ G(y_k, z)}{G(y_k,z_1)}dz\\
   & = \int_{\Om\setminus\calB(v_k, 2r_k)} \frac{ G(y_k, z)}{G(y_k,z_1)}dz + \int_{\Om\cap\calB(v_k, 2r_k)} \frac{ G(y_k, z)}{G(y_k,z_1)}dz
   \notag\\
   & \leq c_3^2 \int_{\Om\setminus\calB(v_k, 2r_k)} K_k(z)dz + \int_{\Om\cap\calB(v_k, 2r_k)} \frac{ G(y_k, z)}{G(y_k,z_1)}dz.\notag
\end{align}

Let $\rho$ be twice the diameter of $\Om$ and let $G_k(x,y)$ be the Green's function in $\calB(y_k,\rho)$ with Dirichlet boundary values.
Note that $\calB(v_k, 2r_k)\subset \calB(y_k, 3r_k)$ and $3r_k < \rho$ for large $k$.
We have for some $c_5,c_6<\infty$ and sufficiently large $k$,
\begin{align}\label{eq:upper}
    \int_{\Om\cap\calB(v_k, 2r_k)} G(y_k, z)dz
    &\leq \int _{\calB(y_k, 3r_k)} G_k(y_k,z)dz
    \leq c_5\int_0^{3r_k} r |\log r| dr 
    \leq c_6 r_k^2 | \log r_k|.
\end{align}

By \cite[Prop. 2]{MS}, if $\lambda < 1$  then for some $c_7>0$ and $\alpha<2$,
\begin{align*}
    G(y_k,z_1) \geq c_7 r_k^\alpha.
\end{align*}

This and \eqref{eq:upper} imply that 
\begin{align*}
     0\leq \int_{\Om\cap\calB(v_k, 2r_k)} \frac{ G(y_k, z)}{G(y_k,z_1)}dz
     \leq \frac{c_6 r_k^2 | \log r_k|}{c_7 r_k^\alpha} = (c_6/c_7) r_k^{2-\alpha}|\log r_k|.
\end{align*}

The quantity on the right-hand side converges to 0 as $k\to+\infty$ because $r_k\underset{k\to +\infty}{\longrightarrow} 0$.
This observation and \eqref{eq:comp11} imply that for sufficiently large $k$,
\begin{align*}
\int_{\Om} K_k(z)dz\geq 
    \int_{\Om\setminus\calB(v_k, 2r_k)} K_k(z)dz \geq \frac{c_4 k}{2c_3^2}.
\end{align*}

Since the kernels $K_k$ were normalized so that $K_k(z_1)=1$, the
function $K_*:z\longmapsto\sum_{j=1}^\infty 2^{-j} K_{2^{2j}}(z)$
is finite, positive and harmonic. We have
\begin{align*}
\int_{\Om} K_*(z)dz\geq \sum_{j=1}^\infty 2^{-j} \frac{c_4 2^{2j}}{2c_3^2}=\infty.
\end{align*}

This contradicts  \cite[Thm. 2]{aikawa} which asserts that all positive harmonic functions are integrable
in $\Om$ if  $\lambda<1$.
\end{proof}

\section{Boundary behavior of the torsion function}\label{sec:boundtor}

For the definition of minimal fine topology and the corresponding convergence, see Section \ref{sec:prelim}. 

\begin{lemma}\label{lem:supbound}
Suppose $\Om\subset \R^2$ is a bounded open convex domain.

\begin{enumerate}[label=(\alph*)]
    \item There exists a set $\prt_r \Om \subset\prt\Om$ with the full harmonic measure such that the following holds for $z\in \prt_r \Om$.\medskip

\begin{enumerate}[label=(\roman*)]
    \item The boundary of $\Om$ does not have a corner at $z$. We will denote the unit inner normal vector  $\normal(z)$.
    \item The  minimal fine limit $ \mflim_{x\to z, x\in\Om} |\nabla u_\Om(z)|$ exists. We will denote it $|\nabla u_\Omega(z)|_{\mathrm{mf}}$.
    \item We have $\lim_{x\to z, x\in\Om} \frac{\nabla u_\Omega(x)}{|\nabla u_\Omega(x)|} =\normal(z)$.
    \item $\lim_{\delta\to 0} \frac 1 \delta u_\Omega(z + \delta \normal(z))=|\nabla u_\Omega(z)|_{\mathrm{mf}}$.
\end{enumerate}\medskip
\item $\sup _{y\in \prt_r \Om} |\nabla u_\Omega(y)|_{\mathrm{mf}}/|\Om|^{1/2} =\sup_{x\in\Om} |\nabla u_\Omega(x)|/|\Om|^{1/2}=J(\Om) $.\medskip
\item There exists  $z\in \prt \Om$ such that for every $n>0$ there exists
a neighborhood $U_n\subset \prt \Om$ of $z$ such that the harmonic measure of 
 $\{y\in U_n \cap \prt_r\Om: |\nabla u_\Omega(y)|_{\mathrm{mf}}/|\Om|^{1/2}> J(\Om)-1/n\}$ is strictly positive.
\end{enumerate}

\end{lemma}

\begin{proof}
\textit{(a) 
(i)} A Brownian motion does not hit a point fixed in advance. There are at most countably many corners on the boundary of a convex domain. The probability of hitting a corner is zero. Hence, the harmonic measure is carried by boundary points without corners. 

\textit{(ii)} The function $x\longmapsto|\nabla u_\Omega(x)|^2$ is subharmonic in $\Om$ by \eqref{subhar}. It follows that if $W_t$ is Brownian motion then $|\nabla u_\Omega(W_t)|^2$ is a continuous submartingale on the interval $[0,\sigma(\Om))$. Since $x\longmapsto|\nabla u_\Omega(x)|^2$ is bounded in $\Omega$, the submartingale has a left limit at the time $\sigma(\Om)$, a.s. This implies that the  minimal fine limit $ \mflim_{x\to z, x\in\Om} |\nabla u_\Omega(z)|^2$ exists for a set $\prt_r\Om$ of $z$'s with the full harmonic measure. It follows that the  minimal fine limit $ \mflim_{x\to z, x\in\Om} |\nabla u_\Omega(z)|$ exists for $z\in\prt_r\Om$.

\textit{(iii)} By the result of \cite{Makar} (translation in \cite{Makar-Limanov-1971}), the superlevel sets of the function $u$ are convex. This easily implies that \textit{(iii)} holds at all non-corner points of $\prt \Om$.

\textit{(iv)} By \cite[Thm. 10.1]{BurdzyBook}, all excursion laws for Brownian motion in a planar convex domain are quasi-locally flat. By \cite[Rem. 10.3 \textit{(iii)}]{BurdzyBook}, these excursion laws are actually locally flat. By \cite[Def. 7.1]{BurdzyBook}, the local properties of these excursion laws are the same as those of excursions in a half-space. 

Suppose that $z\in\prt \Om$ and the excursion law at $z$ is locally flat. Let $L$ be the halfline with the endpoint $z$, stretching in the direction of $\normal(z)$. Suppose that $A$ is a given subset of $L$  and let $A_k=\{y\in A: 2^{-k-1}<|z-y|\leq 2^{-k} \}$ for $k\in \N$. Suppose that $\limsup _{k\to\infty}2^k|A_k| >0 $. Then, standard arguments similar to the Wiener test show that $A$ is not minimal thin at $0$.

For any $\eps>0$,
we let $A(\eps)=\{y\in L :\big||\nabla u_\Omega(y)|-|\nabla u_\Omega(z)|_{\mathrm{mf}}\big|>\eps \}$. By part \textit{(ii)}, the set $A(\eps)$ is minimal thin at $z$ for every $\eps>0$. Hence, if $A_k(\eps)=\{y\in A(\eps): 2^{-k-1}<|z-y|\leq 2^{-k} \}$, then, $\limsup _{k\to\infty}2^k|A_k(\eps)| =0 $.
This and the fact that $|\nabla u_\Omega|$ is bounded imply that 
\begin{align*}
    \lim_{\delta\to 0} \frac 1 \delta \int_{y\in L, |z-y|< \delta} |\nabla u_\Omega(y)| dy = |\nabla u_\Omega(z)|_{\mathrm{mf}}.
\end{align*}

In view of \textit{(iii)}, we conclude that \textit{(iv)} holds.

\textit{(b)}  Since $|\nabla u_\Omega(W_t)|^2$ is a continuous submartingale on  $[0,\sigma(\Om))$, by the optional stopping theorem,  for every $x\in \Om$,
\begin{align*}
   \E_x \left(\lim_{t\uparrow \sigma(\Om)}|\nabla u_\Omega(W_t)|^2 \right)\geq  |\nabla u_\Omega(x)|^2,
\end{align*}
and, therefore,
\begin{align*}
   \P_x \left(\lim_{t\uparrow \sigma(\Om)}|\nabla u_\Omega(W_t)| \geq  |\nabla u_\Omega(x)|\right)>0.
\end{align*}

This implies that the harmonic measure of $\{z\in \prt\Om:\ |\nabla u_\Omega(z)|_{\mathrm{mf}}\geq|\nabla u_\Omega(x)|\}$ is strictly positive.

It follows that $\sup _{y\in \prt_r \Om} |\nabla u_\Omega(y)|_{\mathrm{mf}} \geq \sup_{x\in\Om} |\nabla u_\Omega(x)| $. The opposite inequality follows from the continuity of $|\nabla u_\Omega(W_t)|$.

\textit{(c)} If the claim in part \textit{(c)} is false, then, for every $z\in \prt \Om$, there is a neighborhood $U(z)\subset \prt\Om$ of $z$ and $\eps>0 $, such that 
$|\{y\in U(z): |\nabla u_\Omega(y)|_{\mathrm{mf}}/|\Om|^{1/2}> J(\Om)-\eps\}|=0$. By compactness, we can cover $\prt\Om$ by a finite number of such neighborhoods. Therefore, for some $\eps_1>0$,
$|\{y\in \prt\Om: |\nabla u_\Omega(y)|_{\mathrm{mf}}/|\Om|^{1/2}> J(\Om)-\eps_1\}|=0$. This contradicts the observation made in part (b) of the proof.

\end{proof}

\begin{remark}\label{rem:supbound}
 Lemma \ref{lem:supbound} supplies a proof of an apparent gap in \cite{hoskins2021towards}. The authors of that paper proved an upper bound for $u_\Omega(z + \eps \normal(z))$ for small $\eps>0$ (see, for example, page 7834 of \cite{hoskins2021towards}). Our Lemma \ref{lem:supbound}, especially part (a) \textit{(iv)}, shows how to rigorously translate their bound into a statement about the gradient of the torsion function.

We note that even an important book \cite{Sperb-1981} takes a rather cavalier attitude towards the meaning of the gradient of the torsion function on the boundary.
On page 85 of the book, the author considers a domain $D$ that is connected, finite (this presumably means bounded) and planar. In principle, $D$ could be a fractal domain, such as a snowflake. Then, on the same page, we find the formula $\max_D |\nabla u_D|=\max_{\prt D} |\nabla u_D| $
without any discussion of what $|\nabla u_D(x)|$ means for a boundary point $x$.

\end{remark}

\section{Existence of an optimizer}\label{sec:existopt}

\begin{proof}[Proof of Theorem \ref{main:thm} \textit{(i)}]
\emph{Step 1.}
Let $J_{\max}$ be the supremum of $J(\Om)$ over all planar convex bodies $\Om$.

Suppose that $\Om_n$ are such that $|\Om_n|=1 $ and $J(\Om_n) > J_{\max} - 1/n$. 
Let $\Omega_n' = \{x\,:\, u_{\Omega_n}(x) > 1/n\}$.
Note that $u_{\Omega_n}-1/n$ is the unique solution of the torsion problem in the set $\Omega_n'$, i.e., $u_{\Omega'_n}=u_{\Omega_n}-1/n$. Therefore, we have
$$ J(\Omega_n') = \frac{ J(\Omega_n) -1/n}{\sqrt{|\Om_n'|}}> J_{\max} -2/n.$$

The domains $\Om'_n$ are convex by Remark \ref{oldres} \textit{(ii)}. 
Thus,
domains $\Om''_n :=\frac{1}{\sqrt{|\Om_n'|}}\Om'_n$ are convex, have areas equal to 1, have smooth (analytic) boundaries and $\lim_{n\to\infty} J(\Om''_n) =J_{\max}$. 
Because of the definition of $\Om''_n $, the gradient $\nabla u_{\Om''_n}$ is an analytic function on the closure of $\Omega''_n$.

To simplify notation, we can and will assume without loss of generality that $\Om_n$ is a sequence of convex sets with analytic boundaries such that the gradient $\nabla u_{\Om_n}$ is an analytic function on the closure of $\Omega_n$,
$|\Om_n|=1 $ and $J(\Om_n) > J_{\max} - 1/n$, for every $n$. We will also assume that for every $n$, $\Omega_n \subset \{(x,y)\,:\, y\ge 0\}$ and that $|\nabla u_{\Om_n}|$  attains its maximum at $0\in \partial \Omega_n$ (it may also attain the maximum at other points).
Therefore, 
\begin{align}\label{eq:bound}
   |\nabla u_{\Om_n}(0)| = \lim_{\delta\to 0} \frac 1 \delta u_{\Om_n}(0 + \delta \boldsymbol{e}_2) > J_{\max} - 1/n.
\end{align}

We will argue that $\Omega_n$'s have uniformly bounded diameters. Assume otherwise. Then, without loss of generality, we may assume that  $d(\Omega_n)\to +\infty$, and, therefore,
\begin{equation}\label{ineq:04_12}
    \|\nabla u_{\Omega_n}  \|_{\infty}\leq r(\Omega_n)\leq \frac{2|\Omega_n|}{P(\Omega_n)}\leq \frac{|\Omega_n|}{d(\Omega_n)}\underset{n\to +\infty}{\longrightarrow} 0,
\end{equation}
which contradicts the fact that $J(\Om_n) > J_{\max} - 1/n$, for every $n$. Here, the first inequality is a consequence of arguments from the book of Sperb \cite{Sperb-1981}, see the discussion in \cite[Section 2.1]{hoskins2021towards}. The second inequality is classical and can be found in the book of Bonnesen and Fenchel \cite{bonnesen}. In the last estimate, we just used the fact that the perimeter of a planar convex set is larger than twice its diameter. 

 Therefore, by Blaschke selection theorem (see Remark \ref{oldres} \textit{(iii)}), there exists a convex body $\Omega$ such that $\Om_n \underset{n\rightarrow +\infty}{\longrightarrow} \Om$ in the Hausdorff topology. We also note that $|\Omega|=1$ because of the continuity of area with respect to the Hausdorff distance for convex domains.

We would like to prove that the set $\Omega$ is a maximizer of the supremum of the gradient of the torsion function among planar convex sets with unit measure i.e., we want to show that 
$$J(\Omega)=\|\nabla u_\Omega\|_{\infty} = \lim_{n\rightarrow+\infty} |\nabla u_{\Omega_n}(0)| =\lim_{n\rightarrow+\infty} \|\nabla u_{\Omega_n}\|_{\infty} = J_{\max}. $$

\emph{Step 2.}
In this step, we prove that the limit set $\Om$ does not have a corner at 0. 

Let $C_\alpha=\{re^{i\theta}: r>0, 0<\theta< \alpha\}$.
For $\eta>0$ and $\alpha\in(0,\pi)$, let 
$B_{\alpha,\eta}$ be the disc inscribed in $C_\alpha$ such that  $(\eta,0)\in \prt B_{\alpha,\eta} \cap \prt C_\alpha $. Let
\begin{align*}
    \wt C_{\alpha,\eta} =\{re^{i\theta}: r\in(0,\eta), 0<\theta< \alpha\} \cup B_{\alpha,\eta}.
\end{align*}
In other words, $\wt C_{\alpha,\eta}$ is the convex hull of $B_{\alpha,\eta}$ and the origin $0$.

Suppose that the limit set $\Om$ has a corner at $0$. Without loss of generality, assume that
$\Om\subset C_\alpha$ for some $\alpha\in(\pi/2,\pi)$,
and the corners of $\Om$ and  $C_\alpha$ are at $0$.
Even if the angle at the corner of $\Om$ is smaller than $\pi/2$, we 
can take $\alpha\in(\pi/2,\pi)$.
We let $\eta>0 $ be twice the diameter of $\Om$. 

The function $h_1:(x_1,x_2)\longmapsto x_2$ is positive harmonic in the upper half-plane
with zero boundary values. Let $h_2:z =re^{i\theta}\in C_\alpha \longmapsto h_1(z^{\pi/\alpha}) = r^{\pi/\alpha} \sin\left(\frac{\pi \theta}{\alpha}\right)$.
The function $h_2$ is positive harmonic in $C_\alpha$ with zero boundary values.

The set $\wt C_{\alpha,\eta}$ is a Lipschitz domain with the Lipschitz constant $\lambda< 1$
because $\alpha\in(\pi/2,\pi)$. Let $\wt u_{\alpha,\eta}$ be the torsion function in $\wt C_{\alpha,\eta}$.
By Proposition \ref{Lem:hvsu} \textit{(ii)} and the boundary Harnack principle (Lemma \ref{lem:BHPgeneral}), there exists a constant $c_1>0$ depending only on $\eta$ and $\alpha$ such that for all $x\in \wt C_{\alpha,\eta}$ satisfying $|x| < \eta/2$, one has
\begin{align}\label{ag20.1}
 \wt u_{\alpha,\eta}(x) \leq c_1  h_2(x) \leq c_1 |x|^{\pi/\alpha} .
\end{align}

Since the sequence $(\Om_n)$ converges to $\Om$ in the Hausdorff topology, there exist vectors $v_n\in\R\times(0,\infty)$
such that $v_n\underset{n\to +\infty}{\longrightarrow} 0$ and  $\Om'_n := \{v_n+x: x\in \Omega_n\} \subset \wt C_{\alpha,\eta}$.
Since $|\nabla u_{\Om_n}|$ attains its maximum at $0$, the maximum
of $|\nabla u_{\Om'_n}|$ is attained at $v_n$ and has the same value. We recall that this value was assumed to be larger than $J_{\max} -1/n$ in step 1. 

Let $w_n $ be a point in the interior of $ \Om'_n$ such that 
$|w_n| \leq 2 |v_n|$ and
$|\nabla u_{\Om'_n}(w_n)|\geq J_{\max} -2/n$. Let $D_n := \{x\in \Om'_n: |x-w_n|\leq |v_n|\}$,  see Figures \ref{fig:shift_v_n} and \ref{fig:shift_v_n_zoom}.

We obtain from \eqref{ag20.1} that for large $n$ and $x\in\prt D_n$,
\begin{align}\label{ag20.5}
     \wt u_{\alpha,\eta}(x)  \leq c_1 |x|^{\pi/\alpha}\leq c_1 |3 v_n|^{\pi/\alpha}= c_2 |v_n|^{\pi/\alpha} .
\end{align}

\begin{figure}[h]
    \centering

\tikzset{every picture/.style={line width=0.75pt}} 

\begin{tikzpicture}[x=0.75pt,y=0.75pt,yscale=-.6,xscale=.6]

\draw [color={rgb, 255:red, 74; green, 144; blue, 226 }  ,draw opacity=1 ][line width=2.25]    (77.33,131) -- (291.33,311) ;
\draw [color={rgb, 255:red, 74; green, 144; blue, 226 }  ,draw opacity=1 ][line width=2.25]    (592.33,309) -- (291.33,311) ;
\draw  [color={rgb, 255:red, 208; green, 2; blue, 27 }  ,draw opacity=1 ][dash pattern={on 6.75pt off 4.5pt}][line width=2.25]  (199.33,105) .. controls (230.33,81) and (299.33,78) .. (375.33,98) .. controls (451.33,118) and (554.33,227) .. (555.33,268) .. controls (556.33,309) and (290.33,311) .. (291.33,311) .. controls (292.33,311) and (206.33,240) .. (182.33,216) .. controls (158.33,192) and (168.33,129) .. (199.33,105) -- cycle ;
\draw  [line width=2.25]  (277.33,81) .. controls (384.33,64) and (461.33,117) .. (474.33,129) .. controls (487.33,141) and (544.33,183) .. (564.33,241) .. controls (584.33,299) and (339.33,285) .. (289.33,285) .. controls (239.33,285) and (163.33,200) .. (167.33,176) .. controls (171.33,152) and (170.33,98) .. (277.33,81) -- cycle ;
\draw [color={rgb, 255:red, 208; green, 2; blue, 27 }  ,draw opacity=1 ]   (168.33,77) -- (190.88,98.29) ;
\draw [shift={(192.33,99.67)}, rotate = 223.36] [color={rgb, 255:red, 208; green, 2; blue, 27 }  ,draw opacity=1 ][line width=0.75]    (10.93,-3.29) .. controls (6.95,-1.4) and (3.31,-0.3) .. (0,0) .. controls (3.31,0.3) and (6.95,1.4) .. (10.93,3.29)   ;
\draw  [line width=5.25] [line join = round][line cap = round] (291.33,311.67) .. controls (291.33,311.67) and (291.33,311.67) .. (291.33,311.67) ;
\draw  [line width=5.25] [line join = round][line cap = round] (291.33,285.67) .. controls (291.33,285.67) and (291.33,285.67) .. (291.33,285.67) ;
\draw  [line width=5.25] [line join = round][line cap = round] (306.33,271) .. controls (306.33,271) and (306.33,271) .. (306.33,271) ;
\draw [color={rgb, 255:red, 74; green, 144; blue, 226 }  ,draw opacity=1 ]   (141.33,231) -- (157.88,215.37) ;
\draw [shift={(159.33,214)}, rotate = 136.64] [color={rgb, 255:red, 74; green, 144; blue, 226 }  ,draw opacity=1 ][line width=0.75]    (10.93,-3.29) .. controls (6.95,-1.4) and (3.31,-0.3) .. (0,0) .. controls (3.31,0.3) and (6.95,1.4) .. (10.93,3.29)   ;
\draw [color={rgb, 255:red, 189; green, 16; blue, 224 }  ,draw opacity=1 ][line width=0.75]    (291.33,311) -- (291.33,291.67) ;
\draw [shift={(291.33,289.67)}, rotate = 90] [color={rgb, 255:red, 189; green, 16; blue, 224 }  ,draw opacity=1 ][line width=0.75]    (10.93,-3.29) .. controls (6.95,-1.4) and (3.31,-0.3) .. (0,0) .. controls (3.31,0.3) and (6.95,1.4) .. (10.93,3.29)   ;
\draw    (480.33,95) -- (465.66,111.51) ;
\draw [shift={(464.33,113)}, rotate = 311.63] [color={rgb, 255:red, 0; green, 0; blue, 0 }  ][line width=0.75]    (10.93,-3.29) .. controls (6.95,-1.4) and (3.31,-0.3) .. (0,0) .. controls (3.31,0.3) and (6.95,1.4) .. (10.93,3.29)   ;

\draw (275,316.4) node [anchor=north west][inner sep=0.75pt]  [font=\large]  {$0$};
\draw (265.33,257) node [anchor=north west][inner sep=0.75pt]    {$v_{n}$};
\filldraw[black] (290,285) circle (3pt);
\filldraw[black] (310,272) circle (3pt);
\filldraw[black] (292,310) circle (3pt);
\draw (488,67) node [anchor=north west][inner sep=0.75pt]  [font=\large]  {$\Omega'_{n}$};
\draw (140,50) node [anchor=north west][inner sep=0.75pt]  [font=\large]  {$\textcolor[rgb]{0.82,0.01,0.11}{\Omega }$};
\draw (312,250) node [anchor=north west][inner sep=0.75pt]    {$w_{n}$};
\draw (115,231.4) node [anchor=north west][inner sep=0.75pt]  [font=\large,color={rgb, 255:red, 74; green, 144; blue, 226 }  ,opacity=1 ]  {$C_{\alpha }$};

\end{tikzpicture}

    \caption{The domain $\Omega$ (red dashed line) and the shifted domain $\Omega'_n$ (black).}
    \label{fig:shift_v_n}
\end{figure}
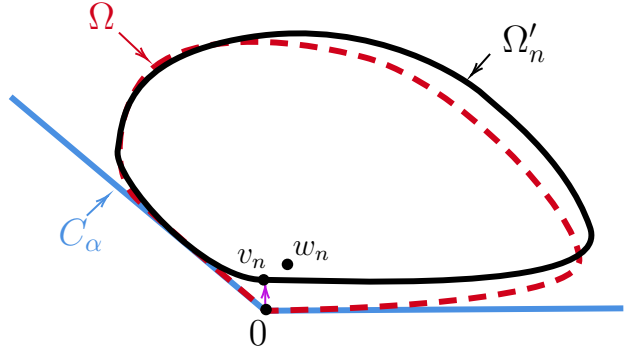

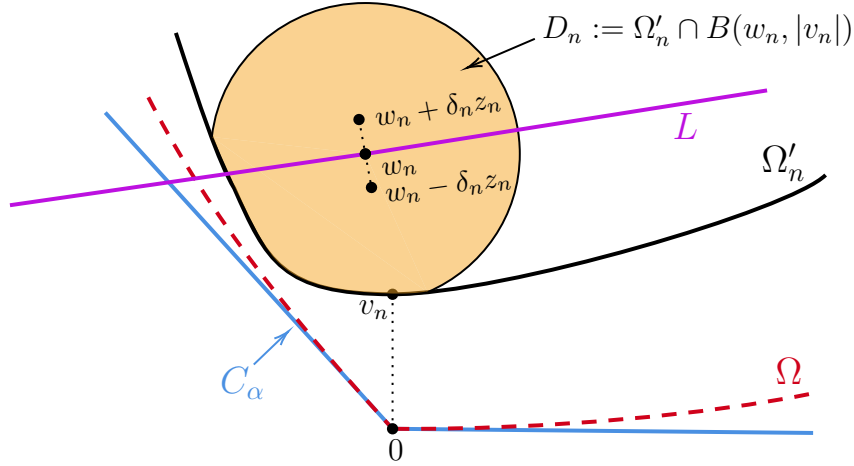
\begin{figure}[h]
    \centering

\tikzset{every picture/.style={line width=0.75pt}} 

\begin{tikzpicture}[x=0.75pt,y=0.75pt,yscale=-.7,xscale=.7]

\draw [color={rgb, 255:red, 74; green, 144; blue, 226 }  ,draw opacity=1 ][line width=1.5]    (74.33,90) -- (281.33,318) ;
\draw [color={rgb, 255:red, 74; green, 144; blue, 226 }  ,draw opacity=1 ][line width=1.5]    (281.33,318) -- (584.33,321) ;
\draw [color={rgb, 255:red, 208; green, 2; blue, 27 }  ,draw opacity=1 ][line width=1.5]  [dash pattern={on 5.63pt off 4.5pt}]  (105.33,79) .. controls (140.33,147) and (189.33,218) .. (281.33,318) ;
\draw [color={rgb, 255:red, 208; green, 2; blue, 27 }  ,draw opacity=1 ][line width=1.5]  [dash pattern={on 5.63pt off 4.5pt}]  (581.33,293) .. controls (512.33,307) and (393.33,319) .. (281.33,318) ;
\draw [line width=1.5]    (125.21,32.55) .. controls (142.21,84.55) and (156.21,126.55) .. (167.21,144.55) .. controls (190.21,196.55) and (201.33,221) .. (281.33,221) .. controls (361.33,221) and (568.33,160) .. (593.33,135) ;
\draw  [draw opacity=0][fill={rgb, 255:red, 245; green, 166; blue, 35 }  ,fill opacity=0.5 ] (151.12,107.55) .. controls (157.37,53.22) and (204.5,11) .. (261.71,11) .. controls (323.17,11) and (373,59.73) .. (373,119.83) .. controls (373,164.36) and (345.65,202.65) .. (306.48,219.5) -- (261.71,119.83) -- cycle ; \draw   (151.12,107.55) .. controls (157.37,53.22) and (204.5,11) .. (261.71,11) .. controls (323.17,11) and (373,59.73) .. (373,119.83) .. controls (373,164.36) and (345.65,202.65) .. (306.48,219.5) ;  
\draw  [dash pattern={on 0.84pt off 2.51pt}]  (281.33,221) -- (281.33,318) ;
\draw [shift={(281.33,318)}, rotate = 90] [color={rgb, 255:red, 0; green, 0; blue, 0 }  ][fill={rgb, 255:red, 0; green, 0; blue, 0 }  ][line width=0.75]      (0, 0) circle [x radius= 3.35, y radius= 3.35]   ;
\draw [shift={(281.33,221)}, rotate = 90] [color={rgb, 255:red, 0; green, 0; blue, 0 }  ][fill={rgb, 255:red, 0; green, 0; blue, 0 }  ][line width=0.75]      (0, 0) circle [x radius= 3.35, y radius= 3.35]   ;
\draw [fill={rgb, 255:red, 245; green, 166; blue, 35 }  ,fill opacity=0.5 ]   (151.12,107.55) .. controls (186.21,189.55) and (194.53,195.12) .. (202.53,202.12) .. controls (210.53,209.12) and (224.53,215.12) .. (243.53,218.12) .. controls (262.53,221.12) and (294.13,222.58) .. (306.48,219.5) ;
\draw  [color={rgb, 255:red, 245; green, 166; blue, 35 }  ,draw opacity=0 ][fill={rgb, 255:red, 245; green, 166; blue, 35 }  ,fill opacity=0.5 ] (306.48,219.5) -- (151.12,107.55) -- (261.71,119.83) -- cycle ;
\draw    (383.84,36.8) -- (336.68,57.01) ;
\draw [shift={(334.84,57.8)}, rotate = 336.8] [color={rgb, 255:red, 0; green, 0; blue, 0 }  ][line width=0.75]    (10.93,-3.29) .. controls (6.95,-1.4) and (3.31,-0.3) .. (0,0) .. controls (3.31,0.3) and (6.95,1.4) .. (10.93,3.29)   ;
\draw [color={rgb, 255:red, 74; green, 144; blue, 226 }  ,draw opacity=1 ]   (181.33,271.67) -- (204.89,249.05) ;
\draw [shift={(206.33,247.67)}, rotate = 136.17] [color={rgb, 255:red, 74; green, 144; blue, 226 }  ,draw opacity=1 ][line width=0.75]    (10.93,-3.29) .. controls (6.95,-1.4) and (3.31,-0.3) .. (0,0) .. controls (3.31,0.3) and (6.95,1.4) .. (10.93,3.29)   ;
\draw [color={rgb, 255:red, 189; green, 16; blue, 224 }  ,draw opacity=1 ][line width=1.5]    (5.66,156.8) -- (261.71,119.83) -- (551.15,74) ;
\draw  [dash pattern={on 0.84pt off 2.51pt}]  (257.15,95) -- (261.71,119.83) ;
\draw [shift={(261.71,119.83)}, rotate = 79.59] [color={rgb, 255:red, 0; green, 0; blue, 0 }  ][fill={rgb, 255:red, 0; green, 0; blue, 0 }  ][line width=0.75]      (0, 0) circle [x radius= 3.35, y radius= 3.35]   ;
\draw [shift={(257.15,95)}, rotate = 79.59] [color={rgb, 255:red, 0; green, 0; blue, 0 }  ][fill={rgb, 255:red, 0; green, 0; blue, 0 }  ][line width=0.75]      (0, 0) circle [x radius= 3.35, y radius= 3.35]   ;
\draw  [dash pattern={on 0.84pt off 2.51pt}]  (261.71,119.83) -- (266.15,144) ;
\draw [shift={(266.15,144)}, rotate = 79.59] [color={rgb, 255:red, 0; green, 0; blue, 0 }  ][fill={rgb, 255:red, 0; green, 0; blue, 0 }  ][line width=0.75]      (0, 0) circle [x radius= 3.35, y radius= 3.35]   ;
\draw [shift={(261.71,119.83)}, rotate = 79.59] [color={rgb, 255:red, 0; green, 0; blue, 0 }  ][fill={rgb, 255:red, 0; green, 0; blue, 0 }  ][line width=0.75]      (0, 0) circle [x radius= 3.35, y radius= 3.35]   ;

\draw (255,223.4) node [anchor=north west][inner sep=0.75pt]    {$v_{n}$};
\draw (270.3,125) node [anchor=north west][inner sep=0.75pt]  [font=\small,rotate=-349.42]  {$w_{n}$};
\draw (276,323.4) node [anchor=north west][inner sep=0.75pt]    {$0$};
\draw (545,110) node [anchor=north west][inner sep=0.75pt]  [font=\large]  {$\Omega '_{n}$};
\draw (555,265) node [anchor=north west][inner sep=0.75pt]  [font=\large,color={rgb, 255:red, 208; green, 2; blue, 27 }  ,opacity=1 ]  {$\Omega $};
\draw (387,17) node [anchor=north west][inner sep=0.75pt]  [font=\normalsize]  {$D_{n} :=\Omega '_{n} \cap B( w_{n} ,|v_{n} |)$};
\draw (155,271.4) node [anchor=north west][inner sep=0.75pt]  [font=\large,color={rgb, 255:red, 74; green, 144; blue, 226 }  ,opacity=1 ]  {$C_{\alpha }$};
\draw (481,88.4) node [anchor=north west][inner sep=0.75pt]  [font=\large]  {$\textcolor[rgb]{0.74,0.06,0.88}{L}$};
\draw (264.77,84.02) node [anchor=north west][inner sep=0.75pt]  [font=\small,rotate=-350.81,xslant=-0.03]  {$w_{n} +\delta _{n} z_{n}$};
\draw (272.59,138.1) node [anchor=north west][inner sep=0.75pt]  [font=\small,rotate=-352.33]  {$w_{n} -\delta _{n} z_{n}$};

\end{tikzpicture}

    \caption{A detailed version of Fig. \ref{fig:shift_v_n}.}
    \label{fig:shift_v_n_zoom}
\end{figure}

Let $\displaystyle z_n=\frac{\nabla u_{\Omega'_n}}{|\nabla u_{\Omega'_n}|}(w_n)$ and $\delta_n\in(0,d(w_n,\partial D_n))$ be sufficiently small so that
\begin{align}\label{ag20.9}
  \frac{1}{2\delta_n}  |u_{\Om'_n}(w_n + \delta_n z_n)- u_{\Om'_n}(w_n - \delta_n z_n)|\geq J_{\max} -3/n.
\end{align}

We have
\begin{align}\label{ag20.2}
  |u_{\Om'_n}(w_n + \delta_n z_n)- u_{\Om'_n}(w_n - \delta_n z_n)|
  =| \E_{w_n + \delta_n z_n} (\sigma(\Om'_n)) - \E_{w_n - \delta_n z_n} (\sigma(\Om'_n))|.
\end{align}

Let $L$ be the line of symmetry for $w_n + \delta_n z_n$ and $w_n - \delta_n z_n$.
We can assume that the expectations in \eqref{ag20.2} correspond
to mirror-coupled Brownian motions with respect to $L$. In other words,
the two Brownian motions are mirror images of each other with respect to $L$
until they hit $L$ or exit $\Om'_n$. If the processes hit $L$ before
exiting $\Om'_n$, their contributions to the difference
$\E_{w_n + \delta_n z_n} (\sigma(\Om'_n)) - \E_{w_n - \delta_n z_n} (\sigma(\Om'_n))$
cancel each other. Thus
\begin{align}\label{ag20.3}
  | \E_{w_n + \delta_n z_n} &(\sigma(\Om'_n)) - \E_{w_n - \delta_n z_n} (\sigma(\Om'_n))|\\
  &= \left| \E_{w_n + \delta_n z_n} \left( \sigma(\Om'_n) 1_{\sigma(\Om'_n) < \tau(L)}\right)
  - \E_{w_n - \delta_n z_n} \left( \sigma(\Om'_n) 1_{\sigma(\Om'_n) < \tau(L)}\right)\right| \notag \\
  &\leq  \E_{w_n + \delta_n z_n} \left( \sigma(\Om'_n) 1_{\sigma(\Om'_n) < \tau(L)}\right)
  + \E_{w_n - \delta_n z_n} \left( \sigma(\Om'_n) 1_{\sigma(\Om'_n) < \tau(L)}\right) \notag \\
  &\leq  \E_{w_n + \delta_n z_n} \left( \tau(\prt \Om'_n \cup L)\right)
  + \E_{w_n - \delta_n z_n} \left( \tau(\prt \Om'_n \cup L)\right). \notag 
\end{align}

Let $W_t$ denote a Brownian motion and $B=\calB( w_{n} ,|v_{n} |)$.
Note that if the Brownian motion starts in $D_n$, then, 
$\{\tau(\prt \Om'_n \cup L \cup \prt B)=\tau( \prt B \cap  \Om'_n)\}
\subset \{\tau( L \cup \prt B)=\tau( \prt B)\}$.

By the strong Markov property applied at the hitting time of $\prt \Om'_n \cup L \cup\prt B$,
\begin{align}\label{ag20.8}
    \E_{w_n + \delta_n z_n} &\left( \tau(\prt \Om'_n \cup L)\right)\\
    &= \E_{w_n + \delta_n z_n} \left( \tau(\prt \Om'_n \cup L \cup \prt B)\right) \notag \\
    &\quad + \E_{w_n + \delta_n z_n} \left(
    \E_{W\left( \tau(\prt \Om'_n \cup L \cup \prt B)\right)}\left( \tau(\prt \Om'_n \cup L)\right)
    1_{\tau(\prt \Om'_n \cup L \cup \prt B)=\tau( \prt B \cap \Om'_n)}
    \right) \notag \\
    &\leq \E_{w_n + \delta_n z_n} \left( \tau(L\cup \prt B)\right) \notag \\
    &\quad + \P_{w_n + \delta_n z_n} \left(\tau(\prt \Om'_n \cup L \cup \prt B)=\tau( \prt B \cap  \Om'_n)\right)
    \sup_{x\in \prt B \cap  \Om'_n}
    \E_{x}\left( \tau(\prt \Om'_n \cup L )\right) \notag \\
    &\leq \E_{w_n + \delta_n z_n} \left( \tau(L\cup \prt B)\right) \notag \\
    &\quad + \P_{w_n + \delta_n z_n} \left(\tau( L \cup \prt B)=\tau( \prt B)\right)
    \sup_{x\in \prt B \cap \wt C_{\alpha,\eta} }
    \E_{x}\left( \tau(\prt \wt C_{\alpha,\eta} )\right). \notag 
\end{align}

Let $D_*$ be the strip of width $|v_n|$, whose boundary includes $L$, and which contains $w_n+ \delta_n z_n$.
In the following, we apply a standard estimate for one-dimensional Brownian motion in an interval of length $|v_n|$ to two-dimensional Brownian motion in $D_*$. We have
\begin{align}\label{ag20.6}
    & \E_{w_n + \delta_n z_n} \left( \tau(L\cup \prt D_n)\right)
    \leq \E_{w_n + \delta_n z_n} \left( \tau(D_*)\right)
    =  \delta_n (|v_n|-\delta_n) \leq \delta_n |v_n|.
\end{align}

Let $B_1 = \{(x_1,x_2)\in \calB((0,0), 1): x_2>0\}$ and $L_1$ be the horizontal axis. The function $(x_1,x_2)\in\R^2 \longmapsto x_2$ is positive harmonic in $B_1$ and vanishes on $L_1$. The same is true of the function $x\longmapsto \P_x(\tau(\prt B_1)< \tau(L_1))$. By the boundary Harnack principle, for some $c_3$ and all $a \in (0, 1/2)$, $ \P_{(0,a)}(\tau(\prt B_1)< \tau(L_1)) \leq c_3 a$. 
Let $B_2 = \{(x_1,x_2)\in \calB((0,0), a_1): x_2>0\}$ for some $a_1>0$. By scaling, for the same constant $c_3$, and all $a_2 \in (0, a_1/2)$, $ \P_{(0,a_2)}(\tau(\prt B_2)< \tau(L_1)) \leq c_3 a_2/a_1$. The estimate also applies to translations and rotations of $B_2$. Hence,
\begin{align}
    & \P_{w_n + \delta_n z_n} \left(\tau( L \cup \prt B)=\tau( \prt B)\right)
    \leq c_3 \delta_n/|v_n|.\label{ag20.7}
\end{align}

It follows from \eqref{ag20.5} that for large $n$,
\begin{align*}
    \sup_{x\in \prt B \cap \wt C_{\alpha,\eta} }
    \E_{x}\left( \tau(\prt \wt C_{\alpha,\eta} )\right) = \sup_{x\in \prt B \cap \wt C_{\alpha,\eta} } \wt u_{\alpha,\eta}(x)
    \leq c_2 |v_n|^{\pi/\alpha} .
\end{align*}

We combine the last estimate with \eqref{ag20.8}, \eqref{ag20.6} and \eqref{ag20.7} to obtain 
\begin{align*}
    \E_{w_n + \delta_n z_n} &\left( \tau(\prt \wt C_{\alpha,\eta} \cup L)\right)
    \leq  \delta_n |v_n| + c_3 (\delta_n/|v_n|) c_2 |v_n|^{\pi/\alpha}
    =  \delta_n |v_n| + c_4 \delta_n  |v_n|^{\frac{\pi}{\alpha}-1}.
\end{align*}

A similar argument gives an analogous estimate for $\E_{w_n - \delta_n z_n}\left( \tau(\prt \wt C_{\alpha,\eta} \cup L)\right)$ (we may have to adjust the value of $c_4$)
so \eqref{ag20.3} yields
\begin{align*}
  |& \E_{w_n + \delta_n z_n} (\sigma(\Om'_n)) - \E_{w_n - \delta_n z_n} (\sigma(\Om'_n))|\\
  &\leq  \E_{w_n + \delta_n z_n} \left( \tau(\prt \wt C_{\alpha,\eta} \cup L)\right)
  + \E_{w_n - \delta_n z_n} \left( \tau(\prt \wt C_{\alpha,\eta} \cup L)\right)
  \leq 2  \delta_n |v_n| + 2c_4 \delta_n  |v_n|^{\pi/\alpha-1},
\end{align*}
and this, combined with \eqref{ag20.2} results in
\begin{align*}
  \frac{1}{2\delta_n}  |u_{\Om'_n}(w_n + \delta_n z_n)- u_{\Om'_n}(w_n - \delta_n z_n)|
  &\leq \frac{1}{2\delta_n} (2  \delta_n |v_n| + 2c_4 \delta_n  |v_n|^{\pi/\alpha-1})\\
 & = 2   |v_n| + 2c_4   |v_n|^{\pi/\alpha-1}.
\end{align*}

Therefore, in view of \eqref{ag20.9}, we see that for all $n$ large enough, one has 
$$J_{\max}-\frac{3}{n}\leq 2   |v_n| + 2c_4   |v_n|^{\pi/\alpha-1},$$
where the constant $c_4$ does not depend on $n$. This provides a contradiction since $|v_n|\underset{n\to +\infty}{\longrightarrow} 0$. We conclude that the limit set $\Omega$ does not have a corner at the origin. \vspace{5mm}

\emph{Step 3.}
From now on, we will assume that the limiting set $\Om$ does not have a corner at $0$. 

Before we continue the main argument, we establish an inequality.
Suppose that $f\geq 0$ and $g$ are real functions, and $\beta>0$. Let
\begin{align*}
    Q&= \{(x,y): - \beta \leq x < \beta, f(x) < y < g(x)\},\\
   M_-&= \{(x,y)\in \prt Q:  x = - \beta\}  ,\\
    M_+&= \{(x,y)\in \prt Q:  x =  \beta\},\\
    \prt_+ Q &=\{(x,y)\in \prt Q:  y=g(x) \}  .
\end{align*}

Fix $\lambda_1\in(0,1/8)$ and assume that $f(0)=0$, $f$ and $g$ are Lipschitz with the constant $\lambda_1$,
and $\beta/2\leq g(x)\leq 2 \beta$ for all $- \beta \leq x < \beta$.
We can apply the boundary Harnack principle
in the spirit of Lemma \ref{lem:BHPgeneral} \textit{(ii)} to conclude that
there exists $p_0<\infty$ depending only on $\lambda_1$ such that for a Brownian motion $W$
and any $(0,y)\in Q$,
\begin{align}\notag
  & \P_{(0,y)}(W_{\sigma(Q)} \in M_- \cup M_+)
\leq p_0 \,   \P_{(0,y)}
(W_{\sigma(Q)} \in \prt_+ Q ),
\end{align}
hence, 
\begin{align}
  & (1+p_0)\P_{(0,y)}(W_{\sigma(Q)} \in M_- \cup M_+)\notag\\
&\qquad \leq p_0 \left(  \P_{(0,y)}(W_{\sigma(Q)} \in \prt_+ Q )+\P_{(0,y)}(W_{\sigma(Q)} \in M_- \cup M_+)\right),\notag
\end{align}
which implies with $p= p_0/(1+p_0)<1$ that,
\begin{align}
\label{ag16.2}
  & \P_{(0,y)}(W_{\sigma(Q)} \in M_- \cup M_+)
\leq p\,   \P_{(0,y)}(W_{\sigma(Q)} \in \prt_+ Q \cup  M_- \cup M_+). 
\end{align}

We will use the estimate \eqref{ag16.2} in later steps. 

\emph{Step 4}.
Fix an arbitrarily small $\eps\in (0,1/4)$ and let $R_\eps> 16\eps$ be a constant depending only on $\eps$  and $\lambda_1$ fixed in Step 3,
 so large that
\begin{align}\label{ag16.3}
    p^{R_\eps-1} < \eps/2,
\end{align}
where $p$ is as in \eqref{ag16.2}. 

Let $c_\alpha = u_\Omega(0,\alpha)$ for $\alpha>0$ such that $(0,\alpha)\in \Om$.
Let $A$ be the connected component of 
$\{(x,y)\in \Omega: |x|\leq R_\eps \alpha,\  u_\Omega(x,y)\leq c_\alpha\}$
that contains $0$ in its boundary.

By Remark \ref{oldres} \textit{(ii)},
the superlevel sets $\{x\in \Om: u_\Om(x) > c_\alpha\}$ are convex and converge in the Hausdorff metric to $\Om$ as $\alpha\to0$.
This implies that there exists $\alpha_0\in(0,1)$ such that for $\alpha\in(0,\alpha_0)$ there exist functions $f$ and $g_\alpha$ such that $f(0)=0$, $f$ is non-negative, and
\begin{align*}
    A = \{(x,y): |x|\leq R_\eps \alpha, f(x) < y < g_\alpha(x)\}.
\end{align*}

\emph{Step 5}.
We will prove that there exists $\alpha_1\in(0,\alpha_0)$ such that if  $\alpha\in(0,\alpha_1)$, then, for $|x|\leq R_\eps \alpha$,
\begin{align}\label{a23.1}
    f(x) &\leq \eps \alpha,\\
    |g_\alpha(x)-\alpha| &< \eps \alpha.\label{a23.2}
\end{align}

Since $\Om$ does not have a corner at $0$, we can
find $y^+_1>0$ (depending only on $\eps$ and $\Om$) so small that if  $y^+:=(y^+_1, y^+_2) \in \prt \Om$ then  $y^+_2/y^+_1 < \eps/(4R_\eps )$ and if $y^-=(-y^+_1, y^-_2) \in \prt \Om$ 
then $y^-_2/y^+_1 < \eps/(4R_\eps )$. 
Suppose that 
\begin{enumerate}[label=(\alph*)]
    \item $L$ is the graph of a convex function on the interval $[-y^+_1,y^+_1]$ and lies in the closed upper half-plane,
    \item if
$(y^+_1,z^+_2)\in L$ then $z^+_2/y^+_1 < \eps/(2R_\eps )$, and
\item if $(-y^+_1,z^-_2)\in L$ then $z^-_2/y^+_1 < \eps/(2R_\eps )$.
\end{enumerate}

Consider any point $v=(v_1,v_2) \in L$ with $|v_1| \leq y^+_1/2$.
Let $z^*_2 =\max(z^-_2, z^+_2)$. Then, by convexity of $L$, $v_2 \leq z^*_2$, and the right-hand side slope of $L$ at $v$ is bounded above by
\begin{align*}
    \frac{|z^+_2 - v_2|}{| y^+_1 - v_1|}
    \leq \frac{ \eps y^+_1/(2R_\eps )}{y^+_1/2} = \eps/R_\eps .
\end{align*}

By symmetry, the left-hand side slope of $L$ at $v$ is bounded below by $-\eps/R_\eps $.
Since $L$ is the graph of a convex function, it follows that the absolute
value of its slope is bounded by $\eps/R_\eps $ on the set $\{(x,y)\in L: |x|\leq y^+_1/2\}$.

We will apply the above estimate to two curves. First, consider $L_1:=\{(x,y)\in \prt \Om: |x|\leq y^+_1/2\}$.
The conditions (a)-(c) hold for $L_1$ because $\Om$ is a convex domain in the upper
half-plane touching $(0,0)$, and by the definition of $y^+_1$.
Suppose that $\alpha_1\in(0, y^+_1/(2R_\eps ))$, $\alpha\in(0,\alpha_1]$, and $v=(v_1,v_2) \in L_1$ with $|v_1| \leq \alpha R_\eps 
\leq y^+_1/2$.

Then, we obtain the following bound by using the slope bound $\eps/R_\eps $ between the points $v,(0,0)\in L_1$,
\begin{align*}
    v_2 = |v_2-0| \leq (\eps/R_\eps ) |v_1 - 0| = \eps |v_1|/R_\eps  \leq \eps(\alpha R_\eps )/R_\eps  = \eps \alpha.
\end{align*}
This proves \eqref{a23.1}.

In view of the definition of $y^+_1$, continuity of $u$, and the fact that $u$ vanishes on $\prt\Om$,
we can make $\alpha_1>0$ smaller, 
if necessary, so that if $\alpha\in(0,\alpha_1]$, 
$L_2= \{(x,y): |x| \leq y^+_1, y=g_\alpha(x)\}$,
$(-y^+_1, w^-_2)\in L_2$, and
$(y^+_1, w^+_2)\in L_2$,
then $w^-_2 < \eps/(2R_\eps ) $ and $w^+_2 < \eps/(2R_\eps ) $.
These conditions and the convexity of the superlevel sets of $u$ (see Remark \ref{oldres} \textit{(ii)}) imply that
the curve $L_2$ satisfies (a)-(c).
Suppose that $\alpha_1\in(0, y^+_1/(2R_\eps ))$, $\alpha\in(0,\alpha_1]$, and $v=(v_1,v_2) \in L_2$ with $|v_1| \leq \alpha R_\eps \leq y^+_1/2$.
Then we obtain the following bound by using the slope bound $\eps/R_\eps $ between the points $v,(0,\alpha)\in L_2$,
\begin{align*}
     |v_2-\alpha| \leq (\eps/R_\eps ) |v_1 - 0| = \eps |v_1|/R_\eps  \leq \eps(\alpha R_\eps )/R_\eps  = \eps \alpha.
\end{align*}

This proves \eqref{a23.2}.

\emph{Step 6}.
Let $A_n$ be the connected component of $\{(x,y)\in  \Om_n: |x|\leq R_\eps \alpha, u_{\Om_n}(x,y) < c_\alpha\}$ whose boundary contains the origin $0$ and 
\begin{align*}
  \prt_+ A_n &= \{(x,y)\in \prt A_n: u_{\Om_n}(x,y) = c_\alpha\},\\
  \prt_s A_n &=  \{(x,y)\in  \prt A_n: |x|= R_\eps \alpha\}.
\end{align*}

We will argue that there exists $n_0$, depending only on $\eps$ and $\alpha$, such that for $n\geq n_0$, 
\begin{align}\label{a25.1}
   & \{(x,y)\in  \prt\Om_n: |x|\leq R_\eps \alpha  \} \subset \{(x,y): |y| < 2\eps \alpha\},\\
   & \{(x,y) \in \prt_+ A_n: |x|\leq R_\eps \alpha\} \subset  \{(x,y): |y-\alpha| < 2\eps \alpha\}.\label{a25.2}
\end{align}

The inclusion \eqref{a25.1} is a simple consequence of \eqref{a23.1}, the Hausdorff convergence of the sequence $\Omega_n$ to $\Omega$ and the fact that we have assumed that $\Omega$ does not have a corner at $0$.
To prove \eqref{a25.2}, note that $u_{\Omega_n}$ converge to $u_{\Omega}$ pointwise and the superlevel sets $\{u_{\Omega_n}\ge c_\alpha\}$ and $\{u_{\Omega}\ge c_\alpha\}$ are convex.
This implies the Hausdorff convergence of the  sets $\{u_{\Omega_n}\ge c_\alpha\}$ to $\{u_{\Omega}\ge c_\alpha\}$. This observation and \eqref{a23.2} imply  \eqref{a25.2}.

For similar reasons, the function representing $\prt_+ A_n$ is Lipschitz with a constant less than $2\eps/R_\eps$ for large $n$.

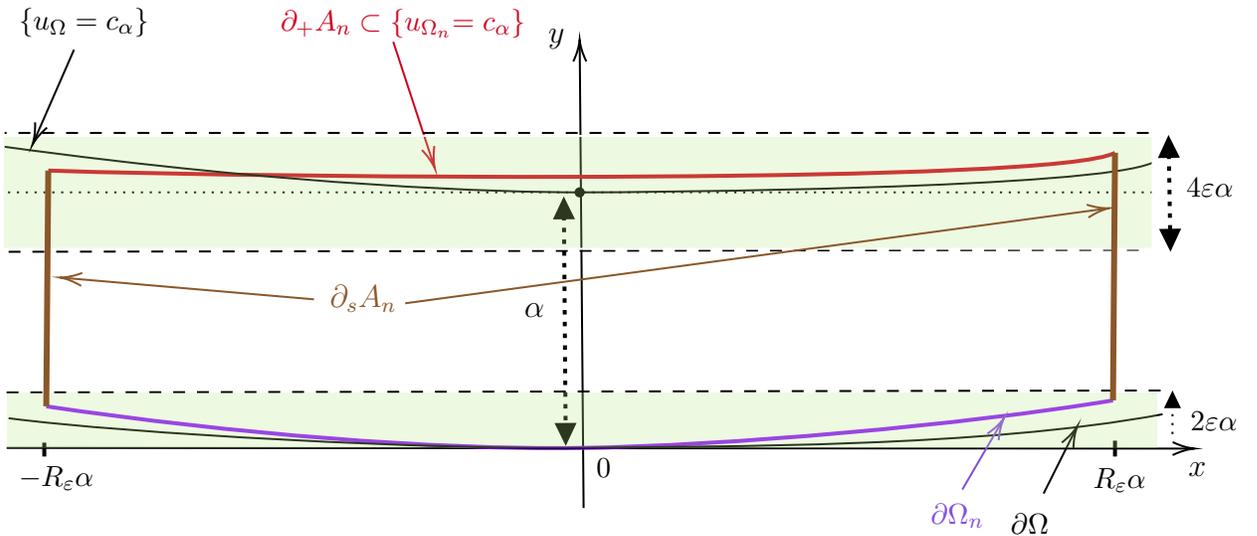
\begin{figure}[h]
\centering
\tikzset{every picture/.style={line width=0.75pt}} 

\begin{tikzpicture}[x=0.75pt,y=0.75pt,yscale=-1,xscale=1]

\draw    (21.33,300) -- (618.33,300.33) ;
\draw [shift={(620.33,300.33)}, rotate = 180.03] [color={rgb, 255:red, 0; green, 0; blue, 0 }  ][line width=0.75]    (10.93,-3.29) .. controls (6.95,-1.4) and (3.31,-0.3) .. (0,0) .. controls (3.31,0.3) and (6.95,1.4) .. (10.93,3.29)   ;
\draw    (310.35,96) -- (312.33,330.33) ;
\draw [shift={(310.33,94)}, rotate = 89.52] [color={rgb, 255:red, 0; green, 0; blue, 0 }  ][line width=0.75]    (10.93,-3.29) .. controls (6.95,-1.4) and (3.31,-0.3) .. (0,0) .. controls (3.31,0.3) and (6.95,1.4) .. (10.93,3.29)   ;
\draw [color={rgb, 255:red, 144; green, 19; blue, 254 }  ,draw opacity=1 ][line width=1.5]    (41.33,279) .. controls (136.33,293) and (236.33,301) .. (310.33,300) .. controls (384.33,299) and (503.83,288.5) .. (579.33,276) ;
\draw [color={rgb, 255:red, 208; green, 2; blue, 27 }  ,draw opacity=1 ][line width=1.5]    (42.33,160) .. controls (162.33,164) and (551.33,167) .. (580.33,151) ;
\draw [color={rgb, 255:red, 208; green, 2; blue, 27 }  ,draw opacity=1 ]   (216.33,95) -- (236.71,157.1) ;
\draw [shift={(237.33,159)}, rotate = 251.83] [color={rgb, 255:red, 208; green, 2; blue, 27 }  ,draw opacity=1 ][line width=0.75]    (10.93,-3.29) .. controls (6.95,-1.4) and (3.31,-0.3) .. (0,0) .. controls (3.31,0.3) and (6.95,1.4) .. (10.93,3.29)   ;
\draw    (544.33,323) -- (560.44,290.79) ;
\draw [shift={(561.33,289)}, rotate = 116.57] [color={rgb, 255:red, 0; green, 0; blue, 0 }  ][line width=0.75]    (10.93,-3.29) .. controls (6.95,-1.4) and (3.31,-0.3) .. (0,0) .. controls (3.31,0.3) and (6.95,1.4) .. (10.93,3.29)   ;
\draw [line width=1.5]    (40.33,304.33) -- (40.33,297.33) ;
\draw [line width=1.5]    (580.33,296.33) -- (580.33,304.33) ;
\draw [line width=1.5]  [dash pattern={on 1.69pt off 2.76pt}]  (302.37,177) -- (303.3,295) ;
\draw [shift={(303.33,299)}, rotate = 269.55] [fill={rgb, 255:red, 0; green, 0; blue, 0 }  ][line width=0.08]  [draw opacity=0] (11.61,-5.58) -- (0,0) -- (11.61,5.58) -- cycle    ;
\draw [shift={(302.33,173)}, rotate = 89.55] [fill={rgb, 255:red, 0; green, 0; blue, 0 }  ][line width=0.08]  [draw opacity=0] (11.61,-5.58) -- (0,0) -- (11.61,5.58) -- cycle    ;
\draw    (20.33,148) .. controls (166.33,168) and (271.33,171) .. (310.33,171) .. controls (349.33,171) and (571.33,169) .. (599.33,156) ;
\draw [fill={rgb, 255:red, 0; green, 0; blue, 0 }  ,fill opacity=1 ] [dash pattern={on 4.5pt off 4.5pt}]  (597.33,141) -- (20.33,141) ;
\draw  [dash pattern={on 4.5pt off 4.5pt}]  (22.33,201) -- (598.33,200) ;
\draw    (22.33,285) .. controls (83.33,293) and (238.33,301) .. (310.33,300) .. controls (382.33,299) and (513.33,300) .. (604.33,283) ;
\draw  [dash pattern={on 4.5pt off 4.5pt}]  (21.33,272) -- (609.33,271) ;
\draw  [fill={rgb, 255:red, 0; green, 0; blue, 0 }  ,fill opacity=1 ] (308.27,170.99) .. controls (308.28,169.85) and (309.21,168.93) .. (310.34,168.94) .. controls (311.48,168.95) and (312.4,169.87) .. (312.39,171.01) .. controls (312.39,172.15) and (311.46,173.07) .. (310.32,173.06) .. controls (309.18,173.05) and (308.27,172.13) .. (308.27,170.99) -- cycle ;
\draw  [dash pattern={on 0.84pt off 2.51pt}]  (599.33,171) -- (19.33,171) ;
\draw [line width=1.5]  [dash pattern={on 1.69pt off 2.76pt}]  (607.4,146) -- (608.27,197) ;
\draw [shift={(608.33,201)}, rotate = 269.03] [fill={rgb, 255:red, 0; green, 0; blue, 0 }  ][line width=0.08]  [draw opacity=0] (11.61,-5.58) -- (0,0) -- (11.61,5.58) -- cycle    ;
\draw [shift={(607.33,142)}, rotate = 89.03] [fill={rgb, 255:red, 0; green, 0; blue, 0 }  ][line width=0.08]  [draw opacity=0] (11.61,-5.58) -- (0,0) -- (11.61,5.58) -- cycle    ;
\draw [line width=0.75]  [dash pattern={on 0.84pt off 2.51pt}]  (609.33,274) -- (609.33,296) ;
\draw [shift={(609.33,271)}, rotate = 90] [fill={rgb, 255:red, 0; green, 0; blue, 0 }  ][line width=0.08]  [draw opacity=0] (8.93,-4.29) -- (0,0) -- (8.93,4.29) -- cycle    ;
\draw [color={rgb, 255:red, 132; green, 74; blue, 226 }  ,draw opacity=1 ]   (503.33,321) -- (523.33,286.73) ;
\draw [shift={(524.33,285)}, rotate = 120.26] [color={rgb, 255:red, 132; green, 74; blue, 226 }  ,draw opacity=1 ][line width=0.75]    (10.93,-3.29) .. controls (6.95,-1.4) and (3.31,-0.3) .. (0,0) .. controls (3.31,0.3) and (6.95,1.4) .. (10.93,3.29)   ;
\draw  [color={rgb, 255:red, 255; green, 255; blue, 255 }  ,draw opacity=1 ][fill={rgb, 255:red, 184; green, 233; blue, 134 }  ,fill opacity=0.24 ] (19.33,142.5) -- (599.33,142.5) -- (599.33,199.5) -- (19.33,199.5) -- cycle ;
\draw  [color={rgb, 255:red, 0; green, 0; blue, 0 }  ,draw opacity=0 ][fill={rgb, 255:red, 184; green, 233; blue, 134 }  ,fill opacity=0.24 ] (21.33,272) -- (601.33,272) -- (601.33,300) -- (21.33,300) -- cycle ;
\draw [color={rgb, 255:red, 139; green, 87; blue, 42 }  ,draw opacity=1 ][line width=2.25]    (42.33,160) -- (41.33,279) ;
\draw [color={rgb, 255:red, 139; green, 87; blue, 42 }  ,draw opacity=1 ][line width=2.25]    (580.33,151) -- (579.33,276) ;
\draw [color={rgb, 255:red, 139; green, 87; blue, 42 }  ,draw opacity=1 ]   (176.33,225) -- (50.33,214.17) ;
\draw [shift={(48.33,214)}, rotate = 4.91] [color={rgb, 255:red, 139; green, 87; blue, 42 }  ,draw opacity=1 ][line width=0.75]    (10.93,-3.29) .. controls (6.95,-1.4) and (3.31,-0.3) .. (0,0) .. controls (3.31,0.3) and (6.95,1.4) .. (10.93,3.29)   ;
\draw [color={rgb, 255:red, 139; green, 87; blue, 42 }  ,draw opacity=1 ]   (222.33,227) -- (575.35,179.27) ;
\draw [shift={(577.33,179)}, rotate = 172.3] [color={rgb, 255:red, 139; green, 87; blue, 42 }  ,draw opacity=1 ][line width=0.75]    (10.93,-3.29) .. controls (6.95,-1.4) and (3.31,-0.3) .. (0,0) .. controls (3.31,0.3) and (6.95,1.4) .. (10.93,3.29)   ;
\draw    (55.33,99) -- (35.13,145.17) ;
\draw [shift={(34.33,147)}, rotate = 293.63] [color={rgb, 255:red, 0; green, 0; blue, 0 }  ][line width=0.75]    (10.93,-3.29) .. controls (6.95,-1.4) and (3.31,-0.3) .. (0,0) .. controls (3.31,0.3) and (6.95,1.4) .. (10.93,3.29)   ;

\draw (317.33,303.73) node [anchor=north west][inner sep=0.75pt]  [font=\small]  {$0$};
\draw (158,76.73) node [anchor=north west][inner sep=0.75pt]  [font=\small,color={rgb, 255:red, 245; green, 166; blue, 35 }  ,opacity=1 ]  {${\textstyle \textcolor[rgb]{0.82,0.01,0.11}{\partial _{+} A_{n} \subset \{}\textcolor[rgb]{0.82,0.01,0.11}{u}\textcolor[rgb]{0.82,0.01,0.11}{_{\Omega _{n}}}\textcolor[rgb]{0.82,0.01,0.11}{=c}\textcolor[rgb]{0.82,0.01,0.11}{_{\alpha }}\textcolor[rgb]{0.82,0.01,0.11}{\}}}$};
\draw (526.33,331.73) node [anchor=north west][inner sep=0.75pt]  [font=\small]  {$\partial \Omega $};
\draw (26,307.73) node [anchor=north west][inner sep=0.75pt]  [font=\small]  {$-R_{\varepsilon } \alpha $};
\draw (568,308.73) node [anchor=north west][inner sep=0.75pt]  [font=\small]  {$R_{\varepsilon } \alpha $};
\draw (616,305.73) node [anchor=north west][inner sep=0.75pt]    {$x$};
\draw (293,86.73) node [anchor=north west][inner sep=0.75pt]    {$y$};
\draw (281,225.73) node [anchor=north west][inner sep=0.75pt]    {$\alpha $};
\draw (615,162.4) node [anchor=north west][inner sep=0.75pt]  [font=\small]  {$4\varepsilon \alpha $};
\draw (617,280.4) node [anchor=north west][inner sep=0.75pt]  [font=\small]  {$2\varepsilon \alpha $};
\draw (26,76.4) node [anchor=north west][inner sep=0.75pt]  [font=\small]  {$\{u_{\Omega } =c_{\alpha }\}$};
\draw (486,326.4) node [anchor=north west][inner sep=0.75pt]  [font=\small,color={rgb, 255:red, 129; green, 74; blue, 226 }  ,opacity=1 ]  {$\partial \Omega _{n}$};
\draw (183,216.4) node [anchor=north west][inner sep=0.75pt]  [color={rgb, 255:red, 139; green, 87; blue, 42 }  ,opacity=1 ]  {$\partial _{s} A_{n}$};
\end{tikzpicture}
    \caption{The domain $A_n$ and relevant notations. }
    \label{fig:domain_A}
\end{figure}

\emph{Step 7}.
Consider a Brownian motion $W_t$ starting from $0 + \delta \mathbf{e}_2$.
Then, 
\begin{align}\label{eq:oc11.2}
    u_{\Om_n}(0 + \delta \mathbf{e}_2) &=\E_{0 + \delta \mathbf{e}_2}\left(\sigma(A_n) \right) 
   +\E_{0 + \delta \mathbf{e}_2}\left(1_{W(\sigma(A_n)) \in \prt \Om_n} \E_{W(\sigma(A_n))}\left(\sigma(\Om_n) \right)\right)\\
&\qquad   +\E_{0 + \delta \mathbf{e}_2}\left(1_{W(\sigma(A_n)) \in \prt_s A_n }\E_{W(\sigma(A_n))}\left(\sigma(\Om_n) \right)\right) \notag \\
&\qquad   +\E_{0 + \delta \mathbf{e}_2}\left( 1_{W(\sigma(A_n))\in \prt_+ A_n} \E_{W(\sigma(A_n))}\left(\sigma(\Om_n) \right)\right) \notag \\
   & \leq \E_{0 + \delta \mathbf{e}_2}\left(\sigma(A_n) \right) 
     +c_\alpha \P_{0 + \delta \mathbf{e}_2}\left( W(\sigma(A_n))\in \prt_s A_n\right) \notag \\
&\qquad   +c_\alpha \P_{0 + \delta \mathbf{e}_2}\left( W(\sigma(A_n))\in \prt_+ A_n\right). \notag 
\end{align}

We will show that we can find $\alpha_2\in(0,\alpha_1)$ such that for $\alpha\in(0,\alpha_2]$,
\begin{align}\label{a24.1}
&\E_{0 + \delta \mathbf{e}_2}\left(\sigma(A_n) \right) 
     +c_\alpha \P_{0 + \delta \mathbf{e}_2}\left( W(\sigma(A_n))\in \prt_s A_n\right)\\
&\qquad   \leq 2 \eps c_\alpha \P_{0 + \delta \mathbf{e}_2}\left( W(\sigma(A_n))\in \prt_+ A_n\right),\notag
\end{align}
and, therefore,
\begin{align}\label{ag15.1}
   & u_{\Om_n}(0 + \delta \mathbf{e}_2)
   \leq (1+2\eps) c_\alpha \P_{0 + \delta \mathbf{e}_2}\left(W(\sigma(A_n))\in \prt_+ A_n\right).
\end{align}

We will prove \eqref{a24.1} in two steps.

\emph{Step 8}.
 First, we will prove that
\begin{align}\label{ag15.2}
&c_\alpha \P_{0 + \delta \mathbf{e}_2}\left( W(\sigma(A_n))\in \prt_s A_n\right)
  \leq (\eps/2) c_\alpha \P_{0 + \delta \mathbf{e}_2}\left( W(\sigma(A_n))\in \prt_+ A_n\right),
\end{align}
i.e.,
\begin{align}\label{ag15.3}
& \P_{0 + \delta \mathbf{e}_2}\left( W(\sigma(A_n))\in \prt_s A_n\right)
  \leq (\eps/2)  \P_{0 + \delta \mathbf{e}_2}\left( W(\sigma(A_n))\in \prt_+ A_n\right).
\end{align}

Let
\begin{align*}
    Q_n^k&= \{(x,y)\in A_n: (k-1) \alpha \leq x < (k+1)\alpha\},\\
    M_n^k&= \{(x,y)\in A_n:  x = k \alpha\}.
\end{align*}

The sets $Q_n^k$ satisfy the assumptions on $Q$ in Step 3 because of \eqref{a25.1}-\eqref{a25.2}
and we have assumed that $\eps<1/4$. Moreover, the function representing $\prt_+ A_n$
is Lipschitz with a Lipschitz constant less than $2\eps/R_\eps< 1/8$ because we assumed that $R_\eps > 16 \eps$.

It follows that we can apply the estimate \eqref{ag16.2} with $\beta = \alpha$ so that for a Brownian motion $W$,
$x\in M^k_n$ and $|k| \leq R_\eps -1$,
\begin{align}\label{ag16.1}
   \P_x(W_{\sigma(Q^k_n)} \in M_n^{k-1} \cup M_n^{k+1})
\leq p \, \P_x(W_{\sigma(Q^k_n)} \in \prt_+ A_n \cup M_n^{k-1} \cup M_n^{k+1}).
\end{align}

Let $T_0=0$, $k_0=0$, and for $m\geq 0$,
\begin{align*}
    T_{m+1} &= \inf\{t\geq T_m: W_t \in \prt Q_n^{k_m}\},\\
    k_{m+1} &=
    \begin{cases}
     j & \text{  if  } W_{T_m}\in M_n^j,\\
        k_m & \text{  if  } W_{T_m} \notin \bigcup_i M^i_n.
    \end{cases}
\end{align*}

If the Brownian motion $W$ starts from $0 + \delta \mathbf{e}_2$, it must hit at least $R_\eps -1$
distinct sets $M_n^i$ if it is to exit the set $A_n $ via $\prt_s A_n$.
Hence, we must have $k_{m+1}\ne k_m$ for at least $R_\eps -1$ indices $m$.
Every time the Brownian motion hits a set $M_n^i$, its chance of hitting
another set $M_n^j$ before exiting $A_n$ through $\prt_+ A_n$ is at most $p$.
By the strong Markov property applied at the stopping times $T_m$
and \eqref{ag16.1}, the probability that $W$ will exit $A_n $ via $\prt_s A_n$
rather than via $\prt_+ A_n$ is at most
$p^{R_\eps -1}$. This is bounded by $\eps/2$ according to \eqref{ag16.3} so \eqref{ag15.3}
is proved and so is \eqref{ag15.2}.

\emph{Step 9}.
We will show that 
  \begin{align}\label{eq:oc11.1}
      \E_{0 + \delta \mathbf{e}_2}\left(\sigma(A_n) \right) 
  \leq  \eps c_\alpha \P_{0 + \delta \mathbf{e}_2}\left( W(\sigma(A_n))\in \prt_+ A_n\right).
  \end{align}

It follows from \eqref{eq:bound} that there exists $\delta_n>0$
such that for $\delta\in(0,\delta_n)$,
\begin{align}\label{a24.3}
    u_{\Om_n}(0 + \delta \mathbf{e}_2) \geq \delta (J_{\max} - 2/n).
\end{align}

Since $\eps < 1/4$,  \eqref{a25.2} implies that $A_n  \subset B:= \{(x,y):0\leq y \leq 2 \alpha\}$.
Therefore, using the formula for the expected lifetime of Brownian motion in an interval, we have
\begin{align}
    \E_{0 + \delta \mathbf{e}_2}\left(\sigma(A_n) \right)
    \leq \E_{0 + \delta \mathbf{e}_2}\left(\sigma(B) \right)
    = \delta (2\alpha-\delta).
\end{align}

This and \eqref{a24.3} imply that there exists $\alpha_3 \in (0,\alpha_2)$ such that for $\alpha< \alpha_3$ and $\delta< \delta_n$,
\begin{align*}
\frac{\E_{0 + \delta \mathbf{e}_2}\left(\sigma(A_n) \right) }{  u_{\Om_n}(0 + \delta \mathbf{e}_2)}
\leq \frac{\delta (2\alpha-\delta)}{\delta (J_{\max} - 2/n)} \leq \eps/2.
\end{align*}

We use this inequality and \eqref{eq:oc11.2} to see that,
\begin{align*}
\mathbb{E}_{0+\delta\mathbf{e}_{2}}\left(\sigma\left(A_{n}\right)\right) & \leq\frac{\eps}{2}u_{\Omega_{n}}\left(0+\delta\mathbf{e}_{2}\right)\\
 & \leq\frac{\eps}{2}\mathbb{E}_{0+\delta\mathbf{e}_{2}}\left(\sigma\left(A_{n}\right)\right)+\frac{\eps}{2}c_{\alpha}\mathbb{P}_{0+\delta\mathbf{e}_{2}}\left(W\left(\sigma\left(A_{n}\right)\right)\in\partial_{s}A_{n}\right)\\
 & +\frac{\eps}{2}c_{\alpha}\mathbb{P}_{0+\delta\mathbf{e}_{2}}\left(W\left(\sigma\left(A_{n}\right)\right)\in\partial_{+}A_{n}\right).
\end{align*}

We subtract $\frac{\eps}{2}\mathbb{E}_{0+\delta\mathbf{e}_{2}}\left(\sigma\left(A_{n}\right)\right)$ from both sides 
and then divide by $\left(1-\frac{\eps}{2}\right)$ to get
\[
\mathbb{E}_{0+\delta\mathbf{e}_{2}}\left(\sigma\left(A_{n}\right)\right)\leq\frac{c_{\alpha}\frac{\eps}{2}}{\left(1-\frac{\eps}{2}\right)}\Big(\mathbb{P}_{0+\delta\mathbf{e}_{2}}\left(W\left(\sigma\left(A_{n}\right)\right)\in\partial_{s}A_{n}\right)+\mathbb{P}_{0+\delta\mathbf{e}_{2}}\left(W\left(\sigma\left(A_{n}\right)\right)\in\partial_{+}A_{n}\right)\Big)
\]
and then use \eqref{ag15.3} to arrive at
\begin{align*}
\mathbb{E}_{0+\delta\mathbf{e}_{2}}\left(\sigma\left(A_{n}\right)\right) & \leq c_{\alpha}\frac{\eps}{2}\sup_{0<\eps<\frac{1}{4}}\left(\frac{1+\frac{\eps}{2}}{1-\frac{\eps}{2}}\right)\mathbb{P}_{0+\delta\mathbf{e}_{2}}\left(W\left(\sigma\left(A_{n}\right)\right)\in\partial_{+}A_{n}\right)\\
 & \leq c_{\alpha}\eps\mathbb{P}_{0+\delta\mathbf{e}_{2}}\left(W\left(\sigma\left(A_{n}\right)\right)\in\partial_{+}A_{n}\right),
\end{align*}
which proves \eqref{eq:oc11.1}.

\emph{Step 10}.
Combining \eqref{eq:oc11.2}, \eqref{ag15.2} and \eqref{eq:oc11.1} yields 
   \begin{align}\label{eq:oc11.4}
     u_{\Om_n}(0 + \delta \mathbf{e}_2)
   \leq (1+2\eps) c_\alpha \P_{0 + \delta \mathbf{e}_2}\left(W(\sigma(A_n))\in \prt_+ A_n\right).
   \end{align}

By \eqref{a25.2}, since $\eps<1/4$,
if $(x,y)\in \prt _+ A_n$ then $y\geq \alpha(1-2\eps)$. Note that since $\eps<1/4$, then $1-2\eps>1/2$.
Let $ B':= \{(x,y):0\leq y \leq  \alpha(1-2\eps)\} $
and $ \prt_+ B':= \{(x,y): y =  \alpha(1-2\eps)\} $. Then
\begin{align}\label{eq:oc11.5}
    \P_{0 + \delta \mathbf{e}_2}\left( W(\sigma(A_n))\in \prt_+ A_n\right)
    \leq \P_{0 + \delta \mathbf{e}_2}\left( W(\sigma(B'))\in \prt_+ B'\right)
    =   \frac{\delta}{\alpha(1-2\eps)}.
\end{align}

Recall from \eqref{a24.3} that for all  $\delta\in (0,\delta_n)$,
$$\frac{u_{\Omega_n}(0 + \delta \mathbf{e}_2)}{\delta} > J_{\max}-\frac{2}{n}.$$
This, \eqref{eq:oc11.4} and \eqref{eq:oc11.5} show that
\begin{align*}
u_\Omega(0,\alpha)&= c_\alpha  \geq 
\frac{  u_{\Om_n}(0 + \delta \mathbf{e}_2)}
{    (1+2\eps)  \P_{0 + \delta \mathbf{e}_2}\left( W(\sigma(A_n))\in \prt_+ A_n\right)}\\
&\geq 
\frac{ \delta (J_{\max} - 2/n)}
{    (1+2\eps)    \delta/(\alpha(1-2\eps))}
=\frac{ \alpha(1-2\eps) (J_{\max} - 2/n)}
{    (1+2\eps)    }  .
\end{align*}
By taking the limit when $n$ tends to $+\infty$, then dividing by $\alpha>0$, we deduce that 
$$\frac{u_\Omega(0,\alpha)}{\alpha} \ge \frac{ 1-2\eps  }
{1+2\eps}\cdot J_{\max}  .$$
By the mean value theorem applied to the function $y\in (0,\alpha)\longmapsto u_\Omega(0,y)$, there exists $y_\alpha \in (0,\alpha)$ such that $\frac{\partial u_\Omega}{\partial y}(0,y_\alpha) = \frac{u_\Omega(0,\alpha)}{\alpha}$. We can then write 
$$|\nabla u_\Omega(0,y_\alpha)|\ge \left|\frac{\partial u_\Omega}{\partial y}(0,y_\alpha)\right| = \frac{u_\Omega(0,\alpha)}{\alpha} \ge \frac{ 1-2\eps  }
{1+2\eps}\cdot J_{\max}.$$
Therefore, we have proved that
$$\forall \eps\in (0,1/4),\ \ \ \ \|\nabla u_\Omega\|_\infty \ge \frac{ 1-2\eps  }
{1+2\eps}\cdot J_{\max}.$$
This shows that  $\|\nabla u_\Omega\|_\infty = J_{\max}$, with $|\Om|=1$, which proves that the limit domain $\Omega$ is actually an optimal set.  
\end{proof}

\section{The  boundary of the optimizer is $C^1$}\label{sec:nocorners}

\begin{proof}[Proof of Theorem \ref{main:thm} \textit{(ii)}]
  Suppose that  $\Om$ is a convex optimizer in the sense of \eqref{defOptimizer}. 
  We assume without loss of generality that $|\Om|=1$.
  We have $J(\Om)>0$. We choose a coordinate system so that $\Om\subset\{(x_1,x_2):x_2\geq 0\}$, $0\in \prt \Om$, and  for every $n>0$ there exists
a neighborhood $U_n\subset \prt \Om$ of $0$ such that the harmonic measure of 
 $\{y\in U_n \cap \prt_r\Om: |\nabla u_\Omega(y)|_{\mathrm{mf}}> J(\Om)-1/n\}$ is strictly positive  (see Lemma \ref{lem:supbound} (c)).

Let $C_\alpha=\{re^{i\theta}: r>0\ \text{and}\ \theta\in(0, \alpha)\}$.
The function $h_1:(x_1,x_2)\longmapsto x_2$ is positive harmonic in the upper half-plane
with zero boundary values. Suppose $\alpha\in(0,\pi)$
and let $h_2:z\longmapsto  h_1(z^{\pi/\alpha}) = r^{\pi/\alpha} \sin\frac{\pi \theta}{\alpha}$ for $z=re^{i\theta}\in C_\alpha$.
The function $h_2$ is positive harmonic in $C_\alpha$
with zero boundary values.

Recall that $\normal_D(z)$ denotes the unit inner normal vector at $z\in\prt D$ in a domain $D$.
Since $\alpha\in(0,\pi)$, for $r>0$,
\begin{align}\label{eq:h2lim}
    \lim_{\delta\to 0} \frac{1}{\delta} h_2(r e^{i \alpha} +\delta\normal_{C_\alpha}(r e^{i \alpha})) = \frac{\pi}{\alpha} r^{\pi/\alpha}.
\end{align}

We have already proved in Section \ref{sec:existopt} that $\prt\Omega$ does not have a corner at the origin. Suppose now that it has a corner at $z_1\neq 0$ so that there is a cone $C$ with vertex $z_1$ and angle $\alpha<\pi$ that contains $\Om$.
We can and will assume that $\alpha$ is the smallest angle with this property.
Let $L_1$ be the axis of symmetry of $C$ and let $L_2 $ be a line orthogonal to $L_1$ 
such that the intersection point of $L_1$ and $L_2$ lies in the interior of $\Om$. See Fig. \ref{fig13}. 

We consider $\alpha':= \max(\alpha,3\pi/4)$. For $\eta>0$, let 
$B_{\alpha',\eta}$ be the disc inscribed in $C_{\alpha'}$ such that  $(\eta,0)\in \prt B_{\alpha,\eta} \cap \prt C_{\alpha'} $. Let
\begin{align*}
    \wt C_{\alpha',\eta} =\{re^{i\theta}: r\in(0,\eta)\ \text{and}\ \theta\in (0,\alpha')\} \cup B_{\alpha',\eta}.
\end{align*}

In other words, $\wt C_{\alpha',\eta}$ is the convex hull of $B_{\alpha',\eta}$ and $0$. We then consider $\eta$ sufficiently large and $\wt C'_{\alpha',\eta}$ the image of $\wt C_{\alpha',\eta}$ via a rigid motion that maps the origin $0$ onto the corner $z_1$ and such that $\Omega\subset \wt C'_{\alpha',\eta}$. We denote by $\wt u_{\alpha',\eta}$ and $\wt u'_{\alpha',\eta}$ the torsion functions of $\wt C_{\alpha',\eta}$ and $\wt C'_{\alpha',\eta}$ respectively. 

By the inclusion $\Omega\subset \wt C'_{\alpha',\eta}$, we have for all $x\in \Omega$,
$$\E_x(\sigma(\Omega)) = u_\Omega(x) \leq \wt u'_{\alpha',\eta}(x).$$

On the other hand, the sets $\wt C_{\alpha',\eta}$ and $\wt C'_{\alpha',\eta}$ are Lipschitz domains with the Lipschitz constant $\lambda< 1$
because $\alpha'\in(\pi/2,\pi)$. Let $\wt u_{\alpha,\eta}$ be the torsion function in $\wt C_{\alpha,\eta}$.
By Proposition \ref{Lem:hvsu} \textit{(ii)} and the boundary Harnack principle (Lemma \ref{lem:BHPgeneral}), for some $c_1>0$ depending on $\eta$ and $\alpha'$, and all $x\in \wt C_{\alpha',\eta}$ such that for $|x-z_1| < \eta/2$, 
$$\wt u_{\alpha',\eta}(x) \leq c_1 |x-z_1|^{\pi/\alpha'}.$$

By combining the latter inequalities, we obtain that for every $x\in\Omega$ such that $|x-z_1| < \eta/2$,
\begin{equation}\label{eq:estimate_alpha_prime}
    \E_x(\sigma(\Omega)) \leq c_1 |x-z_1|^{\pi/\alpha'}. 
\end{equation}

\begin{figure} [h]
\includegraphics[width=0.5\linewidth]{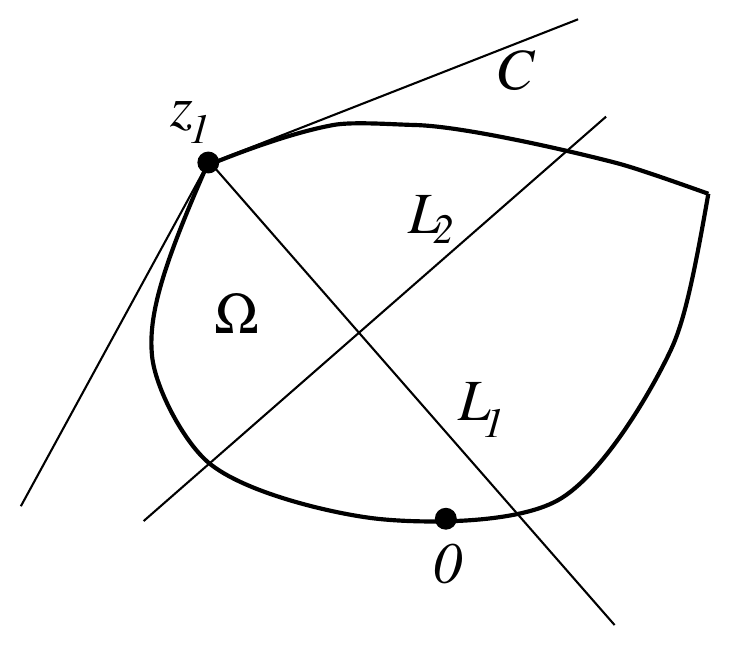}
\caption{A domain $\Omega$ with corners.}
\label{fig13}
\end{figure}
Let $L_3^\eps$ be a line orthogonal to $L_1$ crossing $\Om$, whose exact position, depending on $\eps$, will be specified below. Let $y^\eps_1$ and $y^\eps_2$ be the intersection points of  $L_3^\eps$ with $\prt\Om$.
Let $\rho_1^\eps=|y^\eps_1-y^\eps_2|$ and let $\rho_2^\eps$ be the distance from $z_1$ to $L_3^\eps$.
Let $\Om_\eps'$ be the connected component of $\Om\setminus L_3^\eps$ which contains $z_1$ in its boundary. We choose  $\rho_2^\eps$  so that 
$|\Om_\eps'|= \eps$. See Fig. \ref{fig14}.

There exists a cone $K_\eps$ of vertex $z_1$, radius $\rho_2^\eps$ and angle $\alpha/2$ which is included in $\Omega'_\eps$. Thus, 
$$|K_\eps|=\frac{\alpha}{4}{(\rho_2^\eps)}^2\leq |\Omega'_\eps| = \eps.$$

Therefore, for some $c_2<\infty$, 
\begin{align}\label{eq:rhosqrt}
    \rho_2^\eps \leq c_2 \sqrt{\eps}.
\end{align}

Let $M_1^\eps$ be one of the supporting lines for $\Om$ at $y_1^\eps$, i.e.,
a line that passes through $y_1^\eps$ and does not intersect the interior
of $\Om$. We define $M_2^\eps$ relative to $y_2^\eps$ analogously.
Let $v_\eps$ be the intersection point of $M_1^\eps$ and $M_2^\eps$.
Let $\rho_3^\eps$ be the distance from $v_\eps$ to $L_3^\eps$.
Let $\wt C_\eps$ be the cone with vertex $v_\eps$ and the sides in the lines
$M_1^\eps$ and $M_2^\eps$, containing  $\Om$.
Let $\theta_1$ and $\theta_2$ be defined by
$\wt C_\eps=\{v_\eps+ r e^{i\theta}: r>0, \theta\in(\theta_1, \theta_2)\}$.
Let $\rho_4^\varepsilon = \max(|v_\varepsilon- y_1^\varepsilon|,|v_\varepsilon- y_2^\varepsilon|)$ 
and $\rho_5^\eps$ be such that $v_\eps'= v_\eps+ \rho_5^\eps e^{i(\theta_1+\theta_2)/2}$ belongs to $L_2$ and let  $D_\eps:=\{v_\eps+ r e^{i\theta}: r\in(\rho_4^\eps,(\rho_5^\eps)^2/ \rho_4^\eps), \theta\in(\theta_1, \theta_2)\},$
see Figures \ref{fig14} and \ref{fig:distances}.

\begin{figure} [h]
\includegraphics[width=0.6\linewidth]{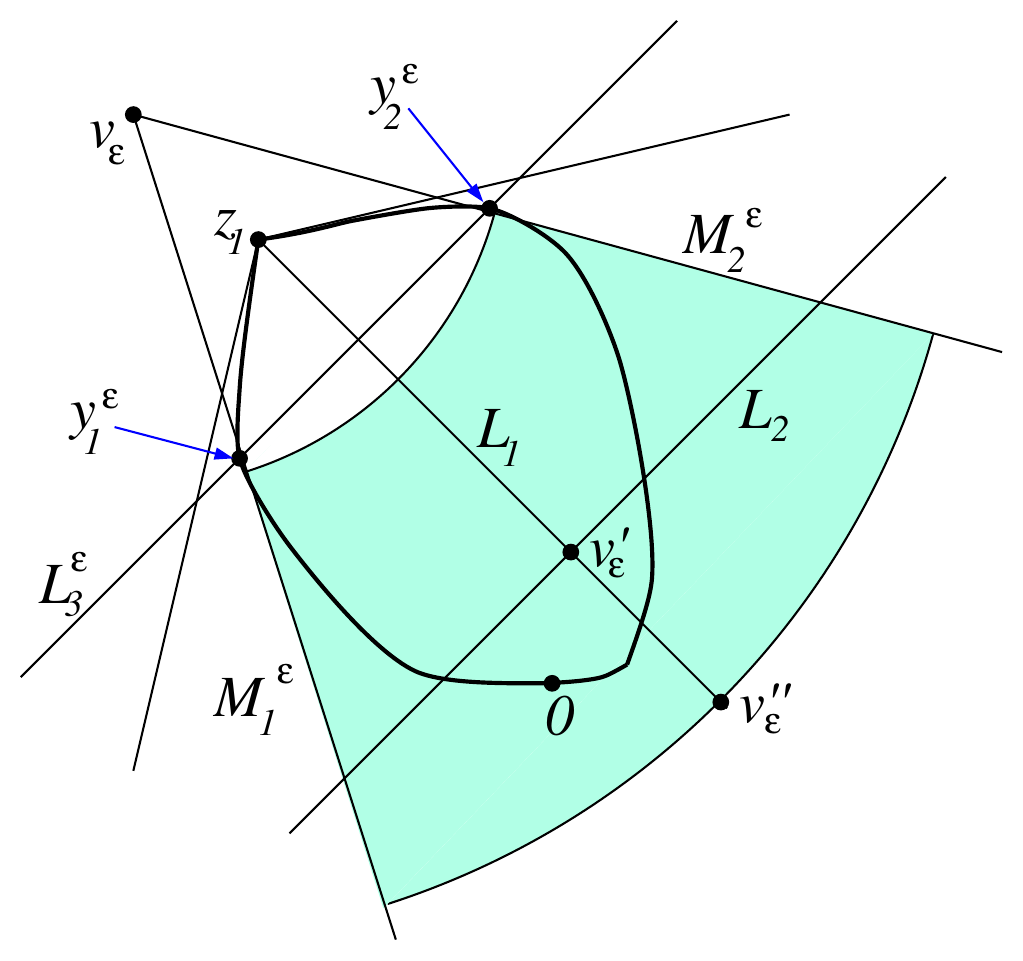}
\caption{The domain $D_\eps$ is colored. The point $v''_\eps$ is 
$v_\eps+ ((\rho_5^\eps)^2/ \rho_4^\eps) e^{i(\theta_1+\theta_2)/2}$.
The drawing is not to scale. In our construction, $|v_\eps''-v_\eps'|\gg |v_\eps-v_\eps'|$.  }
\label{fig14}
\end{figure}

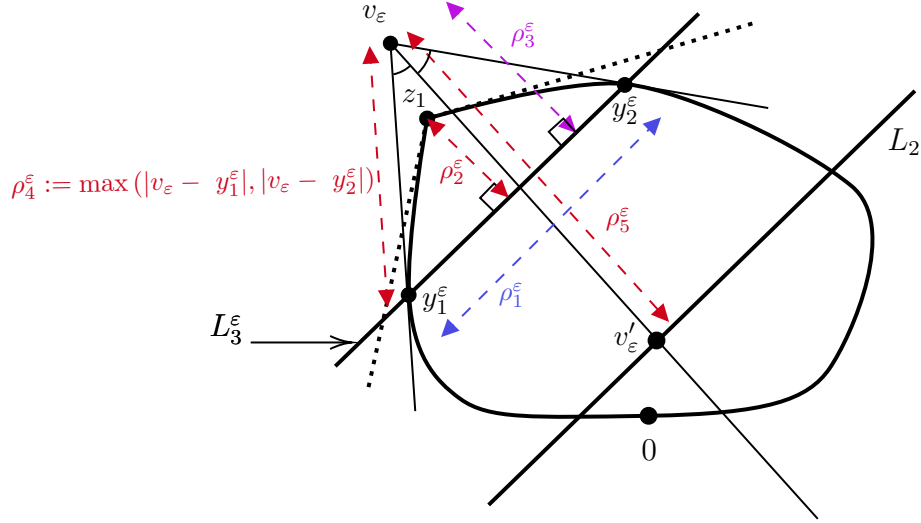
\begin{figure}[h]
    \centering

\tikzset{every picture/.style={line width=0.75pt}} 

\begin{tikzpicture}[x=0.75pt,y=0.75pt,yscale=-1,xscale=1]

\draw  [line width=1.5]  (230.33,81) .. controls (230.33,81) and (301.33,60) .. (331.33,63.67) .. controls (361.33,67.33) and (419.67,94.55) .. (444.33,115) .. controls (469,135.45) and (444.33,191) .. (433.33,207) .. controls (422.33,223) and (402.33,230) .. (342.33,231) .. controls (282.33,232) and (261.33,232.67) .. (247.33,222.67) .. controls (233.33,212.67) and (223.33,199.67) .. (221.33,171.67) .. controls (219.33,143.67) and (230.33,81) .. (230.33,81) -- cycle ;
\draw  [line width=6] [line join = round][line cap = round] (340.33,231.33) .. controls (340.33,231.33) and (340.33,231.33) .. (340.33,231.33) ;
\draw [line width=1.5]  [dash pattern={on 1.69pt off 2.76pt}]  (230.33,81) -- (200,220) ;
\draw [line width=1.5]  [dash pattern={on 1.69pt off 2.76pt}]  (394.33,33) -- (230.33,81) ;
\draw [shift={(230.33,81)}, rotate = 163.69] [color={rgb, 255:red, 0; green, 0; blue, 0 }  ][fill={rgb, 255:red, 0; green, 0; blue, 0 }  ][line width=1.5]      (0, 0) circle [x radius= 3.05, y radius= 3.05]   ;
\draw [line width=1.5]    (476.33,67) -- (261.33,276) ;
\draw [line width=1.5]    (367.33,28) -- (229.34,161.98) -- (184,206) ;
\draw    (211.83,43) -- (224.33,228.67) ;
\draw    (211.83,43) -- (402.33,75.67) ;
\draw [shift={(211.83,43)}, rotate = 9.73] [color={rgb, 255:red, 0; green, 0; blue, 0 }  ][fill={rgb, 255:red, 0; green, 0; blue, 0 }  ][line width=0.75]      (0, 0) circle [x radius= 3.35, y radius= 3.35]   ;
\draw  [line width=6] [line join = round][line cap = round] (346.33,192.67) .. controls (346.33,192.67) and (346.33,192.67) .. (346.33,192.67) ;
\draw  [color={rgb, 255:red, 74; green, 144; blue, 226 }  ,draw opacity=1 ][line width=6] [line join = round][line cap = round] (220.33,170.67) .. controls (220.33,170.67) and (220.33,170.67) .. (220.33,170.67) ;
\draw  [color={rgb, 255:red, 74; green, 144; blue, 226 }  ,draw opacity=1 ][line width=6] [line join = round][line cap = round] (331.33,64.67) .. controls (331.33,64.67) and (331.33,64.67) .. (331.33,64.67) ;
\draw [color={rgb, 255:red, 208; green, 2; blue, 27 }  ,draw opacity=1 ] [dash pattern={on 4.5pt off 4.5pt}]  (201.54,47.33) -- (210.13,173.34) ;
\draw [shift={(210.33,176.33)}, rotate = 266.1] [fill={rgb, 255:red, 208; green, 2; blue, 27 }  ,fill opacity=1 ][line width=0.08]  [draw opacity=0] (8.93,-4.29) -- (0,0) -- (8.93,4.29) -- cycle    ;
\draw [shift={(201.33,44.33)}, rotate = 86.1] [fill={rgb, 255:red, 208; green, 2; blue, 27 }  ,fill opacity=1 ][line width=0.08]  [draw opacity=0] (8.93,-4.29) -- (0,0) -- (8.93,4.29) -- cycle    ;
\draw   (257.04,121.04) -- (264.14,113.75) -- (271.43,120.84) -- (264.34,128.14) -- cycle ;
\draw [color={rgb, 255:red, 208; green, 2; blue, 27 }  ,draw opacity=1 ] [dash pattern={on 4.5pt off 4.5pt}]  (232.49,83.09) -- (269.28,118.76) ;
\draw [shift={(271.43,120.84)}, rotate = 224.11] [fill={rgb, 255:red, 208; green, 2; blue, 27 }  ,fill opacity=1 ][line width=0.08]  [draw opacity=0] (8.93,-4.29) -- (0,0) -- (8.93,4.29) -- cycle    ;
\draw [shift={(230.33,81)}, rotate = 44.11] [fill={rgb, 255:red, 208; green, 2; blue, 27 }  ,fill opacity=1 ][line width=0.08]  [draw opacity=0] (8.93,-4.29) -- (0,0) -- (8.93,4.29) -- cycle    ;
\draw    (211.83,43) -- (427.33,283) ;
\draw [color={rgb, 255:red, 208; green, 2; blue, 27 }  ,draw opacity=1 ] [dash pattern={on 4.5pt off 4.5pt}]  (222.01,39.23) -- (350.32,181.44) ;
\draw [shift={(352.33,183.67)}, rotate = 227.94] [fill={rgb, 255:red, 208; green, 2; blue, 27 }  ,fill opacity=1 ][line width=0.08]  [draw opacity=0] (8.93,-4.29) -- (0,0) -- (8.93,4.29) -- cycle    ;
\draw [shift={(220,37)}, rotate = 47.94] [fill={rgb, 255:red, 208; green, 2; blue, 27 }  ,fill opacity=1 ][line width=0.08]  [draw opacity=0] (8.93,-4.29) -- (0,0) -- (8.93,4.29) -- cycle    ;
\draw  [draw opacity=0] (221.82,54.3) .. controls (219.48,56.46) and (216.54,58) .. (213.26,58.66) -- (209.75,41.75) -- cycle ; \draw   (221.82,54.3) .. controls (219.48,56.46) and (216.54,58) .. (213.26,58.66) ;  
\draw  [draw opacity=0] (231.95,46.79) .. controls (231.11,51.12) and (228.93,54.97) .. (225.86,57.9) -- (211.42,42.8) -- cycle ; \draw   (231.95,46.79) .. controls (231.11,51.12) and (228.93,54.97) .. (225.86,57.9) ;  
\draw [color={rgb, 255:red, 74; green, 74; blue, 226 }  ,draw opacity=1 ] [dash pattern={on 4.5pt off 4.5pt}]  (239.45,189.21) -- (346.21,82.45) ;
\draw [shift={(348.33,80.33)}, rotate = 135] [fill={rgb, 255:red, 74; green, 74; blue, 226 }  ,fill opacity=1 ][line width=0.08]  [draw opacity=0] (8.93,-4.29) -- (0,0) -- (8.93,4.29) -- cycle    ;
\draw [shift={(237.33,191.33)}, rotate = 315] [fill={rgb, 255:red, 74; green, 74; blue, 226 }  ,fill opacity=1 ][line width=0.08]  [draw opacity=0] (8.93,-4.29) -- (0,0) -- (8.93,4.29) -- cycle    ;
\draw [color={rgb, 255:red, 189; green, 16; blue, 224 }  ,draw opacity=1 ] [dash pattern={on 4.5pt off 4.5pt}]  (241.97,24.07) -- (303.58,86.2) ;
\draw [shift={(305.69,88.33)}, rotate = 225.24] [fill={rgb, 255:red, 189; green, 16; blue, 224 }  ,fill opacity=1 ][line width=0.08]  [draw opacity=0] (8.93,-4.29) -- (0,0) -- (8.93,4.29) -- cycle    ;
\draw [shift={(239.86,21.94)}, rotate = 45.24] [fill={rgb, 255:red, 189; green, 16; blue, 224 }  ,fill opacity=1 ][line width=0.08]  [draw opacity=0] (8.93,-4.29) -- (0,0) -- (8.93,4.29) -- cycle    ;
\draw   (291.68,87.64) -- (299.03,80.98) -- (305.69,88.33) -- (298.34,94.99) -- cycle ;
\draw    (141.33,194) -- (190.33,194) ;
\draw [shift={(192.33,194)}, rotate = 180] [color={rgb, 255:red, 0; green, 0; blue, 0 }  ][line width=0.75]    (10.93,-3.29) .. controls (6.95,-1.4) and (3.31,-0.3) .. (0,0) .. controls (3.31,0.3) and (6.95,1.4) .. (10.93,3.29)   ;

\draw (337,241.4) node [anchor=north west][inner sep=0.75pt]    {$0$};
\draw (196,22.4) node [anchor=north west][inner sep=0.75pt]  [font=\small]  {$v_{\varepsilon }$};
\draw (216,65) node [anchor=north west][inner sep=0.75pt]  [font=\small]  {$z_{1}$};
\draw (323,182.4) node [anchor=north west][inner sep=0.75pt]  [font=\small]  {$v'_{\varepsilon }$};
\draw (322,68.4) node [anchor=north west][inner sep=0.75pt]  [font=\small,color={rgb, 255:red, 0; green, 0; blue, 0 }  ,opacity=1 ]  {$y_{2}^{\varepsilon }$};
\fill (330,64) circle (3pt);
\draw (226.34,164.38) node [anchor=north west][inner sep=0.75pt]  [font=\small,color={rgb, 255:red, 0; green, 0; blue, 0 }  ,opacity=1 ]  {$y_{1}^{\varepsilon }$};
\fill (221,170) circle (3pt);
\draw (18.84,107.01) node [anchor=north west][inner sep=0.75pt]  [font=\footnotesize,color={rgb, 255:red, 208; green, 2; blue, 27 }  ,opacity=1 ,rotate=-358.99]  {$\rho _{4}^{\varepsilon } :=\max\left( |v_{\varepsilon } -\ y_{1}^{\varepsilon } |,|v_{\varepsilon } -\ y_{2}^{\varepsilon } |\right)$};
\draw (235,100.4) node [anchor=north west][inner sep=0.75pt]  [font=\footnotesize,color={rgb, 255:red, 208; green, 2; blue, 27 }  ,opacity=1 ]  {$\rho _{2}^{\varepsilon }$};
\draw (319,124.4) node [anchor=north west][inner sep=0.75pt]  [font=\footnotesize,color={rgb, 255:red, 208; green, 2; blue, 27 }  ,opacity=1 ]  {$\rho _{5}^{\varepsilon }$};
\draw (265,162.4) node [anchor=north west][inner sep=0.75pt]  [font=\footnotesize,color={rgb, 255:red, 74; green, 82; blue, 226 }  ,opacity=1 ]  {$\rho _{1}^{\varepsilon }$};
\draw (271,31.4) node [anchor=north west][inner sep=0.75pt]  [font=\footnotesize,color={rgb, 255:red, 189; green, 16; blue, 224 }  ,opacity=1 ]  {$\rho _{3}^{\varepsilon }$};
\draw (461,84.4) node [anchor=north west][inner sep=0.75pt]    {$L_{2}$};
\draw (119,178.4) node [anchor=north west][inner sep=0.75pt]    {$L_{3}^{\varepsilon }$};


\filldraw[black] (346,193) circle (3pt);
\draw (119,178.4) node [anchor=north west][inner sep=0.75pt]    {$L_{3}^{\varepsilon }$};
\filldraw[black] (342,231) circle (3pt);

\end{tikzpicture}

    \caption{Illustration of $\rho_1^\eps$, $\rho_2^\eps$, $\rho_3^\eps$, $\rho_4^\eps$ and $\rho_5^\eps$.}
    \label{fig:distances}
\end{figure}

The function $z\longmapsto 2(\log(z-v_\eps) -\log \rho_5^\eps - i \theta_1)/(\theta_2-\theta_1)$ maps $D_\eps$ onto the rectangle
\begin{align*}
 D'_\eps:=   \{(x_1,x_2): x_1\in (2\log(\rho_4^\eps/\rho_5^\eps)/(\theta_2-\theta_1),2\log(\rho_5^\eps/\rho_4^\eps)/(\theta_2-\theta_1))  , x_2\in(0,2)\}
\end{align*}
and the point $v'_\eps$ onto $i=(0,1)$.

Let 
\begin{align*}
    \prt _L D'_\eps = \{(x_1,x_2): x_1=2\log(\rho_4^\eps/\rho_5^\eps)/(\theta_2-\theta_1)  , x_2\in(0,2)\}.
\end{align*}

A continuous trajectory starting from $v'_\eps$ which hits $L_3^\eps$ before exiting $\Om$
must hit $\prt \calB(v_\eps, \rho_4^\eps)$ before hitting any other part of the  boundary of $D_\eps$.
This and the conformal invariance of harmonic measure imply that
\begin{align}\label{eq:omcone}
    \P_{v'_\eps}(\tau(L_3^\eps) < \sigma(\Om))
    \leq  \P_{v'_\eps} (\tau(\prt \calB(v_\eps, \rho_4^\eps) )= \tau (\prt D_\eps))
    = \P_i(\tau(\prt _L D'_\eps) = \tau(\prt D'_\eps)).
\end{align}

According to \cite[Lemma IV.5.1]{GarMar},
\begin{align}\label{eq:harest}
    \P_i(\tau(\prt _L D'_\eps) = \tau(\prt D'_\eps))
    &\leq \frac {16} \pi \exp((\pi/2)\cdot 2\log(\rho_4^\eps/\rho_5^\eps)/(\theta_2-\theta_1)  )\\
    & \leq \frac {16} \pi \exp(\pi\log(\rho_4^\eps/\rho_5^\eps)/\alpha )\notag\\ 
    & = \frac {16} \pi (\rho_4^\eps)^{\pi/\alpha} (\rho_5^\eps)^{-\pi/\alpha }.  \notag
\end{align}

It is easy to see that $\rho_3^\eps/\rho_2^\eps \downarrow 1$ when $\eps\downarrow 0$.
So for some $\eps_5>0$ and all $\eps\in(0,\eps_5)$, $\rho_3^\eps\leq 2\rho_2^\eps $, and,
in view of \eqref{eq:rhosqrt}, $\rho_3^\eps\leq 2 c_2 \sqrt{\eps}$.
Since $z_1$ is the vertex of $C$ and the angle of $C$ is $\alpha < \pi$, $L_1$ the symmetry axis of $C$ and $L_3^\eps$ the orthogonal axis to $L_1$ which is of distance $\rho_2^\eps$ from the vertex $z_1$, we have $\rho_1^\eps \leq 2\tan(\alpha/2)\rho_2^\eps$, see Fig. \ref{fig:distances}.
We then apply \eqref{eq:rhosqrt} again to see that there is $c_{3}<\infty$
such that $\rho_1^\eps\leq  c_3 \sqrt{\eps}$.  We combine these estimates to obtain $\rho_4^\eps \leq \rho_1^\eps + \rho_3^\eps
\leq c_{8} (2c_2+c_3) \sqrt{\eps}= c_4\sqrt{\eps}$. At last, it is elementary to see that $\rho_5^\eps$ does not go to $0$ when $\eps \downarrow0$.

Therefore, by combining the latter estimates with \eqref{eq:estimate_alpha_prime}, we obtain the existence of a constant $c_5<\infty$ such that for every $y\in L_3^\eps\cap \Omega$
\begin{equation}\label{eq:15_10_1}
    \E_y(\sigma(\Om)) \leq  c_5 \eps^{\pi/(2\alpha')}. 
\end{equation}

These observations,  \eqref{eq:omcone} and \eqref{eq:harest} imply that
for some $c_{6},c_{7}<\infty$ and $\eps_6>0$ and all $\eps\in(0,\eps_6)$,
\begin{align*}
    \P_{v'_\eps}(\tau(L_3^\eps) < \sigma(\Om))
    \leq   c_{6}(\rho_4^\eps)^{\pi/\alpha} \leq c_{7} \eps^{\pi/(2\alpha)} .
\end{align*}

The function $y\longmapsto \P_{y}(\tau(L_3^\eps) < \sigma(\Om))$
is positive harmonic in $\Om\setminus L_3^\eps$, so  a standard application of the boundary Harnack principle and the last estimate show that 
for some $c_{8}<\infty$, independent from $\eps$, and all $y\in L_2\cap \Om_\eps$,
\begin{align*}
    \P_{y}(\tau(L_3^\eps) < \sigma(\Om))
    \leq   c_{8}\eps^{\pi/(2\alpha)}.
\end{align*}

By the strong Markov property applied at $\tau(L_2)$,
for some $r_1>0$ and $x\in \Om$ with $|x|<r_1$,
\begin{align}\label{eq:esc1}
    \P_{x}(\tau(L_3^\eps) < \sigma(\Om))
    \leq   c_{8}\eps^{\pi/(2\alpha)} \P_{x}(\tau(L_2) < \sigma(\Om)).
\end{align}

Let $\wt\Om_\eps= \Om\setminus\Om_\eps'$
and note that $\wt\Om_\eps$ is convex and $|\wt\Om_\eps|=|\Om|-\eps$.
We have for $x\in \wt\Om_\eps $,
\begin{align}\label{eq:omminus}
    \E_x(\sigma(\wt\Om_\eps)) 
    = \E_x(\sigma(\Om)) 
    - \E_x\left((\sigma(\Om)-\tau(L_3^\eps)) 1_{\{\tau(L_3^\eps) < \sigma(\Om)\}}\right).
    \end{align}
    
By using \eqref{eq:estimate_alpha_prime}, \eqref{eq:15_10_1} and \eqref{eq:esc1}, we show that for some $r_3,\eps_8, c_{9}>0$, $\eps\in(0,\eps_8)$ and $x\in \Om$ with $|x|<r_3$,
\begin{align*}
    \E_x\left((\sigma(\Om)-\tau(L_3^\eps)) 1_{\{\tau(L_3^\eps) < \sigma(\Om)\}}\right)
    &=\int_{L_3^\eps} \E_y(\sigma(\Om)) \P_x(W_{\tau(L_3^\eps)}\in dy, \tau(L_3^\eps) < \sigma(\Om))\\
    &\leq\int_{L_3^\eps} c_5 \eps^{\pi/(2\alpha')} \P_x(W_{\tau(L_3^\eps)}\in dy, \tau(L_3^\eps) < \sigma(\Om))\\
    &= c_5 \eps^{\pi/(2\alpha')} \P_x( \tau(L_3^\eps) < \sigma(\Om))\\
    &\leq c_{5} \eps^{\pi/(2\alpha')} c_{8}\eps^{\pi/(2\alpha)} \P_{x}(\tau(L_2) < \sigma(\Om))\\
    &= c_{9}\eps^{\pi/(2\alpha)+\pi/(2\alpha')} \P_{x}(\tau(L_2) < \sigma(\Om)).
\end{align*}
To simplify the notation, we consider $\beta:=\pi/(2\alpha)+\pi/(2\alpha') >1$, so that the latter estimate is written
\begin{align}\label{eq:ommintau}
    \E_x&\left((\sigma(\Om)-\tau(L_3^\eps)) 1_{\{\tau(L_3^\eps) < \sigma(\Om)\}}\right)
    \leq c_{9}\eps^\beta \P_{x}(\tau(L_2) < \sigma(\Om)).
\end{align}

Suppose that $A\subset \Om$ is a disc with a positive radius and a positive distance to $\prt\Om$.
There exists $c_{10}>0$ such that $ \E_y (\sigma(\Om)) > c_{10}$, for all $y\in A$.
An application of the strong Markov property at the hitting time of $A$
 and the boundary Harnack principle
stated in Lemma \ref{lem:BHPgeneral} \textit{(i)} shows that
for some $r_4,\eps_{9}, c_{10},c_{11}>0$, $\eps\in(0,\eps_{9})$ and $x\in\wt \Om_\eps$  with $|x|<r_4$,
\begin{align*}
    \E_x (\sigma(\Om)) \geq c_{10}\P_x (\tau(A) < \sigma(\Om))\geq c_{11}\P_x (\tau(L_2) < \sigma(\Om)).
\end{align*}

We combine this with  \eqref{eq:omminus} and \eqref{eq:ommintau} to see that
there are $r_5, \eps_{10}>0$ such that
for  $\eps\in(0,\eps_{10})$ and $x\in \Om_\eps$ with $|x|\in(0,r_5)$,
\begin{align}\notag
    \E_x(\sigma(\wt\Om_\eps)) 
    &= \E_x(\sigma(\Om)) 
    - \E_x\left((\sigma(\Om)-\tau(L_3^\eps)) 1_{\{\tau(L_3^\eps) < \sigma(\Om)\}}\right)\\
&\geq \E_x(\sigma(\Om)) 
-  c_{9}\eps^\beta \P_{x}(\tau(L_2) < \sigma(\Om))\notag\\
&\geq \E_x(\sigma(\Om)) 
-  (c_{9}/c_{11})\eps^\beta \E_x(\sigma(\Om))\notag\\
&= (1-  (c_{9}/c_{11})\eps^\beta) \E_x(\sigma(\Om))).\label{eq:afin}
    \end{align}
    
    Recall that $|\Om|=1$ and $|\wt\Om_\eps|=1-\eps$.
Fix an $\eps>0$ satisfying the above inequality and such that 
\begin{align}\label{eq:cs}
    (1-  (c_{9}/c_{11})\eps^\beta) >(1+\eps/4)\sqrt{|\Om|-\eps}=(1+\eps/4)\sqrt{1-\eps}.
\end{align}

Lemma \ref{lem:supbound} shows that we can find $z_n\in \prt_r\wt\Om_\eps$ such that $\lim_{n\to\infty} z_n = 0$, and 
\begin{align*}
    \lim_{\delta\to 0} \frac 1 \delta \E_{z_n + \delta \normal(z_n)} (\sigma(\wt\Om_\eps))&=|\nabla u_{\wt\Om_\eps}(z_n)|_{\mathrm{mf}},\\
    \lim_{\delta\to 0} \frac 1 \delta \E_{z_n + \delta \normal(z_n)} (\sigma(\Om))&\geq J(\Om)-1/n.
\end{align*}

This, \eqref{eq:cs} and \eqref{eq:afin} imply that
\begin{align*}
    |\nabla u_{\wt\Om_\eps}(z_n)|_{\mathrm{mf}}
    &=\lim_{\delta\to 0} \frac 1 \delta \E_{z_n + \delta \normal(z_n)} (\sigma(\wt\Om_\eps))\\
 &   \geq (1-  (c_{9}/c_{11})\eps^\beta) \lim_{\delta\to 0} \frac 1 \delta \E_{z_n + \delta \normal(z_n)} (\sigma(\Om))\\
  &  \geq(1+\eps/4) \sqrt{|\Om|-\eps} (J(\Om)-1/n).
\end{align*}

Since $\eps>0$ is fixed,  for some $n$,
\begin{align*}
    |\nabla u_{\wt\Om_\eps}(z_n)|_{\mathrm{mf}}  \geq (1+\eps/8) \sqrt{|\Om|-\eps} J(\Om).
\end{align*}

We see that $\|\nabla u_{\wt \Om_\eps}\|_\infty /\sqrt{|\Om|-\eps}\geq (1+\eps/8) J(\Om)$. Therefore, $\Om$ is not the optimizer
in the sense of \eqref{defOptimizer}.
The assumption that $\Om$ has a corner leads to a contradiction.

\end{proof}

\section{Line segment on the boundary of optimizer}\label{sec:line}

In this section, we prove Theorem \ref{main:thm} \textit{(iii)}. First, we restate the result to better align it with our argument. 

\begin{theorem}\label{th:line}
    Suppose that  $\Om$ is a convex optimizer in the sense of \eqref{defOptimizer}. We choose a coordinate system so that $\Om\subset\{(x_1,x_2):x_2\geq 0\}$, $0\in \prt \Om$, and  for every $n>0$ there exists
a neighborhood $U_n\subset \prt \Om$ of $0$ such that the harmonic measure of
 $\{y\in U_n \cap \prt_r\Om: |\nabla u_\Omega(y)|_{\mathrm{mf}}> J(\Om)-1/n\}$ is strictly positive  (see Lemma \ref{lem:supbound} (c)). Then,  there exists $a>0$ such that
    \begin{align}\label{segment}
        \{(x_1,0): |x_1|\leq a\} \subset\partial\Om.
    \end{align}
\end{theorem}

\begin{proof}

Suppose that \eqref{segment} is false. We will modify $\Om$ to construct a convex domain $\Omega_1$ such that $|\Omega_1|=|\Omega|$ and
\begin{align}\label{eq:contr}
    J(\Omega_1) > J(\Omega).
\end{align}

\emph{Step 1. Construction of the modified domain $\Om_1$}.

For the general idea of the construction of $\Omega_1$, we refer to Fig. \ref{fig1}. We will construct a region $\Om'$ that will be added to $\Om$. See Figure  \ref{fig3} for a magnification of $\Om'$.

\begin{figure} [h]
\includegraphics[width=0.4\linewidth]{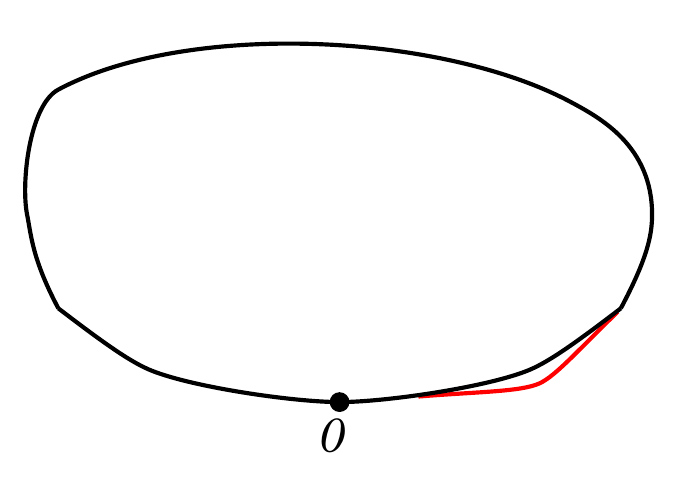}
\caption{A domain $\Omega$ that does not satisfy \eqref{segment}. A new region $\Om'$ close to $0$, with lower boundary colored red, will be added to $\Omega$. See Fig. \ref{fig3} for more detail. }
\label{fig1}
\end{figure}

According to Theorem \ref{main:thm} \textit{(ii)}, $\prt \Om$ has no corners so
we can find a small $a_*>0$ so that
we can parameterize a part of $\prt \Om$ as follows,
\begin{align*}
 \{v_a=(a,f(a))\in \prt \Om : -a_*\leq a \leq a_*\}.
\end{align*}

Since $\prt \Om$ has no corners, in particular at the origin,
we can decrease $a_*$, if necessary, so that
$f(a) < a/16$ for $-a_*\leq a\leq a_*$.

We are assuming that \eqref{segment} does not hold. Without loss of generality, we will assume that $f(a) >0$ for $0< a\leq a_*$.

Since $\prt \Om$ has no corners, for every $b>0$, the set
$\{v_a: 0\leq a \leq b\}$ cannot be a line segment.
It follows that  for every $b>0$, we can find $a_2\in(0,b)$ such that $v_{a_2}$ 
 is not in the interior of a line segment in $\prt \Om$, i.e., for all $\eps>0$, $\{v_{a}: a\in(a_2-\eps, a_2+\eps)\}  $ is not a line segment.
We will specify the value of $a_2>0$ later in the proof.
We continue the definitions assuming that $v_{a_2}$
 is not in the interior of a line segment in $\prt \Om$.
Given $a_2>0$, we find $a_1\in(0,a_2/2)$  such that $v_{a_1}$
 is not in the interior of a line segment in $\prt \Om$, i.e., for all $\eps>0$, $\{v_{a}: a\in(a_1-\eps, a_1+\eps)\}  $
is not a line segment.

 For $\delta > 0$, we consider the shift 
 \begin{align}\label{shiftT}
     \calT_\delta((x_1,x_2))  = (x_1,x_2-\delta)
 \end{align}
 and let $ x_\delta = \calT_\delta(x)$.

An idea for the enlargement of $\Om$ is to take the union of $\Om$ and the shift of a piece of the domain $\{(x_1,x_2)\in \calT_\delta(\Om) \setminus \Om: a_1< x_1 < a_2\}$, and let $\Om_1$ be the convex hull of the union. This causes an issue, as forming the convex hull may alter the boundary of $\Omega$ in a manner that is not localized within a region of arbitrarily small size.
For this reason, we will find $a_1\leq b_1 < b_2\leq a_2$ such that $b_1-a_1 < a_2/64$ and $a_2-b_2 < a_2/64$. Then, we will apply our idea with $\{(x_1,x_2)\in \calT_\delta(\Om) \setminus \Om: b_1< x_1 < b_2\}$ rather than $\{(x_1,x_2)\in \calT_\delta(\Om) \setminus \Om: a_1< x_1 < a_2\}$.
 
\begin{itemize}
    \item If
\begin{align}\label{notline}
  \forall\eps>0,\ \ \   \{v_a  :  a\in (a_1-\eps,a_1)\}\ \ \text{is not a line segment,}
\end{align}
 then, we let $z_1 = v_{a_1}$
and $z_3= \calT_\delta(z_1)$, where $\delta>0$ should be thought of as a very small number, to be specified later.
Given $a_1$, there is $\delta_*>0$ such that if  $\delta\in(0,\delta_*)$,
then the convex hull of $\Om \cup \{z_3\}$
contains a unique point $z_6 =v_{a_3}$ such that $0< a_3< a_1$ and the line segment
between $z_3$ and $z_6$ is outside the closure of $\Om$ except for $z_6$.

\item Suppose that
\begin{align*}
 \exists \eps_1>0,\ \  \   \{v_a  :  a\in (a_1-\eps_1,a_1)\}\ \ \text{is a line segment}
\end{align*}
 and let $K_1$ be the line containing this line segment.
In this case, we let
$z_6 = v_{a_1}$ and
$z_3=(z_3^1,z_3^2)$ be the intersection point of $K_1$ and $\calT_\delta(\prt\Om)$
with  $z_3^1>a_1$.
We let $z_1 = v_{a_4}$ where $a_4 = z_3^1$. \end{itemize}

We define points $z_2,z_4$ and $z_5$ analogously. Specifically,
\begin{itemize}
    \item if
\begin{align*}
   \forall \eps>0,\ \  \ \{v_a  :  a\in (a_2, a_2+\eps)\}\ \ \text{is not a line segment,}
\end{align*}
 then, we let $z_2 =v_{a_2}$
and $z_4= \calT_\delta(z_2)$.
There is $\delta_*>0$ such that if  $\delta\in(0,\delta_*)$
then the convex hull of $\Om \cup \{z_4\}$
contains a unique point $z_5 =v_{a_5}$ such that $a_5>a_2$ and the line segment
between $z_4$ and $z_5$ is outside the closure of $\Om$ except for $z_5$.

\item Suppose that
\begin{align}\label{line}
\exists \eps_2>0,\ \ \     \{v_a  :  a\in (a_2, a_2+\eps_2)\}\ \ \text{is a line segment}
\end{align}
 and let $K_2$ be the line containing this line segment.
In this case, we let
$z_5 = v_{a_2}$ and
$z_4=(z_4^1,z_4^2)$ be the intersection point of $K_2$ and $\calT_\delta(\prt\Om)$
with  $z_4^1<a_2$.
We let $z_2 = v_{a_6}$ where $a_6 = z_4^1$.
\end{itemize}

Let $b_1$ and $b_2$ denote the first coordinates of $z_1$ and $z_2$.
For small $\delta$, $a_1\leq b_1 < b_2\leq a_2$ and, moreover,
$\lim_{\delta\to0} b_1= a_1$ and $\lim_{\delta\to0}b_2= a_2$.
From now on, we will assume that
$\delta>0$ is so small that $b_1-a_1 < a_2/64$ and $a_2-b_2 < a_2/64$.

Fig. \ref{fig3} illustrates cases described in \eqref{notline} and \eqref{line}.

\begin{figure} [h]
\includegraphics[width=0.6\linewidth]{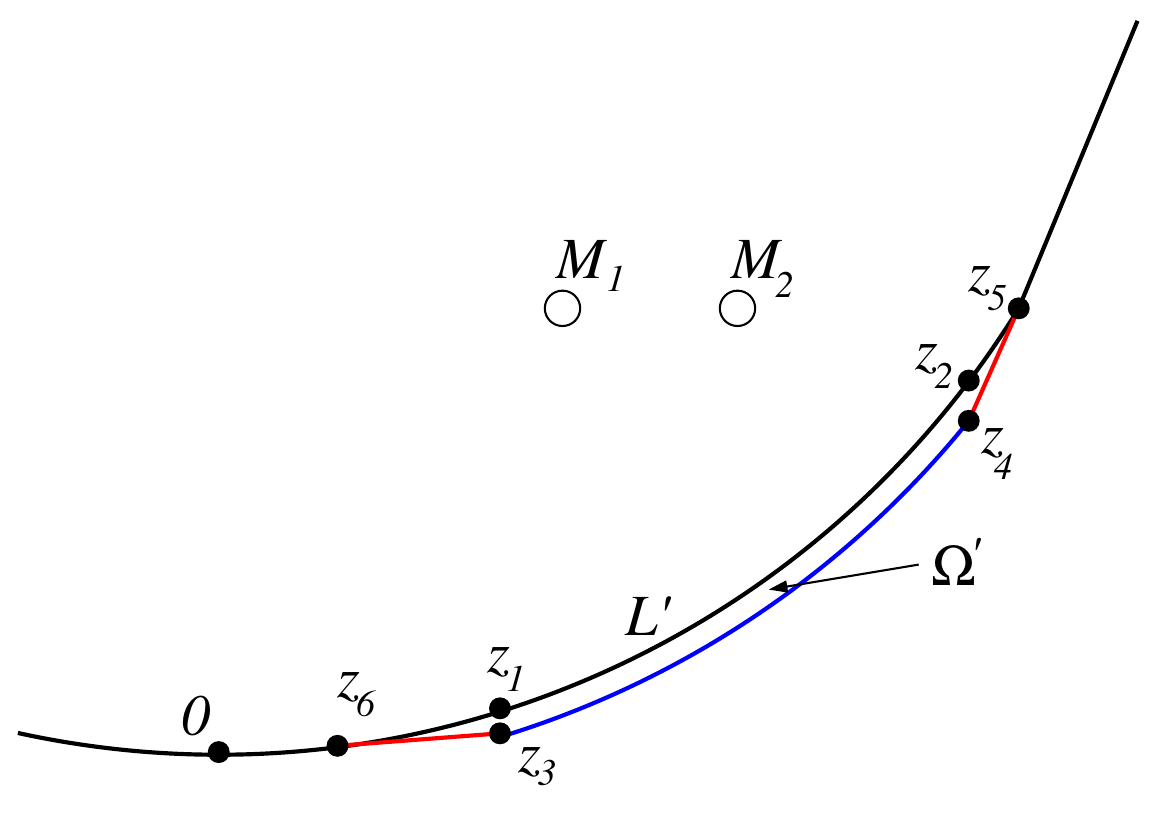}
\caption{Illustration of the cases described in \eqref{notline} and \eqref{line}. $L'$ is the arc of $\prt \Om$ between $z_1$ and $z_2$. }
\label{fig3}
\end{figure}

We will now formally define the lens $\Om' $ depicted in Fig. \ref{fig3}.
The set $\Om'$ is the closure of the union of the following three sets:

\textit{(i)} $\{(x_1,x_2)\in \calT_\delta(\Om) \setminus \Om: b_1< x_1 < b_2\}$,

\textit{(ii)} the intersection of the triangle with vertices $z_1, z_3$ and $z_6$ and $\R^2\backslash\Om$,

\textit{(iii)} the intersection of the triangle with vertices $z_2, z_4$ and $z_5$ and $\R^2\backslash\Om$.

The domain $\Om'$ depends on $a_1,a_2$ and $\delta$ which will be specified later.
We define $\Om_1 :=\Om \cup \Om' $ (to be precise, $\Om_1$ is the interior of this set). 

 Let 
$L'$ be the arc of $\prt \Om$ between $z_1$ and $z_2$. Its length is in the interval $[b_2/2, 2 b_2]$, assuming that $b_2>0$ is small.
The following estimates follow easily from the definition of  $\Om'$. See Fig. \ref{fig3}.
    For some $a_*,c_1, \delta_0>0$ and all $a_2\in(0,a_*)$ and $\delta\in(0,\delta_0)$,
    \begin{align}
         |\Om'| \leq c_1 \delta b_2.\label{area2}
    \end{align}

\emph{Step 2: Estimates.}

For $x\in\Om_1$, by the strong Markov property applied at  $\sigma(\Om)$,
\begin{align}\label{maineq}
    \E_x(\sigma(\Om_1) )  = \E_x(\sigma(\Om) )+ \E_x\left((\sigma(\Om_1)-\sigma(\Om)) 1_{\{ \sigma(\Om)< \sigma(\Om_1)\}}\right).
\end{align}
 Our goal now is to find a good lower bound for the second term on the right.

We will suppress the dependence on $b_2$ in the following notation. Let 
    \begin{align}\label{lm1}
        M_1&=\calB((7b_2/8,b_2/4), b_2/16),\\
        M_2 &=\calB((5b_2/8,b_2/4), b_2/16),\label{lm2}\\
        \prt_-M_1 &= \{(y_1,y_2)\in \prt M_1: y_2 < b_2/4\},\label{lm3}\\
\calS(r)&= \{(x_1,x_2): \max(|x_1|,|x_2|)\leq r\},\qquad r>0,\notag\\
    N_1&= \calS(b_1/1000) \cap \Om,\notag\\
    N_2 &= (\calS(6b_2)\setminus \calS(4b_2)) \cap \Om,\notag\\
    R_1&= \prt \calS(5b_2) \cap \Om,\notag\\
    R_2&= \prt\calS(7b_2) \cap\Om.\label{lm4}
\end{align}

Since $\prt\Om$ is tangential to the horizontal axis at $0$, we can choose
$b_2>0$ and $\delta>0$ so small that $N_1 \cap \Om'=\emptyset$,
$\Om'\subset \calS(3b_2)$
and the slope of $\prt \Om \cap \calS(8b_2)$ is between $-1/10$ and $1/10$. See Fig. \ref{fig4}.

\begin{figure} [h]
\includegraphics[width=0.7\linewidth]{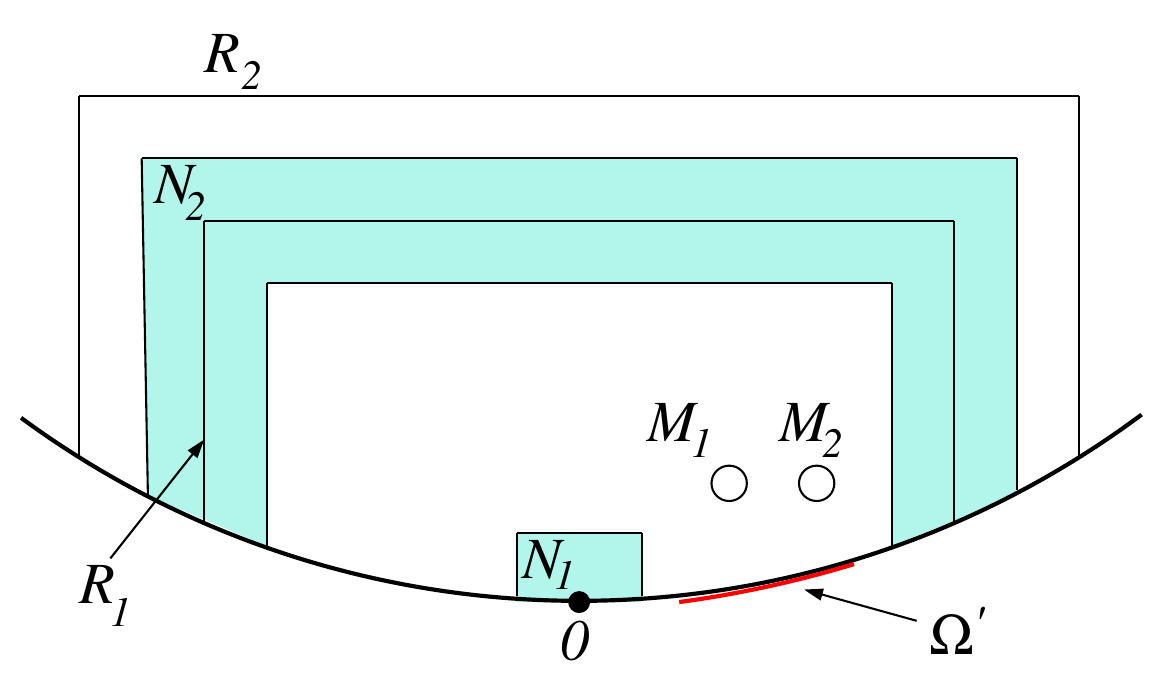}
\caption{
The red curve is a simplified representation of $\Om'$. Note that $\Om'$ is very thin and that the drawing is not to scale.}
\label{fig4}
\end{figure}

We have introduced discs $M_1$ and $M_2$ so that we can define Brownian trajectories that hit $\prt_-M_1$, then exit $\Om$, then hit $M_2$, and then spend some time in $\Om_1$ before exiting the latter domain. We can estimate the probability of going from $\prt_-M_1$ to $\prt \Om$ and then to $M_2$ before exiting $\Om_1$. Our estimate is presented in Lemma \ref{lem:dilat}. We do not know how to generate a useful estimate for the probability of exiting $\Om$ and spending some time in $\Om_1$ afterward in a more direct way.

The functions $x\longmapsto \P_x(\tau(R_2)< \sigma(\Om))$
and $x\longmapsto \P_x(\tau(\prt_-M_1)< \sigma(\Om))$ are positive harmonic
in $\calS(b_2/8)\cap \Om$ and have zero boundary values on $\prt \Om$.
The diameters of $\prt_-M_1$ and $R_2$ are comparable for every $b_2$, and, after rescaling by $b_2^{-1}$, their shapes
converge to limiting shapes, so it is easy to see that for
some constant $c'$, independent of $b_2$, 
\begin{align*}
    \P_{(0,b_2/16)}(\tau(\prt_-M_1)< \sigma(\Om)) \geq c' \P_{(0,b_2/16)}(\tau(R_2)< \sigma(\Om)).
\end{align*}

This and  Lemma \ref{lem:BHP} \textit{(i)} imply that  there exist $c''>0$
such that for all  $x\in\calS(b_2/8)\cap \Om$ and, therefore, for all $x\in N_1$,
$$  \frac{\P_x(\tau(\prt_-M_1)< \sigma(\Om))}{\P_{(0,b_2/16)}(\tau(\prt_-M_1)< \sigma(\Om))} 
    \geq c'' \frac{\P_x(\tau(R_2)< \sigma(\Om))}{\P_{(0,b_2/16)}(\tau(R_2)< \sigma(\Om))},$$
which implies that for $c_2=c'c''$ and all $x\in N_1$, we have
\begin{align}
    \P_x(\tau(\prt_-M_1)< \sigma(\Om)) &\geq c_2 \P_x(\tau(R_2)< \sigma(\Om)).\label{bhpN1}
\end{align}

A similar argument shows that for some $c_3$ and all $x\in N_1$,
\begin{align}\label{oc21.1}
    \P_x(\tau(\prt\Om \setminus \prt\Om_1)< \tau(R_2)) &\leq c_3 \P_x(\tau(R_2)< \sigma(\Om)) .
\end{align}

It is easy to see that for small $\delta$ and all $y\in \prt\Om \setminus \prt\Om_1$,  $\P_y(\tau(R_2) < \sigma(\Om_1)) < 1/2 $. This, \eqref{oc21.1} and the strong Markov property applied at the hitting time of $ \prt\Om \setminus \prt\Om_1$, show that for $x\in N_1$,
\begin{align*}
    \P_x(\tau(R_2)< \sigma(\Om_1)) &= \P_x(\tau(R_2)< \sigma(\Om)) +  \P_x(\sigma(\Om )  < \tau(R_2) < \sigma(\Om_1))\\
    &= \P_x(\tau(R_2)< \sigma(\Om)) +  \P_x(\tau(\prt\Om \setminus \prt\Om_1)  < \tau(R_2) < \sigma(\Om_1))\\
    &\leq \P_x(\tau(R_2)< \sigma(\Om)) + (1/2) \P_x(\tau(\prt\Om \setminus \prt\Om_1)  < \tau(R_2) )\\
    &\leq \P_x(\tau(R_2)< \sigma(\Om)) + (1/2) c_3 \P_x(\tau(R_2)< \sigma(\Om))\\
    &= (1+ c_3/2) \P_x(\tau(R_2)< \sigma(\Om)).
\end{align*}

This, combined with \eqref{bhpN1} yields for some $c_4$ and all $x\in N_1$,
\begin{align}
    \P_x(\tau(\prt_-M_1)< \sigma(\Om)) &\geq c_4 \P_x(\tau(R_2)< \sigma(\Om_1)).\label{bhpN1a}
\end{align}

By Lemma \ref{lem:dilat}, there exist $c_3,\eps_1>0$ such that for every $\delta\in(0,\eps_1 b_2 ) $ and $y\in \prt_- M_1$,
\begin{align}\label{hitLpp}
    \P_y(\tau(L') < \tau(M_2) <\tau(R_2) < \sigma(\Om_1)) \geq c_3 \delta/b_2.
\end{align}

Estimates \eqref{bhpN1a}, \eqref{hitLpp} and the strong Markov property applied at $\tau(\prt_-M_1)$ and then again at $\tau (M_2)$ imply that for $x\in N_1$,
\begin{align}\label{hitR2}
 \P_x&(\sigma(\Om )  < \tau(R_2) < \sigma(\Om_1))\\
 &\geq \int_{\prt_-M} \P_y(\tau(L') < \tau(M_2) < \tau(R_2) < \sigma(\Om_1)) 
 \P_x(W_{\tau(\prt_-M_1)}\in dy, \tau(\prt_-M_1) <  \sigma(\Om) ) \notag  \\
 &\geq \int_{\prt_-M} c_3 (\delta/b_2)
 \P_x(W_{\tau(\prt_-M_1)}\in dy, \tau(\prt_-M_1) <  \sigma(\Om) ) \notag\\
 &= c_3 (\delta/b_2)  \P_x (\tau(\prt_-M_1) <  \sigma(\Om) )\notag\\
 &\geq  c_3 (\delta/b_2) c_2 \P_x(\tau(R_2)< \sigma(\Om_1)).\notag
\end{align}

\emph{Step 3: Estimates for
Radon-Nikodym derivatives}

Consider a set $Q\subset R_2$ such that $v\longmapsto \P_v(W_{\tau(R_2)}\in Q, \tau(R_2) < \sigma(\Om_1) )$
is strictly positive in $N_2$.
Then, for  $Q\subset R_2$, the functions $v\longmapsto \P_v(W_{\tau(R_2)}\in Q, \tau(R_2) < \sigma(\Om_1) )$
and $v\longmapsto \P_v( \tau(R_2) < \sigma(\Om_1) )$
are non-negative harmonic in $N_2$ and vanish on $\prt \Om_1$.
By Lemma \ref{lem:BHP} \textit{(ii)}, there is $c_5>0$ such that for $Q\subset R_2$ and all $v,y\in R_1$,
\begin{align*}
\frac{\P_v(W_{\tau(R_2)}\in Q, \tau(R_2) < \sigma(\Om_1) )}{\P_y(W_{\tau(R_2)}\in Q,\tau(R_2) < \sigma(\Om_1) )}
\leq c_5 \frac{\P_v( \tau(R_2) < \sigma(\Om_1) )}{\P_y(\tau(R_2) < \sigma(\Om_1) )}.
\end{align*}

Integrating both sides with respect to the measure $\P_x(W_{\tau(R_1)}\in dv, \tau(R_1) < \sigma(\Om_1) )$
for a fixed $x\in N_1$, we obtain for $y \in R_1$,
\begin{align*}
&\frac{\int_{R_1}\P_v(W_{\tau(R_2)}\in Q, \tau(R_2) < \sigma(\Om_1) )
\P_x(W_{\tau(R_1)}\in dv, \tau(R_1) < \sigma(\Om_1) )}{\P_y(W_{\tau(R_2)}\in Q,\tau(R_2) < \sigma(\Om_1) )}\\
&\qquad\leq c_5 \frac{\int_{R_1}\P_v( \tau(R_2) < \sigma(\Om_1) )
\P_x(W_{\tau(R_1)}\in dv, \tau(R_1) < \sigma(\Om_1) )}{\P_y(\tau(R_2) < \sigma(\Om_1) )},
\end{align*}
which is equivalent to the following inequality
\begin{align*}
\frac{\P_x(W_{\tau(R_2)}\in Q, \tau(R_2) < \sigma(\Om_1) )}{\P_y(W_{\tau(R_2)}\in Q,\tau(R_2) < \sigma(\Om_1) )}
&\leq c_5 \frac{\P_x( \tau(R_2) < \sigma(\Om_1) )}{\P_y(\tau(R_2) < \sigma(\Om_1) )},
\end{align*}
that can be written as follows
\begin{align*}
\frac{\P_y(W_{\tau(R_2)}\in Q,\tau(R_2) < \sigma(\Om_1) )}{\P_x(W_{\tau(R_2)}\in Q, \tau(R_2) < \sigma(\Om_1) )}
&\geq \frac1 {c_5} \frac{\P_y(\tau(R_2) < \sigma(\Om_1) )}{\P_x( \tau(R_2) < \sigma(\Om_1) )}.
\end{align*}

Now, we integrate both sides of the last inequality with respect to the measure
\begin{align*}
    \P_x(W_{\tau(R_1)}\in dy, \sigma(\Om) <\tau(R_1) < \sigma(\Om_1) )
\end{align*}
to see that
\begin{align*}
&\frac{\int_{R_1}\P_y(W_{\tau(R_2)}\in Q,\tau(R_2) < \sigma(\Om_1) )
\P_x(W_{\tau(R_1)}\in dy, \sigma(\Om) <\tau(R_1) < \sigma(\Om_1) )}
{\P_x(W_{\tau(R_2)}\in Q, \tau(R_2) < \sigma(\Om_1) )}\\
&\geq \frac{\int_{R_1}\P_y(\tau(R_2) < \sigma(\Om_1) )
\P_x(W_{\tau(R_1)}\in dy, \sigma(\Om) <\tau(R_1) < \sigma(\Om_1) )}
{c_5 \P_x( \tau(R_2) < \sigma(\Om_1) )}.
\end{align*}

The last formula is equivalent to 
\begin{align*}
&\frac{\P_x(W_{\tau(R_2)}\in Q, \sigma(\Om) <\tau(R_2) < \sigma(\Om_1) )}
{\P_x(W_{\tau(R_2)}\in Q, \tau(R_2) < \sigma(\Om_1) )}
\geq \frac{\P_x(\sigma(\Om) <\tau(R_2) < \sigma(\Om_1) )}
{c_5 \P_x( \tau(R_2) < \sigma(\Om_1) )}.
\end{align*}

Since $Q$ is an arbitrarily subset of $R_2$, we obtain the following estimate for the
Radon--Nikodym derivative, for $x\in N_1$ and $y\in R_2$,
\begin{align*}
&\frac{\P_x(W_{\tau(R_2)}\in dy, \sigma(\Om) <\tau(R_2) < \sigma(\Om_1) )}
{\P_x(W_{\tau(R_2)}\in dy, \tau(R_2) < \sigma(\Om_1) )}
\geq \frac{\P_x(\sigma(\Om) <\tau(R_2) < \sigma(\Om_1) )}
{c_5 \P_x( \tau(R_2) < \sigma(\Om_1) )}.
\end{align*}

We combine this formula with \eqref{hitR2} to obtain for $x\in N_1$ and $y\in R_2$,
\begin{align}\label{comp}
&\frac{\P_x(W_{\tau(R_2)}\in dy, \sigma(\Om) <\tau(R_2) < \sigma(\Om_1) )}
{\P_x(W_{\tau(R_2)}\in dy, \tau(R_2) < \sigma(\Om_1) )}\\
&\qquad\geq \frac{c_4 c_3 (\delta/b_2) c_2 \P_x(\tau(R_2)< \sigma(\Om_1))}
{c_5 \P_x( \tau(R_2) < \sigma(\Om_1) )}
= (c_2 c_3 c_4 /c_5)\delta/b_2 =: c_6 \delta/b_2.\notag
\end{align}

We fix a reference disc $B=\calB(z_*,r_*)\subset \Om\setminus \calS(8b_2)$, with $r_*>0$, which is
located at a positive distance from $\prt \Om \cup \calS(8b_2)$.
 
Hence, using \eqref{comp}, for $x\in N_1$,
\begin{align}\label{lowBhit}
    \P_x&( \sigma(\Om) < \tau(B) < \sigma(\Om_1))\\
    &= \int_{R_2}\P_y(\tau(B) < \sigma(\Om_1))\P_x(W_{\tau(R_2)} \in dy, \sigma(\Om) < \tau(R_2) < \sigma(\Om_1)) \notag \\
    &\geq  c_6 (\delta/b_2) \int_{R_2}\P_y(\tau(B) < \sigma(\Om_1))\P_x(W_{\tau(R_2)}\in dy, \tau(R_2) < \sigma(\Om_1) ) \notag \\
    &= c_6 (\delta/b_2) \P_x(\tau(B) < \sigma(\Om_1)). \notag 
\end{align}

\emph{Step 4: 
Combining the estimates.}

It is easy to see that $\E_y(\sigma(\Om_1)) > c_{7}$ for some $c_{7}>0$ and all $y\in B$. Thus, in view of \eqref{lowBhit}, for $x\in N_1$,
\begin{align*}
    \E_x((\sigma(\Om_1)- \sigma(\Om)) 1_{\{\sigma(\Om) < \tau(B) < \sigma(\Om_1)\}})&\ge \E_x((\sigma(\Om_1)- \tau(B)) 1_{\{\sigma(\Om) < \tau(B) < \sigma(\Om_1)\}})\\
    &\ge c_7 \P_x(\sigma(\Om) < \tau(B) < \sigma(\Om_1)) \notag\\
    &\geq c_{7}  c_6 (\delta/b_2) \P_x(\tau(B) < \sigma(\Om_1)) \notag
\end{align*}
and, therefore, for $c_{8}=c_7c_6$ and $x\in N_1$,
\begin{align*}
\E_x((\sigma(\Om_1)-\sigma(\Om)) 1_{\{ \sigma(\Om)< \sigma(\Om_1)\}})
    &\geq \E_x((\sigma(\Om_1)-\sigma(\Om)) 1_{\{\sigma(\Om) < \tau(B) < \sigma(\Om_1)\}})\\
    &\geq 
    c_{8} (\delta/b_2) \P_x(\tau(B) < \sigma(\Om_1))\\
    &\geq 
    c_8 (\delta/b_2) \P_x(\tau(B) < \sigma(\Om)).
\end{align*}
The last inequality in the above formula holds because $\Om\subset\Om_1$.

This and \eqref{maineq} yield for $x\in N_1$,
\begin{align}\label{eq:lifees}
    \E_x(\sigma(\Om_1) )  &= \E_x(\sigma(\Om) )
     + \E_x((\sigma(\Om_1)-\sigma(\Om)) 1_{\{ \sigma(\Om)< \sigma(\Om_1)\}})\\
     & \geq  \E_x(\sigma(\Om) ) + c_8 (\delta/b_2) \P_x(\tau(B) < \sigma(\Om)).\notag
\end{align}

We can apply Proposition \ref{Lem:hvsu} \textit{(ii)} to $\Om$ because it is a Lipschitz domain
with the Lipschitz constant strictly less than 1. This follows from the convexity
of $\Om$ and the $C^1$ regularity of the boundary of $\Om$ proved in Theorem \ref{main:thm} \textit{(ii)}. Hence, according to \eqref{eq:ulessh},
for some $c_9<\infty$ and $x\in N_1$,
\begin{align*}
  \P_x(\tau(B) < \sigma(\Om)) \geq c_9  \E_x(\sigma(\Om) ) .
\end{align*}

This and \eqref{eq:lifees} show that for  $x\in N_1$,
\begin{align}\label{eq:nv252}
     \E_x(\sigma(\Om_1) )  & \geq  \E_x(\sigma(\Om) ) + c_8  c_9 (\delta/b_2) \E_x(\sigma(\Om) )
     = (1+ c_8  c_9 (\delta/b_2)) \E_x(\sigma(\Om) ) .
\end{align}

This and \eqref{area2} imply that for $x\in N_1$,
\begin{align}\label{eq:compa}
     \frac{\E_x(\sigma(\Om_1) )}{|\Om_1|^{1/2}}  & \geq 
     \frac{1+ c_8  c_9 (\delta/b_2)}{(1+ c_1\delta b_2)^{1/2}}\cdot \frac{\E_x(\sigma(\Om) )}{|\Om|^{1/2}}.
\end{align}

Note that $c_1, c_8$ and $c_9$ do not depend on $b_2$ and $\delta$. 

We choose $b_2>0$ so small that for some $\delta_*>0$ and all $\delta\in(0,\delta_*)$,
\begin{align*}
    \frac{1+ c_8  c_9 (\delta/b_2)}{(1+ c_1\delta b_2)^{1/2}} >1.
\end{align*}

The above choice agrees with the order of quantifiers in Lemma \ref{lem:dilat}.
We fix some $b_2,\delta>0$ as above so that for some $\eps>0$,
\begin{align*}
    \frac{1+ c_8  c_9 (\delta/b_2)}{(1+ c_1\delta b_2)^{1/2}} = 1+\eps,
\end{align*}
and, in view of \eqref{eq:compa}, for $x\in N_1$,
\begin{align}\label{eq:comp}
\frac{u_{\Om_1}(x)}{|\Om_1|^{1/2}}
     =\frac{\E_x(\sigma(\Om_1) )}{|\Om_1|^{1/2}}  & \geq 
     (1+\eps) \frac{\E_x(\sigma(\Om) )}{|\Om|^{1/2}}
     =(1+\eps)\frac{u_{\Om}(x)}{|\Om|^{1/2}} .
\end{align}

Recall that we assumed that for every $n>0$ there exists
a neighborhood $U_n\subset \prt \Om$ of $0$ such that the harmonic measure of 
 $\{y\in U_n: |\nabla u_\Om(y)|_{\mathrm{mf}} /|\Om|^{1/2}> J(\Om)-1/n\}$ is strictly positive.
In view of Lemma \ref{lem:supbound} (a) \textit{(iv)} and \eqref{eq:comp}, 
for every $n>0$ there exists
a neighborhood $U_n'\subset \prt \Om_1$ of $0$ such that the harmonic measure of 
 $\{y\in U_n': |\nabla u_{\Om_1}(y)|_{\mathrm{mf}} /|\Om_1|^{1/2}> (1+\eps)(J(\Om)-1/n)\}$ is strictly positive.
This and Lemma \ref{lem:supbound} (a)  imply 
that $J(\Om_1) \geq (1+\eps) J(\Om)$. 
We proved \eqref{eq:contr}.
\end{proof}

For the notation used in the following lemma and its proof, see the proof of Theorem \ref{th:line}.
We will now point out where the definitions of the quantities used in the statement of the lemma can be found. For the definitions of $R_2,M_1, M_2$ and $\prt_-M_1$, see \eqref{lm1}-\eqref{lm4}. Note that $b_2$ is a parameter in the definitions of $M_1$ and  $M_2$.
See Fig. \ref{fig3} for the illustration of $L'$ and the accompanying text for the formal definition. The domain $\Om_1$ is defined above \eqref{area2}. The definition of $\Om_1$ depends on a parameter $\delta$ (among other things).

\begin{lemma}\label{lem:dilat}

There exist $\eps_*,c_1, \eps_1>0$ such that for all $b_2\in(0,\eps_*)$, $\delta\in (0,\eps_1 b_2)$ and $x\in \prt_-M_1$,
    \begin{align*}
        \P_x(\tau(L') < \tau(M_2) <\tau(R_2) < \sigma(\Om_1)) \geq c_1 \delta/b_2.
    \end{align*}
\end{lemma}

\begin{proof}

{\it Step 1: Definitions and notation}.

Recall from \eqref{shiftT} the shift $\calT_\delta((x_1,x_2))  = (x_1,x_2-\delta)$ and $ x_\delta = \calT_\delta(x)$.
    For $0<b_1< b_2/2$ and $\delta>0$, let 
    \begin{align*}
    D&=\{ (x_1,x_2)\in \Om: x_1\in(b_1,b_2), x_2 < b_2   \},\\
        M_1^\delta &=  \calT_\delta( M_1),\qquad
        M_2^\delta =  \calT_\delta( M_2),\\
        D^*_\delta &=  \calT_\delta( D),\qquad
        \ \ D_\delta = D \cup D^*_\delta.
    \end{align*}
    Let $K^\delta$ be the middle third of the upper edge of $D^*_\delta$, i.e.,
\begin{align*}
    K^\delta = \{ (x_1, b_2-\delta) : b_1+(b_2-b_1)/3<x<b_2-(b_2-b_1)/3\}.
\end{align*}

Let $L^\delta =\{(y_1,b_2/4-\delta/2):y_1\in \R\}$ and $\calA^\delta$ denote the symmetry with respect to $L^\delta$.
Let $L'_- = \calT_\delta(L')$
and $L'_+ = \calA^\delta(L'_-)$.
The discs $M_1$ and $M_1^\delta$ are symmetric with respect to $L^\delta$ and the same is true for $M_2$ and $M_2^\delta$.
We define lower parts of the boundaries of the discs $M_1$ and $M_1^\delta$ as follows (cf. \eqref{lm3}),
\begin{align*}
    \prt_-M_1 &= \{(y_1,y_2)\in \prt M_1: y_2 < b_2/4\},\\
    \prt_-M_1^\delta &= \{(y_1,y_2)\in \prt M_1^\delta: y_2 < b_2/4-\delta\}.
\end{align*}
See Fig. \ref{fig18}.
    
Since $\prt \Om$ has no corners, in particular at the origin, we can choose $\eps_*>0$ sufficiently small so that for $b_2\in(0,\eps_*)$,
\begin{align}\label{eq:include}
    (b_1,b_2)\times (b_2/16, b_2) \subset D \subset (b_1,b_2) \times (0,b_2).
\end{align}

\begin{figure}[h!]
\includegraphics[width=0.7\linewidth]{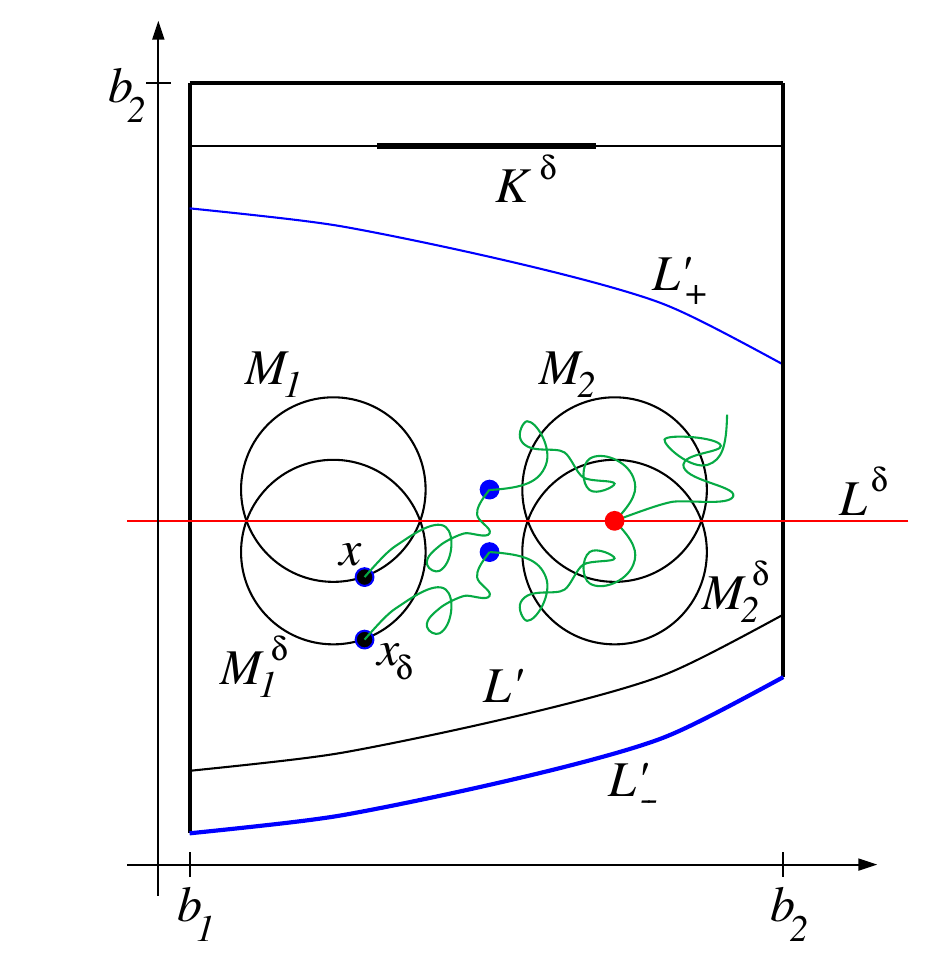}
\caption{The Brownian motions $W_t$ and $V_t$ are represented by green curves. 
Blue dots represent $W_{T_*}$ and $V_{T_*}$. The red dot 
represents $W_{T_1}=V_{T_1}$.
$V_t$ is a shift of $W_t$ until time $T_*$. $V_t$ is a mirror image of $W_t$ with respect to the line $L^\delta$
for $T_*\leq t \leq T_1$. The two processes are identical after time $T_1$.
The boundary of $D_\delta$ is represented by thick lines. }
\label{fig18}
\end{figure}

{\it Step 2}. We will prove that for some $c_2>0$ and all $x\in M_1$,
\begin{align}\label{enlarge}
    \P_{x_\delta}(\tau(M_2^\delta) < \sigma(D_\delta)) 
    \geq  \P_{x}(\tau(M_2) < \sigma(D))  +c_2 \delta /b_2.
\end{align}

By shift invariance, for $x\in M_1$,
\begin{align}\label{eq:shiftinv}
    \P_{x_\delta}(\tau(M_2^\delta) < \sigma(D^*_\delta))
    = \P_{x}(\tau(M_2) < \sigma(D)).
\end{align}

It follows from \eqref{eq:include} 
and scaling 
that there exists $c_3>0$ such that for all  $x\in M_1^\delta$,
\begin{align}\label{j17.1}
    \P_x(\tau(K^\delta) < \tau(M_2^\delta \cup (\prt D^*_\delta \setminus K^\delta) )) \geq c_3.
\end{align}

It follows from \eqref{eq:include}, scaling and the boundary Harnack principle
that there exist $c_4,\eta>0$ such that for all $\delta\in(0,\eta)$ and $y \in K^\delta$,
$$\P_y(\tau(M_2^\delta) < \sigma(D_\delta)) \geq c_4 \delta /b_2.$$
This estimate, \eqref{j17.1} and the strong Markov property applied at $\tau(K^\delta)$ imply that
there exists $c_2>0$ such that for all  $x\in M_1^\delta$,
\begin{align*}
    \P_x(\tau(K^\delta) < \tau(M_2^\delta) < \sigma(D_\delta)) \geq c_2 \delta /b_2.
\end{align*}
This and \eqref{eq:shiftinv} show that for $x\in M_1$,
   \begin{align*}
   \P_{x_\delta}(\tau(M_2^\delta) < \sigma(D_\delta)) 
   &\geq \P_{x_\delta}(\tau(M_2^\delta) < \sigma(D^*_\delta)) 
       +  \P_{x_\delta}(\tau(K^\delta) < \tau(M_2^\delta) < \sigma(D_\delta))  \\
    &\geq \P_{x_\delta}(\tau(M_2^\delta) < \sigma(D^*_\delta))  +c_2 \delta /b_2\\
    & =\P_{x}(\tau(M_2) < \sigma(D))  +c_2 \delta /b_2.
    \end{align*}
We have proved \eqref{enlarge}.

{\it Step 3}. We will prove that for all $x\in \prt_- M_1$,
\begin{align}\label{eq:weakinold}
    \P_{x}(\tau(M_2) < \sigma(D_\delta)) 
    \geq \P_{x_\delta}(\tau(M_2^\delta) < \sigma(D_\delta)) .
\end{align}
    
Let $W_t= (W^1_t,W^2_t)$ be a Brownian motion starting from $W_0=x\in \prt_-M_1$
and 
\begin{align*}
T_*&=  \inf\{t\geq 0: W^2_t = b_2/4   \text{  or  } W_t\notin D_\delta\},\\
T_1&=\inf\{t\geq T_*: W_t \in L^\delta \text{  or  } W_t\notin D_\delta\}  ,\\
    V_t &= 
    \begin{cases}
    \calT_\delta(W_t), & \text{  for  } 0\leq t \leq T_*,\\
    \calA^\delta(W_t), & \text{  for  } T_*< t\leq T_1,\\
    W_t , & \text{  for  }t> T_1.
    \end{cases}
\end{align*}

The process $V_t$ is a Brownian motion. The pair $(W,V)$ is a ``coupling''.
Before $T_*$, the two processes form a ``synchronous'' coupling. The process $V$
is  a shift of $W$ on the interval $[0,T_*]$ so $V_0 = x_\delta$. Between $T_*$ and $T_1$, the processes form
a ``mirror'' coupling with respect to the line $L^\delta$. In other words,  $V_t$ and $W_t$
are symmetric with respect to $L^\delta$ so $W_{T_1}=V_{T_1}$. Finally, after $T_1$, the processes are ``coupled,''
i.e., they are identical. See Fig. \ref{fig18}.

We will use subscripts as in $\tau_W$ to indicate which of the processes we are referring to.
In an alternative notation, Inequality \eqref{eq:weakinold}, that we want to prove in this step, can be stated as
\begin{align}\label{eq:weakin}
    \P(\tau_W(M_2) < \sigma(D_\delta)) 
    \geq \P(\tau_V(M_2^\delta) < \sigma(D_\delta)).
\end{align}

In the argument given below, the processes $W$ and $V$ will be killed upon exiting $D_\delta$.
By abuse of notation, we will call killed processes $W$ and $V$.

Before $T_*$, 
if $V$ hits $M_2^\delta$ before exiting $D_\delta$
then $W$ hits $M_2$ (at the same time) before exiting $ D_\delta$. Hence,
\begin{align}\label{beforeT*}
    \P(\tau_W(M_2) <  T_*)
    \geq \P(\tau_V(M_2^\delta) <  T_*). 
\end{align}

Note that $L'_+\subset D$ (except for the endpoints), so if
$V$ hits $L'_-$ at a time $t\in[T_*,T_1]$, $W$ hits $L'_+$ a the same time and, therefore, it is in the interior of $D$. Thus, in view of the symmetry
of $W_t$ and $V_t$ for $t\in[T_*,T_1]$, 
\begin{align}\label{indini}
    \P (T_*\leq \tau_W(M_2) < T_1 )
    \geq \P(T_*\leq \tau_V(M_2^\delta) < T_1).
\end{align}

For $k\geq 1$, let
\begin{align*}
    S_k&= \inf\{t\geq T_k: W_t \in L'_-\cup L'_+ \cup \prt D_\delta\}= \inf\{t\geq T_k: V_t \in L'_-\cup L'_+ \cup \prt D_\delta \},\\
    T_{k+1}&= \inf\{t\geq S_{k}: W_t \in L^\delta \cup \prt D_\delta\}.
\end{align*}

Recall that $W_t=V_t$ for $t\geq T_1$.
If $T_1 < \tau_V(M_2^\delta) \land \sigma_V(D_\delta)$, then,
by symmetry, 
\begin{align}\label{indT}
    \P (T_1\leq \tau_W(M_2) < S_1 )
    = \P (T_1\leq \tau_V(M_2^\delta) < S_1 ).
\end{align}

If $W(S_1)\in L'_-$, then, the processes $W$ and $V$ are killed at 
the time $S_1$ and, therefore, they will not hit $M_2\cup M_2^\delta$
after $S_1$.

If $W(S_1)\in L'_+$, then, the following events are possible for $\{W_t, t\in(S_1, T_2)\}$:
\begin{itemize}
    \item The trajectory does not hit $M_2\cup M_2^\delta$.
    \item The trajectory hits both $M_2$ and $ M_2^\delta$.
    \item The trajectory hits $M_2$ but does not hit $M_2^\delta$.
\end{itemize}

It is impossible for $\{W_t, t\in(S_1, T_2)\}$ to hit $M_2^\delta$ but not $M_2$
because $M_2^\delta \setminus M_2$ lies totally below $L^\delta$. Hence, 
\begin{align}\label{indS}
    \P(S_1\leq \tau_W(M_2) < T_2 )
    \geq \P(S_1\leq \tau_V(M_2^\delta) < T_2 ).
\end{align}

At time $T_2$, the process is at $L^\delta$, unless it has been killed.
Hence, we can apply induction and obtain the following analogs of \eqref{indT}
 and \eqref{indS} for $k\geq 2$,
 \begin{align*}
    \P (T_k\leq \tau_W(M_2) < S_k )
    &= \P (T_k\leq \tau_V(M_2^\delta) < S_k ),\\
     \P(S_k\leq \tau_W(M_2) < T_{k+1} )
    &\geq \P(S_k\leq \tau_V(M_2^\delta) < T_{k+1} ).
\end{align*}

This, \eqref{beforeT*}, \eqref{indini}, \eqref{indT} and \eqref{indS} imply that for $W_0=x\in \prt_- M_1$,
\begin{align*}
    \P&(\tau_W(M_2) < \sigma(D_\delta)) \\
    &= \P\Bigg (
    \{\tau_W(M_2) < T_*\} \cup
    \{T_* \leq \tau_W(M_2) < T_1\} \\
    &\qquad \cup \bigcup_{k\geq 1}
    (\{ T_k< \tau_W(M_2) < S_k \}
    \cup \{S_k < \tau_W(M_2) < T_{k+1} \} )
    \Bigg)\\
     &\geq \P\Bigg (
    \{\tau_V(M_2) < T_*\} \cup
    \{T_* \leq \tau_V(M_2) < T_1\} \\
    &\qquad \cup \bigcup_{k\geq 1}
    (\{ T_k< \tau_V(M_2) < S_k \}
    \cup \{S_k < \tau_V(M_2) < T_{k+1} \} )
    \Bigg)\\
    &= \P(\tau_V(M_2^\delta) < \sigma(D_\delta)) .
\end{align*}

This proves \eqref{eq:weakin} and, therefore, \eqref{eq:weakinold}. 

{\it Step 4}.
Combined with \eqref{enlarge}, \eqref{eq:weakinold} yields
for $x\in \prt_- M_1$,
  \begin{align*}
  \P_{x}(\tau(M_2) < \sigma(D_\delta)) \geq
   \P_{x_\delta}(\tau(M_2^\delta) < \sigma(D_\delta)) 
   \geq \P_{x}(\tau(M_2) < \sigma(D)) +c_2 \delta /b_2.
    \end{align*}
    
Therefore, we deduce that for $x\in\prt_- M_1$,
 $$
 \P_x(\tau(L') < \tau(M_2) < \sigma(\Om_1)) \geq  \P_{x}(\tau(M_2) < \sigma(D_\delta)) - \P_{x}(\tau(M_2) < \sigma(D)) \geq c_2 \delta/b_2.  
 $$

It is easy to see that for some $c_6>0$ and all $y\in M_2$, $\P_y( \tau(R_2) < \sigma(\Om_1)) > c_6$. Hence, for $c_7=c_2 c_6$ and all $x\in\prt_- M_1$,
   \begin{align*}
  \P_x&(\tau(L') < \tau(M_2) < \tau(R_2) < \sigma(\Om_1))\\
   &= \int_{y\in \partial M_2} \P_y(\tau(R_2) < \sigma(\Om_1))\cdot \P_x(W_{\tau(M_2)}\in dy, \tau(L') < \tau(M_2)) \geq c_7 \delta/b_2.
    \end{align*}
    
The proof is complete.

\end{proof}

\section{The perimeter constraint}\label{sec:perim}

\begin{proof}[Proof of Theorem \ref{th:perimeter}]

The proofs of all parts of the theorem are essentially the same as those of Theorem \ref{main:thm}. We will present only the modifications needed to replace normalization by area with normalization by perimeter.

    \textit{(i)}  Suppose that  $\Om_n$ is a sequence of convex sets with analytic boundaries such that the gradient $\nabla u_{\Om_n}$ is an analytic function on the closure of $\Omega_n$,
$P(\Om_n)=1 $, $J_P(\Om_n) > J_P^{\max} - 1/n$,  $\Omega_n \subset \{(x,y)\,:\, y\ge 0\}$ and $|\nabla u_{\Om_n}|$  attains its maximum at $0\in \partial \Omega_n$, for every $n$.
To apply the Blaschke selection theorem (see Remark \ref{oldres} \textit{(iii)}), we need to show that $|\Om_n|$ does not go to 0. If this were the case, then by \eqref{ineq:04_12}, $J_P(\Om_n)$ would tend to $0$, contradicting the fact that $\Omega_n$ is a maximizing sequence for $J_P$. 
    
    \textit{(ii)} 
One can prove that a maximizer domain does not have a corner where the gradient attains the maximum, just like in Step 2 of the proof of   Theorem \ref{main:thm} \textit{(i)}.
    
Similarly to the proof of   Theorem \ref{main:thm} \textit{(ii)}, assume that an optimizer $\Omega$, with $P(\Om)=1$, contains a corner of angle $\alpha<\pi$ on its boundary (not the origin, where the gradient of the torsion function is assumed to attain its maximum). Recall the definition of $\wt{\Omega}_\eps$ from that proof. It is a set obtained by removing a curvilinear cap with area $\eps$ (labeled $ABB'$ in Fig. \ref{fig:perimeter}). 

We will estimate the perimeter of $\wt \Omega_\eps$. Choose $\theta \in (0,\alpha/2)$ (see Fig. \ref{fig:perimeter}) sufficiently close to $\alpha/2$ so that
    $$\frac{1}{\cos\theta}-\tan\frac{\alpha}{2} \ge \frac{1-\sin\frac{\alpha}{2}}{2\cos\frac{\alpha}{2}}>0.$$
    
    The area of the triangle $ACC'$, equal to $\tan(\alpha/2)\rho_\eps^2$, is larger than the area of the cap $ABB'$, equal to $\eps$. Thus, 
    $$\rho_\eps \ge \frac{1}{\sqrt{\tan(\alpha/2)}}\cdot \sqrt{\eps}.$$

    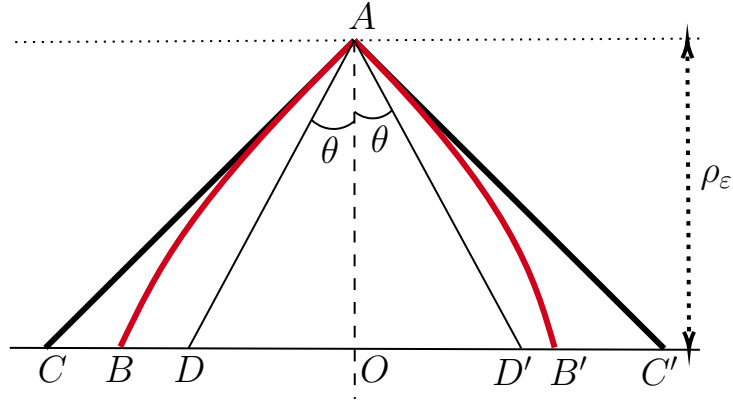
\begin{figure}[h!]
    \centering
\tikzset{every picture/.style={line width=0.75pt}} 

\begin{tikzpicture}[x=0.75pt,y=0.75pt,yscale=-.6,xscale=.6]

\draw    (30,339) -- (610.08,340.62) ;
\draw  [dash pattern={on 4.5pt off 4.5pt}]  (319.61,80.41) -- (320,382) ;
\draw [line width=2.25]    (319.61,80.41) -- (60,339) ;
\draw [line width=2.25]    (319.61,80.41) -- (580.08,339.62) ;
\draw [color={rgb, 255:red, 208; green, 2; blue, 27 }  ,draw opacity=1 ][line width=2.25]    (123,338) .. controls (155,272) and (177,222) .. (319.61,80.41) ;
\draw [color={rgb, 255:red, 208; green, 2; blue, 27 }  ,draw opacity=1 ][line width=2.25]    (319.61,80.41) .. controls (452.65,213.94) and (471,279.72) .. (488,338.72) ;
\draw  [draw opacity=0] (319.89,148.06) .. controls (315.95,151.63) and (310.94,154.04) .. (305.3,154.7) .. controls (297.16,155.65) and (289.43,152.74) .. (283.88,147.4) -- (302.14,127.61) -- cycle ; \draw   (319.89,148.06) .. controls (315.95,151.63) and (310.94,154.04) .. (305.3,154.7) .. controls (297.16,155.65) and (289.43,152.74) .. (283.88,147.4) ;  
\draw  [draw opacity=0] (351.22,139.28) .. controls (348.15,142.34) and (344.18,144.57) .. (339.61,145.5) .. controls (332.35,146.99) and (325.19,144.88) .. (319.98,140.4) -- (335,122.99) -- cycle ; \draw   (351.22,139.28) .. controls (348.15,142.34) and (344.18,144.57) .. (339.61,145.5) .. controls (332.35,146.99) and (325.19,144.88) .. (319.98,140.4) ;  
\draw    (319.61,80.41) -- (180.12,339.53) ;
\draw    (319.61,80.41) -- (459.94,339.37) ;
\draw [line width=1.5]  [dash pattern={on 1.69pt off 2.76pt}]  (599.59,84.08) -- (600.98,336.72) ;
\draw [shift={(601,339.72)}, rotate = 269.68] [color={rgb, 255:red, 0; green, 0; blue, 0 }  ][line width=1.5]    (14.21,-4.28) .. controls (9.04,-1.82) and (4.3,-0.39) .. (0,0) .. controls (4.3,0.39) and (9.04,1.82) .. (14.21,4.28)   ;
\draw [shift={(599.57,81.08)}, rotate = 89.68] [color={rgb, 255:red, 0; green, 0; blue, 0 }  ][line width=1.5]    (14.21,-4.28) .. controls (9.04,-1.82) and (4.3,-0.39) .. (0,0) .. controls (4.3,0.39) and (9.04,1.82) .. (14.21,4.28)   ;
\draw  [dash pattern={on 0.84pt off 2.51pt}]  (609.65,78.82) -- (30.65,81.82) ;

\draw (289,158.4) node [anchor=north west][inner sep=0.75pt]  [font=\large]  {$\theta $};
\draw (331,149.4) node [anchor=north west][inner sep=0.75pt]  [font=\large]  {$\theta $};
\draw (322.04,343.21) node [anchor=north west][inner sep=0.75pt]  [font=\large]  {$O$};
\draw (312,46) node [anchor=north west][inner sep=0.75pt]  [font=\large]  {$A$};
\draw (107,343.4) node [anchor=north west][inner sep=0.75pt]  [font=\large]  {$B$};
\draw (481,343.4) node [anchor=north west][inner sep=0.75pt]  [font=\large]  {$B'$};
\draw (51,343.4) node [anchor=north west][inner sep=0.75pt]  [font=\large]  {$C$};
\draw (558,343.4) node [anchor=north west][inner sep=0.75pt]  [font=\large]  {$C'$};
\draw (165,343.4) node [anchor=north west][inner sep=0.75pt]  [font=\large]  {$D$};
\draw (433,343.4) node [anchor=north west][inner sep=0.75pt]  [font=\large]  {$D'$};
\draw (608,183.4) node [anchor=north west][inner sep=0.75pt]  [font=\large]  {$\rho _{\varepsilon }$};

\end{tikzpicture} 
    \caption{The red curves are parts of $\prt \Om$. The curvilinear cap $ABB'$ has area $\eps$.}
    \label{fig:perimeter}
\end{figure}

    Therefore, for sufficiently small $\eps>0$,
    \begin{align}\notag
    1-P(\wt \Omega_\eps) &= P(\Omega)-P(\wt \Omega_\eps) = \wt{AB}+\wt{AB'}-OB-OB' \ge 2(AD - OC)\\
    &= 2\left(\frac{1}{\cos\theta}-\tan\frac{\alpha}{2}\right)\rho_\eps
    \ge \frac{1-\sin\frac{\alpha}{2}}{\cos\frac{\alpha}{2}}\rho_\eps
    \ge \frac{1-\sin\frac{\alpha}{2}}{\sqrt{\cos\frac{\alpha}{2}\sin\frac{\alpha}{2}}}\cdot\sqrt{\eps}=:K_\alpha\sqrt{\eps}.\label{nv25.1}
\end{align}

The following is a modified ending of the proof of Theorem \ref{main:thm} \textit{(ii)}.
Lemma \ref{lem:supbound} shows that we can find $z_n\in \prt_r\wt\Om_\eps$ such that $\lim_{n\to\infty} z_n = 0$, and 
\begin{align*}
    \lim_{\delta\to 0} \frac 1 \delta \E_{z_n + \delta \normal(z_n)} (\sigma(\wt\Om_\eps))&=|\nabla u_{\wt\Om_\eps}(z_n)|_{\mathrm{mf}},\\
    \lim_{\delta\to 0} \frac 1 \delta \E_{z_n + \delta \normal(z_n)} (\sigma(\Om))&\geq J_P(\Om)-1/n.
\end{align*}
This and \eqref{eq:afin} imply that
\begin{align*}
    |\nabla u_{\wt\Om_\eps}(z_n)|_{\mathrm{mf}}
    &=\lim_{\delta\to 0} \frac 1 \delta \E_{z_n + \delta \normal(z_n)} (\sigma(\wt\Om_\eps))
    \geq (1-  (c_{9}/c_{11})\eps^\beta) (J_P(\Om)-1/n).
\end{align*}
Since $\eps>0$ is fixed,  for some $n$,
\begin{align*}
     |\nabla u_{\wt\Om_\eps}(z_n)|_{\mathrm{mf}}
    \geq (1-  2(c_{9}/c_{11})\eps^\beta) J_P(\Om),
\end{align*}
and, in view of \eqref{nv25.1},
\begin{align*}
    \|\nabla u_{\wt \Om_\eps}\|_\infty &\geq (1-  2(c_{9}/c_{11})\eps^\beta) J_P(\Om),\\
    J_P(\wt \Om_\eps)&=\frac{\|\nabla u_{\wt \Om_\eps}\|_\infty}{P(\wt \Om_\eps)} \geq \frac{(1-  2(c_{9}/c_{11})\eps^\beta) }{P(\wt \Om_\eps)}J_P(\Om)
    \geq \frac{(1-  2(c_{9}/c_{11})\eps^\beta) }{1-K_\alpha\sqrt{\eps}}J_P(\Om).
\end{align*}
Recall from \eqref{eq:ommintau} that $\beta>1$. Thus, for small $\eps>0$, $J_P(\wt \Om_\eps) > J_P(\Om)$, contradicting the assumption that  $\Om$ is an optimizer
in the sense of \eqref{defPOptimizer}.

\textit{(iii)} 
Recall a maximizer $\Om$ and $\Om_1$ from the proof of Theorem \ref{main:thm} \textit{(iii)}. We assume in the present proof that $P(\Om)=1$. We have proved in \eqref{eq:nv252} that for  $x\in N_1$,
\begin{align}\label{new820}
    \E_x(\sigma(\Om_1) )  & \geq   (1+ c_8  c_9 (\delta/b_2)) \E_x(\sigma(\Om) ) .
\end{align}

For the perimeters, we have (see Fig. \ref{fig3}),
\begin{align}\label{nv26.2}
    P(\Omega_1)- P(\Omega) &= z_6z_3+\wt{z_3z_4}+z_4z_5-\wt{z_6z_1}-\wt{z_1z_2}-\wt{z_2z_5}\\
    &= z_6z_3+z_4z_5-\wt{z_6z_1}-\wt{z_2z_5}\notag \\
    &\leq (\wt{z_6z_1}+\delta) + (\wt{z_2z_5}+\delta) - \wt{z_6z_1}-\wt{z_2z_5}
    = 2\delta.\notag
\end{align}
Therefore,
\begin{align}\label{nv26.1}
    P(\Omega_1)\leq P(\Omega) + 2\delta= 1+ 2\delta.
\end{align}
We choose $b_2>0$ sufficiently small so that for some $\delta,\eps>0$,
$$ \frac{1+c_8c_9(\delta/b_2)}{1+2\delta}>1+\eps.$$
This and \eqref{new820}-\eqref{nv26.1} yield,
\begin{align}\label{eq:compnv}
\frac{u_{\Om_1}(x)}{P(\Omega_1)}
     =\frac{\E_x(\sigma(\Om_1) )}{P(\Omega_1)}  
     &\geq     \frac{ (1+ c_8  c_9 (\delta/b_2))\E_x(\sigma(\Om) )}{( 1+ 2\delta)P(\Omega)}\\
     &\geq     (1+\eps) \frac{\E_x(\sigma(\Om) )}{P(\Omega)}
     =(1+\eps)\frac{u_{\Om}(x)}{P(\Omega)} .\notag
\end{align}

We now repeat the end of the proof of Theorem \ref{main:thm} \textit{(iii)}, with minor modifications.
We assumed that for every $n>0$ there exists
a neighborhood $U_n\subset \prt \Om$ of $0$ such that the harmonic measure of 
 $\{y\in U_n: |\nabla u_\Om(y)|_{\mathrm{mf}} /P(\Omega)> J(\Om)-1/n\}$ is strictly positive.
In view of Lemma \ref{lem:supbound} (a) \textit{(iv)} and \eqref{eq:compnv}, 
for every $n>0$ there exists
a neighborhood $U_n'\subset \prt \Om_1$ of $0$ such that the harmonic measure of 
 $\{y\in U_n': |\nabla u_{\Om_1}(y)|_{\mathrm{mf}} /P(\Omega_1)> (1+\eps)(J(\Om)-1/n)\}$ is strictly positive.
This and Lemma \ref{lem:supbound} (a)  imply 
that $J_P(\Om_1) \geq (1+\eps) J_P(\Om)$. 
This contradicts the assumption that $\Om$ is an optimizer.
\end{proof}

\section*{Acknowledgments}

We are grateful to Hiroaki Aikawa, Dorin Bucur, Stefan Steinerberger and Ruofei Yao for very helpful advice. We are also thankful to Chiu-Yen Kao for her help with the numerical approximation of the optimal shapes presented in Table \ref{tab:max_area_peri}. 

This project started at a workshop at the American Institute of Mathematics. The authors are grateful to AIM for the excellent research environment.

\bibliographystyle{plain}
\bibliography{torsion}

\end{document}